\documentclass[reqno]{amsart}
\usepackage{setspace,tikz,mathrsfs,listings,multicol}
\usepackage{amssymb}
\usepackage[vcentermath]{youngtab}
\usepackage[all,cmtip]{xy}
\usetikzlibrary{arrows,matrix}
\usepackage{comment}
\usepackage[sc]{mathpazo}
\usepackage[T1]{fontenc}
\usepackage{needspace}
\usepackage{tabls}
\usepackage{ytableau}
\usepackage{mathtools}  % for the underbracket
\usepackage[colorlinks=true, pdfstartview=FitV, linkcolor=blue, citecolor=blue, urlcolor=blue]{hyperref}

\newcommand{\mbf}{\mathbf}

\newcommand{\swt}[1]{\left\langle #1 \right\rangle} % Spectral weight or amplitude
\newcommand{\merge}[1]{\vee_{#1}} % merge
\newcommand{\SymGp}[1]{\mathfrak{S}_{#1}} % symmetric group
\newcommand{\std}[1]{\widetilde{#1}} % standardization
\newcommand{\diag}[1]{\left[#1\right]} % diagram

\DeclareMathOperator{\inter}{int} % queuing interval
\DeclareMathOperator{\wt}{wt} % weight
\DeclareMathOperator{\SP}{SP} % parenthesis sequence
\DeclareMathOperator{\word}{word} % word
\DeclareMathOperator{\xcoord}{x} % the x-coordinate
\DeclareMathOperator{\ycoord}{y} % the y-coordinate
\DeclareMathOperator{\id}{id} % identity
\DeclareMathOperator{\SST}{SST} % Semistandard tableaux
\DeclareMathOperator{\KM}{KM}
\newcommand{\ff}{\mathbf{f}}
\newcommand{\mm}{\mathbf{m}}
\newcommand{\nn}{\mathbf{n}}
\newcommand{\pp}{\mathbf{p}}
\newcommand{\qq}{\mathbf{q}}
\newcommand{\uu}{\mathbf{u}}
\newcommand{\vv}{\mathbf{v}}
\newcommand{\xx}{\mathbf{x}}

\newcommand{\mcA}{\mathcal{A}}

\newcommand{\mcW}{\mathcal{W}}
\newcommand{\mcI}{\mathcal{I}}

\newcommand{\NN}{\mathbb{N}}
\newcommand{\ZZ}{\mathbb{Z}}

\newcommand{\RR}{\mathbb{R}}

\newcommand{\fraks}{\mathfrak{s}}

\let\sumnonlimits\sum
\let\prodnonlimits\prod
\let\cupnonlimits\bigcup
\let\capnonlimits\bigcap
\renewcommand{\sum}{\sumnonlimits\limits}
\renewcommand{\prod}{\prodnonlimits\limits}
\renewcommand{\bigcup}{\cupnonlimits\limits}
\renewcommand{\bigcap}{\capnonlimits\limits}

\newenvironment{subproof}{\textit{Proof.} }{\hfill$\blacksquare$ \medskip}
\newenvironment{verlonglong}{}{}
\newenvironment{verlong}{}{}
\newenvironment{vershort}{}{}

\excludecomment{verlonglong}
\excludecomment{verlong}
\includecomment{vershort}
\excludecomment{noncompile}

\newcommand{\powset}[2][]{\ifthenelse{\equal{#2}{}}{\mathcal{P}\left(#1\right)}{\mathcal{P}_{#1}\left(#2\right)}}
\newcommand{\set}[1]{\left\{ #1 \right\}}
\newcommand{\abs}[1]{\left| #1 \right|}
\newcommand{\tup}[1]{\left( #1 \right)}
\newcommand{\ive}[1]{\left[ #1 \right]}
\definecolor{darkred}{rgb}{0.7,0,0} % darkred color
\newcommand{\defn}[1]{{\color{darkred}\emph{#1}}} % emphasis of a definition

\definecolor{dblackcolor}{rgb}{0.0,0.0,0.0}
\definecolor{dbluecolor}{rgb}{0.01,0.02,0.7}
\definecolor{dgreencolor}{rgb}{0.2,0.4,0.0}
\definecolor{dgraycolor}{rgb}{0.30,0.3,0.30}

\definecolor{UQgold}{RGB}{196, 158, 54} % UQ gold
\definecolor{UQpurple}{RGB}{73, 7, 94} % UQ purple
\definecolor{UMNgold}{RGB}{255,200,46} % UMN gold
\definecolor{UMNmaroon}{RGB}{106,0,50} % UMN maroon

\usepackage{xparse}

\makeatletter
\protected\def\specialmergetwolists{%
  \begingroup
  \@ifstar{\def\cnta{1}\@specialmergetwolists}
    {\def\cnta{0}\@specialmergetwolists}%
}
\def\@specialmergetwolists#1#2#3#4{%
  \def\tempa##1##2{%
    \edef##2{%
      \ifnum\cnta=\@ne\else\expandafter\@firstoftwo\fi
      \unexpanded\expandafter{##1}%
    }%
  }%
  \tempa{#2}\tempb\tempa{#3}\tempa
  \def\cnta{0}\def#4{}%
  \foreach \x in \tempb{%
    \xdef\cnta{\the\numexpr\cnta+1}%
    \gdef\cntb{0}%
    \foreach \y in \tempa{%
      \xdef\cntb{\the\numexpr\cntb+1}%
      \ifnum\cntb=\cnta\relax
        \xdef#4{#4\ifx#4\empty\else,\fi\x#1\y}%
        \breakforeach
      \fi
    }%
  }%
  \endgroup
}
\makeatother

\theoremstyle{plain}
\newtheorem{thm}{Theorem}[section]
\newtheorem{lemma}[thm]{Lemma}

\newtheorem{prop}[thm]{Proposition}
\newtheorem{cor}[thm]{Corollary}
\newtheorem{claim}[thm]{Claim}
\theoremstyle{definition}
\newtheorem{dfn}[thm]{Definition}
\newtheorem{example}[thm]{Example}
\newtheorem{remark}[thm]{Remark}
\numberwithin{equation}{section}
\usepackage[colorinlistoftodos]{todonotes}

\begin{document}
\title[MLQs with spectral parameters]{Multiline queues with spectral parameters}

\author[E.~Aas]{Erik Aas}
\address[E. Aas]{Department of Mathematics, Pennsylvania State University, McAllister Building, State College, PA 116802, USA}
\email{eaas@kth.se}

\author[D.~Grinberg]{Darij Grinberg}
\address[D. Grinberg]{School of Mathematics, University of Minnesota, 206 Church St. SE, Minneapolis, MN 55455}
\email{darijgrinberg@gmail.com}
\urladdr{http://www.cip.ifi.lmu.de/~grinberg/}

\author[T.~Scrimshaw]{Travis Scrimshaw}
\address[T. Scrimshaw]{School of Mathematics and Physics, The University of Queensland, St.\ Lucia, QLD 4072, Australia}
\email{tcscrims@gmail.com}
\urladdr{https://people.smp.uq.edu.au/TravisScrimshaw/}

\date{\today}

  % Symmetric functions

\thanks{TS was partially supported by the Australian Research Council DP170102648 and the National Science Foundation RTG grant DMS-1148634.}

\begin{abstract}
Using the description of multiline queues as functions on words, we introduce the notion of a spectral weight of a word by defining a new weighting on multiline queues.
We show that the spectral weight of a word is invariant under a natural action of the symmetric group, giving a proof of the commutativity conjecture of Arita, Ayyer, Mallick, and Prolhac.
We give a determinant formula for the spectral weight of a word, which gives a proof of a conjecture of the first author and Linusson.
\end{abstract}

\maketitle

\tableofcontents

%=====================================================================
\section{Introduction}
\label{sec:introduction}

One of the fundamental models of particles moving in a 1-dimensional lattice is the asymmetric simple exclusion process (ASEP), and it has received broad attention in many different variations.
The earliest known publication of the ASEP was done to model the dynamics of ribosomes along RNA~\cite{MGP68}.
For statistical mechanics, it is a model for gas particles in a lattice with an induced current, where the exclusion mimics the short-range interactions among the particles.
Despite admitting very simple descriptions of the particle dynamics, the ASEP has very rich macroscopic behaviors, such as
\begin{itemize}
\item boundary-induced phase transitions~\cite{Krug91},
\item spontaneous symmetry breaking with possibly multiple broken symmetry phases~\cite{AHR98,AHR99,CEM01,EFGM95,EPSZ05,GLEMSS95,PK07},
\item describing the formations of shocks~\cite{DJLS93,Ferrari92,FF94,FF94II,Liggett76}, and
\item phase separation and condensation~\cite{EKKM98,JNHWW09,KLMST02,RSS00}.
\end{itemize}
We also refer the reader to~\cite{PEM09,Schutz01,SZ95,TJHJ16} and references therein.

The term exclusion process was coined by Spitzer~\cite{Spitzer70}, where he was focused on an application with Brownian motion with hard-core interactions.
Moreover, it was~\cite{Spitzer70} that initiated the investigation of exclusion processes using probability theory.
However, the applications of the ASEP (and its variations) has since spread to other areas, such as
\begin{itemize}
\item transportation processes in capillary vessels~\cite{Levitt73} or proteins within the cells along actin filaments~\cite{KNL05},
\item anistropic conductors known as solid electrolytes~\cite{CL99},
\item discrete models of traffic flow~\cite{Schad01},
\item partition growth processes~\cite{Lam15},
\item (non)symmetric Macdonald polynomials, Koornwinder polynomials, and (deformed) Knizhnik--Zamolodchikov (KZ) equations~\cite{CdGW15,KT07},
\item random matrix theory~\cite{Johansson00,TW09}, and
\item moments of Askey--Wilson polynomials~\cite{CW11,USW04}.
\end{itemize}

If we prohibit the particles from moving backwards, we obtain the totally asymmetric exclusion process (TASEP), a non-equilibrium stochastic process that has its own vast literature.
For example, we refer the reader to~\cite{AasLin17,AAMP,BE07,BP14,DEHP93,KMO15,KMO16,Liggett99} and references therein.
In this paper, we consider the TASEP on a ring with $n$ sites and $\ell$ species of particles.
Thus, we will consider the states to be words $u$ in the alphabet $\{1, \dotsc, \ell\}$ of length $n$, where we take the indices to be $\ZZ / n \ZZ$.
We will also consider our process to be discrete in time, where our transition map interchanges a pair $u_i u_{i+1}$ with $u_i > u_{i+1}$ to $u_{i+1} u_i$ and is done at a uniform rate.

The steady state of the TASEP on a ring is known in terms of another process using ordinary multiline queues (MLQs) and applying the Ferrari--Martin (FM) algorithm~\cite{FM06,FM07}.
This is a generalization of 2-line queues used by Angel~\cite{Angel06} and the work of Ferrari, Fontes, and Kohayakawa~\cite{FFK94}.
In~\cite{KMO15,KMO16}, the FM algorithm was reformulated in terms of the combinatorial $R$-matrix~\cite{NY97,Shimozono02} and using type $A_{n-1}^{(1)}$ Kirillov--Reshetikhin crystals~\cite{KKMMNN92}.
This interpretation gives a connection with five-vertex models, corner transfer matrices~\cite{Baxter89}, 3D integrable lattice models, and the tetrahedron equation~\cite{Zam80}, yielding a matrix product formula for the steady state distribution different than~\cite{CdGW15,EFM09,PEM09}.

In this paper, we introduce a new weighting of MLQs, which is the weight of the MLQ considered as a tensor product of Kirillov--Reshetikhin crystals.
We also interpret MLQs as functions on words of a fixed length $n$ following~\cite{AAMP}, where it was referred to as the generalized FM algorithm.
This allows us to define the spectral weight or amplitude of a word $u$ to be the sum over all the weight of all ordinary MLQs $\qq$ such that $u = \qq(1^n)$, where $1^n$ is the word $1 \dotsm 1$.
Furthermore, we introduce the notion of a $\sigma$-twisted MLQ, where $\sigma$ is a permutation, although this is implicitly considered in~\cite{AAMP}.
Our main result (Theorem~\ref{thm:permutation}) is that for a fixed permutation $\sigma$, the sum of the weights of all $\sigma$-twisted MLQs $\qq_{\sigma}$ such that $u = \qq_{\sigma}(1^n)$ equals the spectral weight of $u$.
To this end, we construct an action of the symmetric group on MLQs that corresponds, under the usual FM algorithm, to the natural action by letters on words.
We show that does not change the MLQ as a function on words.
This action has previously appeared in a number of different guises, such as in Danilov and Koshevoy~\cite{DanilovKoshevoy} (see also~\cite[Ch.~4]{Gorodentsev2}), van Leeuwen~\cite[Lemma~2.3]{vanLeeuwen-dc}, and Lothaire~\cite[Ch.~5, (5.6.3)]{Loth}.
In the context of Kirillov--Reshetikhin crystals, it can be described as applying a combinatorial $R$-matrix to an MLQ, where the weight remaining invariant is a condition of being a crystal isomorphism.

As a consequence of this action and specializing our weight parameters to $1$, we obtain a proof of the commutativity conjecture of~\cite{AAMP}.
However, we note that the interlacing property of~\cite{AAMP} does not generalize to our weighting of MLQs.
Furthermore, we give a determinant expression for the spectral weight of decreasing words by using the Lindstr\"om--Gessel--Viennot Lemma~\cite{GV85,Lindstrom73}.
By combining these results, we obtain a proof of~\cite[Conj.~3.10]{AasLin17}, which in turn proves a number of other conjectures in~\cite{AasLin17}.

We note that our weighting scheme can be extended to multiline process used to determine the steady state distribution of the totally asymmetric zero range process (TARZP) on a ring, where multiple particles can occupy the same site.
This comes from the fact that the TARZP steady state distribution can also be computed using a tensor product of Kirillov--Reshetikhin crystals (under rank-level duality) using combinatorial $R$-matrices with analogous connections to corner transfer matrices and the tetrahedron equation~\cite{KMO16TARZP,KMO16TARZPII}.
Thus, we expect that a similar description of $\sigma$-twisted multiline process can be defined such that the weighting is invariant under the action of the combinatorial $R$-matrix.
Yet it seems unlikely that our weighting is related to the steady state distribution for the inhomogeneous TASEP~\cite{AM13,AL14} or TARZP~\cite{KMO16II}.

This paper is organized as follows.
In Section~\ref{sec:background}, we give the necessary background and definitions of MLQs and spectral weight.
In Section~\ref{sec:result}, we give our main results, and use them to prove some of the conjectures in~\cite{AasLin17}.
In Section~\ref{sec:tasep}, we describe the connection between MLQs and the TASEP.
In Section~\ref{sec:JT-proofs}, we give a proof of our Jacobi-Trudi-type formula (Theorem~\ref{thm:determinant_form}).
In Section~\ref{sec:thm_proof}, we give a proof of our main theorem (Theorem~\ref{thm:permutation}).
In Section~\ref{sec:remarks}, we give some additional remarks about our results.

\subsection{Acknowledgements}

We thank Atsuo Kuniba for explaining the results in his papers~\cite{KMO15,KMO16II,KMO16,KMO16TARZP,KMO16TARZPII}.
We thank Olya Mandelshtam for useful discussions on the inhomogeneous TASEP.
We thank Jae-Hoon Kwon for pointing out that the $\SymGp{n}$-action on MLQs comes from an $(\mathfrak{sl}_m \oplus \mathfrak{sl}_n)$-action.
We thank the referees for their comments improving this manuscript.
This work benefited from computations using \textsc{SageMath}~\cite{sage,combinat}.

%=====================================================================
\section{Background and definitions}
\label{sec:background}

Fix a positive integer $n$.
For a nonnegative integer $k$, let $\ive{k}$ denote the set $\set{1, 2, \ldots, k}$, and so $[0] = \emptyset$.
Let $\SymGp{k}$ denote the symmetric group on $\ive{k}$, and let $s_i \in \SymGp{k}$ be the simple transposition of $i$ and $i+1$.
Let $w_0 \in \SymGp{k}$ be the longest element: the permutation $k (k-1) \dotsm 321$ (written in one line notation) that reverses the order of all elements.

We shall refer to the elements $1, 2, \ldots, n \in \ZZ / n \ZZ$ as \defn{sites}.
We visualize them as points on a line that ``wraps around'' cyclically; thus, for example, the sites weakly to the right of a site $i$ are $i, i+1, \ldots, n-1, n, 1, 2, 3, \ldots$ (in this order).

%%%%%%%%%%
\subsection{Words and queues}

Let $\mcW_n$ be the set of words $u = u_1 \dotsm u_n$ in the ordered alphabet $\mcA := \{1 < 2 < 3 < \cdots \}$.
We consider the indices of letters in a word to be taken modulo $n$ (that is, $u_{k+n} = u_k$ for all $k$).
Thus, if $u$ is a word and $i \in \ZZ / n \ZZ$ is a site, then the $i$-th letter $u_i$ of $u$ is well-defined.
We sometimes refer to a letter $u_i = t$ as a \defn{particle at site $i$ of class $t$}.

The \defn{type} of a word $u$ is the vector $\mm = (m_1, m_2, \ldots)$, where $m_i$ is the number of occurrences of $i$ in $u$.
Let $\ell = \max\{i \mid m_i \neq 0 \}$, which we say is the number of \defn{classes} in $u$ or $\mm$.
A word $u$ or type $\mm$ with $\ell$ classes is \defn{packed} if $m_i \neq 0$ for all $1 \leq i \leq \ell$.
A word $w$ of type $\mm$ is \defn{standard} if $m_i \leq 1$ for all $i$.
We will write $1^n = 1 \dotsm 1$ for the (unique) word of type $(n, 0, \dotsc)$.

We \defn{merge} two adjacent classes $i,i+1$ in a word $u$ to obtain a new word by replacing all occurrences of $j$ by $j-1$ in $u$ for each $j = i+1, i+2, \ldots$ in that order.
We denote the merging of $i$ and $i+1$ in $u$ by $\merge{i} u$.
Note that $\merge{i} u$ is packed whenever $u$ is packed.
For $T = \set{t_1 < \cdots < t_k} \subseteq \ive{\ell-1}$, we set $\bigvee_T u := \merge{t_1} \cdots \merge{t_k} u$.
Similarly, the merging of $i,i+1$ in a type $\mm = (m_1, m_2, \ldots)$ is $\merge{i}(\mm) = (m_1, \dotsc, m_{i-1}, m_i + m_{i+1}, m_{i+2}, \ldots)$.
Likewise, we define $\bigvee_T \mm$ for a type $\mm$.
These operations interact as one would hope:
If the type of a word $u$ is $\mm$, then the type of $\merge{i} u$ is $\merge{i}(\mm)$.

Fix a word $u \in \mcW_n$, and let $\mm = (m_1, m_2, \ldots)$ be the type of $u$.
For each $i \geq 0$, set
\begin{equation}
\label{eq:type_partial_sums}
p_i(\mm) := m_1 + m_2 + \cdots + m_i.
\end{equation}
When $\mm$ is clear, we simply write $p_i$ for this.
(Thus, $p_0 = 0$ and $p_i = n$ for sufficiently large $i$.)

\begin{vershort}
We define an \defn{$r$-queue} $q$ to be any subset of $\ive{n}$ of size $r$. When $r$ is clear, we will simply call $q$ a \defn{queue}.
We equate $q$ with a function from $\mcW_n$ to itself defined as follows.
For any $u \in \mcW_n$, the following algorithm from~\cite[Sec.~4.5]{AAMP} computes $v = q(u)$.
\end{vershort}
\begin{verlong}
We define a \defn{queue} to be any set of sites.
A queue of size $r$ will be called an \defn{$r$-queue}.

We shall now define an action of queues on words\footnote{%
A more formal description will follow after
Example~\ref{ex:first_queue}.}.
Namely, if $q$ is any queue and $u \in \mcW_n$ is a word,
then a new word $v = q(u) \in \mcW_n$ is defined by the following algorithm from~\cite[Sec.~4.5]{AAMP}:
\end{verlong}
Choose an ordering of the sites $i \in \ive{n}$ so that the corresponding values $u_i$ are weakly increasing.
We construct a perfect matching on the sites in $u$ and $v$ as follows.
\begin{description}
\item[Phase~I]
   For each of the last $n - r$ sites $i \in \ive{n}$ in decreasing order (in our ordering), match the first site $j$ weakly to the left (cyclically) of $i$ such that $j \notin q$ that is unmatched and set $v_j = u_i + 1$.
\item[Phase~II]
   For each of the remaining $r$ sites $i \in \ive{n}$ in increasing order (in our ordering), match the site with value $u_i$ with the first site $j$ weakly to the right (cyclically) of $i$ such that $j \in q$ that is unmatched and set $v_j = u_i + 1$.
\end{description}
There is a choice inherent in ordering the sites (since some letters of $u$ can be equal), but we will later see that this choice does not affect the word $v$ obtained.

\begin{example}
\label{ex:first_queue}
We consider the $4$-queue $q = \{1, 4, 8, 9\}$, and let $u = 346613321$.
We order the sites as $598167243$.
To compute $q(u)$, draw the following diagram
(whose upper row shows $u$, whose lower row shows $q(u)$,
and whose middle row represents the set $q$ by placing balls in the positions of its elements):
\[
\begin{tikzpicture}[>=latex,rounded corners,yscale=1.5,xscale=1.2,baseline=0]
\def\passwidth{3pt};
\node (i1) at (1,1) {$3$};
\node (i2) at (2,1) {$4$};
\node (i3) at (3,1) {$6$};
\node (i4) at (4,1) {$6$};
\node (i5) at (5,1) {$1$};
\node (i6) at (6,1) {$3$};
\node (i7) at (7,1) {$3$};
\node (i8) at (8,1) {$2$};
\node (i9) at (9,1) {$1$};
\node (t1) at (1,-1) {$2$};
\node (t2) at (2,-1) {$7$};
\node (t3) at (3,-1) {$7$};
\node (t4) at (4,-1) {$3$};
\node (t5) at (5,-1) {$4$};
\node (t6) at (6,-1) {$4$};
\node (t7) at (7,-1) {$5$};
\node (t8) at (8,-1) {$1$};
\node (t9) at (9,-1) {$1$};
\node[circle,draw=black] (q1) at (1,0) {};
\node[circle,draw=black] (q2) at (4,0) {};
\node[circle,draw=black] (q3) at (8,0) {};
\node[circle,draw=black] (q4) at (9,0) {};
\draw[->,red] (i4) -- (4,0.5) .. controls (3.8,0.2) and (3.5,0) .. (3.1,0) -- (2,0) -- (t2);
\draw[->,red] (i3) -- (t3);
\draw[->,red] (i2) -- (2,0.15) .. controls (1.8,-0.2) and (1.6,-0.25) .. (1.3,-0.25) -- (0,-0.25);
\draw[>->,red] (10,-0.25) -- (8,-0.25) -- (7,-0.25) -- (t7);
\draw[->,red] (i6) -- (t6);
\draw[->,red] (i7) -- (7,0) -- (5,0) -- (t5);
\draw[white,line width=\passwidth] (5,0.28) -- (7.3,0.28) .. controls (7.7,0.28) and (7.85,0.12) .. (q3);  % To simulate underpass
\draw[->,blue] (i5) -- (5,0.28) -- (7.3,0.28) .. controls (7.7,0.28) and (7.85,0.12) .. (q3);
\draw[white,line width=\passwidth] (q3) -- (t8);  % To simulate underpass
\draw[->,blue] (q3) -- (t8);
\draw[->,blue] (i9) -- (q4);
\draw[white,line width=\passwidth] (q4) -- (t9);  % To simulate underpass
\draw[->,blue] (q4) -- (t9);
\draw[white,line width=\passwidth] (1,0.28) -- (3.3,0.28) .. controls (3.7,0.28) and (3.85,0.12) .. (q2);  % To simulate underpass
\draw[->,blue] (i1) -- (1,0.28) -- (3.3,0.28) .. controls (3.7,0.28) and (3.85,0.12) .. (q2);
\draw[->,blue] (q2) -- (t4);
\draw[white,line width=\passwidth] (i8) -- (8,0.28) -- (10,0.28);  % To simulate underpass
\draw[->,blue] (i8) -- (8,0.28) -- (10,0.28);
\draw[>->,blue] (0,0.28) -- (0.3,0.28) .. controls (0.7,0.28) and (0.85,0.12) .. (q1);
\draw[white,line width=\passwidth] (q1) -- (t1);  % To simulate underpass
\draw[->,blue] (q1) -- (t1);
\end{tikzpicture}
\]
\begin{vershort}
where the paths in red correspond to Phase~I and those in blue are from Phase~II.
\end{vershort}
\begin{verlong}
where the paths in red correspond to Phase~I and those in blue are from Phase~II
(and where we have picked the permutation $\tup{i_1, i_2, \ldots, i_n}
= \tup{5, 9, 8, 1, 7, 6, 2, 4, 3}$ out of a total of $2! \cdot 3! \cdot 2! = 24$ permutations
satisfying $u_{i_1} \leq u_{i_2} \leq \cdots \leq u_{i_n}$;
but any other among them would lead to the same result).
\end{verlong}
Hence, we have $q(346613321) = 277344511$.
\end{example}

Let us give a more detailed formulation of the algorithm that will be useful for our proofs.
In the beginning, no letter of $v \in \mcW_n$ is set.
Choose a permutation $\tup{i_1, i_2, \ldots, i_n}$ of $\tup{1, 2, \ldots, n}$ such that $u_{i_1} \leq u_{i_2} \leq \cdots \leq u_{i_n}$.
(This corresponds to the ordering of the sites that had to be chosen in the previous formulation of the algorithm.)
\begin{verlong}
(The result of the algorithm will not depend on this choice, as we will show in Lemma~\ref{lemma:order_indep} below.)
\end{verlong}

\begin{description}
\item[Phase~I]
  For $i = i_n, i_{n-1}, \ldots, i_{\abs{q}+1}$, do the following.
    Find the first site $j$ weakly to the left (cyclically) of $i$ such that $j \notin q$ and $v_j$ is not set.
    Then set $v_j = u_i + 1$.

\item[Phase~II]
  For $i = i_1, i_2, \ldots, i_{\abs{q}}$, do the following.
    Find the first site $j$ weakly to the right (cyclically) of $i$ such that $j \in q$ and $v_j$ is not set.
    Then set $v_j = u_i$.
\end{description}

\begin{remark}
\label{rmk:order-agnostic}
Consider the above algorithm.
Notice that Phase~I sets $v_j$ for all $j \notin q$ (because it contains
$n - \abs{q}$ steps, and sets one such $v_j$ per step),
whereas Phase~II sets $v_j$ for all $j \in q$ (for similar reasons).
Hence, at the end of the algorithm, all letters of $v$ are set, and we never
run out of $j$'s in either phase.

Since Phase~I only deals with $j \notin q$, and Phase~II only with $j\in q$,
the two phases can be arbitrarily interleaved
(\textit{i.e.}, we can perform the steps of the algorithm in any order as long as
the steps of Phase~I (resp.\ Phase~II) are processed in the order $i = i_n, i_{n-1}, \ldots, i_{\abs{q}+1}$
(resp.\ $i = i_1, i_2, \ldots, i_{\abs{q}}$).
\end{remark}

\begin{lemma}
\label{lemma:order_indep}
The resulting word $v = q(u)$ does not depend on the choice of permutation $(i_1, i_2, \dotsc, i_n)$
(as long as $u_{i_1} \leq u_{i_2} \leq \cdots \leq u_{i_n}$ holds).
\end{lemma}

\begin{proof}
Consider some $k \in \ive{n-1}$ such that $u_{i_k} = u_{i_{k+1}}$.
If we switch the two adjacent entries $i_k$ and $i_{k+1}$ of the
permutation $\tup{i_1, i_2, \ldots, i_n}$,
then the resulting word $v$ is unchanged.
Indeed, if we set $h = u_{i_k} = u_{i_{k+1}}$, then:
\begin{itemize}
 \item If $k < \abs{q}$, then this switch interchanges two consecutive
       steps in Phase~II, causing the corresponding letters $v_j$ to
       get set in a possibly different order; but this does not change $v$
       because these two letters are set to the same value
       (namely, to $h+1$).
 \item If $k > \abs{q}$, then a similar argument works (using Phase~I instead).
 \item If $k = \abs{q}$, then recall from Remark~\ref{rmk:order-agnostic} that
       the two phases can be arbitrarily interleaved.
       In particular, we can first perform all but the last step of Phase~I,
       then perform all but the last step of Phase~II,
       and finally perform the remaining two steps.
       The switch only affects these final two steps.
       However, the effect of these two steps is simply that the unique
       remaining unset $v_j$ with $j \in q$ gets set to $h$,
       and the unique remaining unset $v_j$ with $j \notin q$ gets set to $h+1$.
       The switch clearly does not change this behavior,
       since it does not depend on $i_k$ and $i_{k+1}$.
\end{itemize}
\end{proof}

Lemma~\ref{lemma:order_indep} states that the order between sites $i$ with equal $u_i$ does not matter.

\begin{remark}
\label{rmk:t-splitting}
Let $u$ and $q$ be as before. Let $r = \abs{q}$.
There exists a $t \in \ive{\ell}$ such that
$
p_{t-1} \leq r \leq p_t.
$
The word $v = q(u)$ then has type
\begin{equation}
\label{eq:queue_type_change}
(m_1, \dots, m_{t-1}, r-p_{t-1}, p_{t}-r, m_{t+1}, m_{t+2}, \ldots).
\end{equation}
Note that $p_{t} - r = m_{t} + (p_{t-1} - r)$.
We think of this as splitting the class $t$ into two new classes $t$ and $t+1$.

For all $i$ processed in Phase~I (resp.\ Phase~II) of the algorithm, we have $u_i \geq t$ (resp.~$u_i \leq t$).
The queue $q$ can be reconstructed from $v$ and $t$ as the set of all $j \in \ive{n}$ satisfying $v_j \leq t$.
\end{remark}

\begin{example}
For $u$ in Example~\ref{ex:first_queue}, the type of $u$ is $\mm = (2, 1, 3, 1, 0, 2, 0, \ldots)$ with $p_2 = 3$ and $p_3 = 6$.
Thus, the $t$ in Remark~\ref{rmk:t-splitting} equals $3$.
The word $q(u)$ has type $(2,1,1,2,1,0,2,0,\ldots)$.
\end{example}

We illustrate the situation $v = q(u)$ with a $2 \times n$ array, where the first row is the word $u$
and the second row has a circle labeled $v_j$ for $j \in q$ or a square labeled $v_j$ for $j \notin q$ in position $j$.
Using this convention, we can write Example~\ref{ex:first_queue} as
\begin{equation}
\begin{tikzpicture}[baseline=10]
  \def\ll{0.65}   % level 2
  \foreach \i in {5,9} { \node at (\i,\ll) {$1$}; }
  \foreach \i in {8} { \node at (\i,\ll) {$2$}; }
  \foreach \i in {1,6,7} { \node at (\i,\ll) {$3$}; }
  \foreach \i in {2} { \node at (\i,\ll) {$4$}; }
  \foreach \i in {3,4} { \node at (\i,\ll) {$6$}; }
  \foreach \i in {1,4,8,9} { \draw (\i,0) circle (0.3); }
  \foreach \i in {2,3,5,6,7} { \draw (\i-.3,0-.3) rectangle +(0.6,+0.6); }
  \foreach \i in {8,9} { \node at (\i,0) {$1$}; }
  \foreach \i in {1} { \node at (\i,0) {$2$}; }
  \foreach \i in {4} { \node at (\i,0) {$3$}; }
  \foreach \i in {5,6} { \node at (\i,0) {$4$}; }
  \foreach \i in {7} { \node at (\i,0) {$5$}; }
  \foreach \i in {2,3} { \node at (\i,0) {$7$}; }
\end{tikzpicture}
\label{eq:boxes-and-balls-1}
\end{equation}
We call this the \defn{graveyard diagram} of $q$ and $u$.

There is a natural duality in the algorithm above.
For each queue $q$, let $q^*$ be the \defn{contragredient dual queue}, defined by $\tup{i \in q^*} \Longleftrightarrow \tup{n+1-i \notin q}$.
Similarly, for each word $u$ with $\ell$ classes, let $u^*$ be the \defn{contragredient dual word} defined by $u^*_i = \ell + 1 - u_{n+1-i}$.
For a fixed $k \in \ive{\ell}$, we call $\ell + 1 - k$ the \defn{contragredient dual letter} of $k$.
Note that if $q$ is an $r$-queue, then $q^*$ is the $(n-r)$-queue obtained by reflecting $[n] \setminus q$ through the middle of $[n]$.
Similarly, $u^*$ is obtained by reversing the word $u$ and taking the contragredient dual letters.

\begin{lemma}[Contragredient duality]
  \label{le:dual}
  Let $q$ be a queue and $u$ be a word with $\ell$ classes.
  Then we have $\bigl(q(u) \bigr)^* = q^*(u^*)$, \textit{i.e.}, we have $q(u)_i = \ell + 2 - q^*(u^*)_{n+1-i}$ for all $i$.
\end{lemma}

Here, we treat $q(u)$ as a word with $\ell+1$ classes, even if it may have only $\ell$ classes (in the degenerate case when $q = \ive{n}$).

\begin{proof}
Phase~I (resp.~II) in the construction of $q(u)$ corresponds to Phase~II (resp.~I) in the construction of $q^*(u^*)$ when the word is reversed and the letters are replaced by their contragredient duals.
Hence the claim follows.
\end{proof}

We note that our construction is a minor variation of the construction given by~\cite[\S 3.1]{AAMP}, which is a reformulation of that of~\cite{FM07}.
We more precisely describe the relationship with~\cite{FM07} in Appendix~\ref{app:queue-relations}, where we also discuss a small variant that has appeared in~\cite{AssSea18}.

%%%%%%%%%%
\subsection{Multiline queues}

We now give our main definition of a multiline queue and spectral weight.

\begin{dfn}
For $\sigma \in \SymGp{\ell-1}$, a \defn{$\sigma$-twisted multiline queue (MLQ) of type $\mm = \tup{m_1, m_2, \ldots, m_\ell, 0, 0, \ldots}$}, with $\ell$ classes, is a sequence of queues $\qq = (q_1, \dotsc, q_{\ell-1})$ such that $q_i$ is a $p_{\sigma(i)}(\mm)$-queue and $m_{\ell} = n - p_{\ell-1}(\mm)$ (that is, $p_\ell(\mm) = n$).
When $\sigma$ is the identity permutation, we simply call $\qq$ an \defn{(ordinary) MLQ of type $\mm$}.
We also consider $\qq$ as a function on words by
\[
\qq(u) := q_{\ell-1}\bigl( \cdots q_2\bigl( q_1(u) \bigr) \cdots \bigr).
\]
\end{dfn}

\begin{remark}
Our notion of an (ordinary) MLQ is equivalent to what is called a ``discrete MLQ'' in~\cite[\S 2.2]{AasLin17}, where we can recover the labeling of level $k$ from the word $q_k( \cdots q_1(1^n) \cdots )$.
See Appendix~\ref{app:queue-relations} for more details.
We omit the word ``discrete'' as these are the only MLQs in this note.
\end{remark}

We shall now introduce generating functions for queues.

\begin{dfn}
Let $\xx := \{x_1, x_2, \ldots, x_n\}$ be commuting indeterminates indexed by elements of $\ZZ / n \ZZ$.
(Thus, $x_{n+k} = x_k$ for all $k \in \ZZ$.)
The \defn{weight} of a queue $q$ is
\[
  \wt(q) := \prod_{i \in q} x_i.
\]
The \defn{weight} of a $\sigma$-twisted MLQ $\qq = (q_1, \dotsc, q_{\ell-1})$ is
\[
  \wt(\qq) := \prod_{i=1}^{\ell-1} \wt(q_i).
\]
\end{dfn}

\begin{dfn}
For $\sigma \in \SymGp{\ell-1}$ and a packed word $u$ of type $\mm$ with $\ell$ classes, we define the \defn{$\sigma$-spectral weight} or \defn{$\sigma$-amplitude} as
\[
  \swt{u}_{\sigma} := \sum_{\qq} \wt(\qq),
\]
where the sum is over all $\sigma$-twisted MLQs $\qq$ of type $\mm$ satisfying $u = \qq(1^n)$.
(Recall that $1^n$ denotes the word in $\mcW_n$ whose all letters are $1$.)
When $\sigma = \id$ is the identity permutation, we simply call this the \defn{spectral weight} or \defn{amplitude} and denote it by $\swt{u} := \swt{u}_{\id}$.
\end{dfn}

\begin{example}
\label{ex:spectral_weight_example}
Consider $u = 2312$, a word of type $\mm = \tup{1,2,1,0,0,\ldots}$.
We have the following MLQs $\qq$ (of type $\mm$) such that $\qq(1111) = 2312$:
\[
\begin{tikzpicture}[baseline=-30]
  \def\ll{0.75}   % level scaling
  \foreach \i in {3} { \draw (\i,\ll*-1) circle (0.3); }
  \foreach \i in {1,2,4} { \draw (\i-.3,\ll*-1 -.3) rectangle +(0.6,+0.6); }
  \foreach \i in {1,3,4} { \draw (\i,\ll*-2) circle (0.3); }
  \foreach \i in {2} { \draw (\i-.3,\ll*-2 -.3) rectangle +(0.6,+0.6); }
  \foreach \i in {1,2,3,4} { \node at (\i, \ll*0-.15) {$1$}; }
  \foreach \i in {3} { \node at (\i, \ll*-1) {$1$}; }
  \foreach \i in {1,2,4} { \node at (\i, \ll*-1) {$2$}; }
  \foreach \i in {3} { \node at (\i, \ll*-2) {$1$}; }
  \foreach \i in {1,4} { \node at (\i, \ll*-2) {$2$}; }
  \foreach \i in {2} { \node at (\i, \ll*-2) {$3$}; }
\end{tikzpicture}
\hspace{50pt}
\begin{tikzpicture}[baseline=-30]
  \def\ll{0.75}   % level scaling
  \foreach \i in {2} { \draw (\i,\ll*-1) circle (0.3); }
  \foreach \i in {1,3,4} { \draw (\i-.3,\ll*-1 -.3) rectangle +(0.6,+0.6); }
  \foreach \i in {1,3,4} { \draw (\i,\ll*-2) circle (0.3); }
  \foreach \i in {2} { \draw (\i-.3,\ll*-2 -.3) rectangle +(0.6,+0.6); }
  \foreach \i in {1,2,3,4} { \node at (\i, \ll*0-.15) {$1$}; }
  \foreach \i in {2} { \node at (\i, \ll*-1) {$1$}; }
  \foreach \i in {1,3,4} { \node at (\i, \ll*-1) {$2$}; }
  \foreach \i in {3} { \node at (\i, \ll*-2) {$1$}; }
  \foreach \i in {1,4} { \node at (\i, \ll*-2) {$2$}; }
  \foreach \i in {2} { \node at (\i, \ll*-2) {$3$}; }
\end{tikzpicture}
\]
and hence we obtain $\swt{2312} = x_1 x_3^2 x_4 + x_1 x_2 x_3 x_4$.
Next, we let $\sigma = s_1$,
and we construct the corresponding $\sigma$-twisted MLQs $\qq$ (of type $\mm$)
\[
\begin{tikzpicture}[baseline=-30]
  \def\ll{0.75}   % level scaling
  \foreach \i in {1,3,4} { \draw (\i,\ll*-1) circle (0.3); }
  \foreach \i in {2} { \draw (\i-.3,\ll*-1 -.3) rectangle +(0.6,+0.6); }
  \foreach \i in {3} { \draw (\i,\ll*-2) circle (0.3); }
  \foreach \i in {1,2,4} { \draw (\i-.3,\ll*-2 -.3) rectangle +(0.6,+0.6); }
  \foreach \i in {1,2,3,4} { \node at (\i, \ll*0-.15) {$1$}; }
  \foreach \i in {1,3,4} { \node at (\i, \ll*-1) {$1$}; }
  \foreach \i in {2} { \node at (\i, \ll*-1) {$2$}; }
  \foreach \i in {3} { \node at (\i, \ll*-2) {$1$}; }
  \foreach \i in {1,4} { \node at (\i, \ll*-2) {$2$}; }
  \foreach \i in {2} { \node at (\i, \ll*-2) {$3$}; }
\end{tikzpicture}
\hspace{50pt}
\begin{tikzpicture}[baseline=-30]
  \def\ll{0.75}   % level scaling
  \foreach \i in {1,2,4} { \draw (\i,\ll*-1) circle (0.3); }
  \foreach \i in {3} { \draw (\i-.3,\ll*-1 -.3) rectangle +(0.6,+0.6); }
  \foreach \i in {3} { \draw (\i,\ll*-2) circle (0.3); }
  \foreach \i in {1,2,4} { \draw (\i-.3,\ll*-2 -.3) rectangle +(0.6,+0.6); }
  \foreach \i in {1,2,3,4} { \node at (\i, \ll*0-.15) {$1$}; }
  \foreach \i in {1,2,4} { \node at (\i, \ll*-1) {$1$}; }
  \foreach \i in {3} { \node at (\i, \ll*-1) {$2$}; }
  \foreach \i in {3} { \node at (\i, \ll*-2) {$1$}; }
  \foreach \i in {1,4} { \node at (\i, \ll*-2) {$2$}; }
  \foreach \i in {2} { \node at (\i, \ll*-2) {$3$}; }
\end{tikzpicture}
\]
to obtain $\swt{2312}_{\sigma} = x_1 x_3^2 x_4 + x_1 x_2 x_3 x_4$.
\end{example}

\begin{lemma}
\label{lemma:mlq-type}
Let $\ell \geq 1$ and $\sigma \in \SymGp{\ell-1}$.
Let $\qq$ be a $\sigma$-twisted MLQ.
Then the type of the word $\qq(1^n)$ is the type of $\qq$.
\end{lemma}

\begin{vershort}
\begin{proof}
Repeatedly apply~\eqref{eq:queue_type_change}.
\end{proof}
\end{vershort}

\begin{verlong}
\begin{proof}
Write the $\sigma$-twisted MLQ $\qq$ as $\tup{q_1, \ldots, q_{\ell-1}}$.
Let $\mm = \tup{m_1, m_2, \ldots, m_\ell, 0, 0, \ldots}$ be the type of $\qq$.
Thus, $m_1 + m_2 + \cdots + m_\ell = p_{\ell}(\mm) = n$.
Now, the word $1^n$ has type $\tup{n, 0, 0, \ldots} = \tup{m_1 + m_2 + \cdots + m_\ell, 0, 0, \ldots}$
(since $n = m_1 + m_2 + \cdots + m_\ell$).
Every time we apply one of the queues $q_i$ to this word, the type changes in a simple way (because of~\eqref{eq:queue_type_change}): Namely, the plus sign between $m_{\sigma(i)}$ and $m_{\sigma(i)+1}$ turns into a comma (so, for example, the application of $q_1$ transforms it into $\tup{m_1 + m_2 + \cdots + m_{\sigma(1)}, m_{\sigma(1)+1} + m_{\sigma(1)+2} + \cdots + m_\ell, 0, 0, \ldots}$).
Hence, the action of $\qq$ transforms $1^n$ into a word whose type has all the plus signs replaced by commas -- \textit{i.e.}, whose type is $\tup{m_1, m_2, \ldots, m_\ell, 0, 0, \ldots} = \mm$.
In other words, the type of the word $\qq(1^n)$ is $\mm$.
This proves Lemma~\ref{lemma:mlq-type}.
\end{proof}

\begin{proof}[Alternative proof of Lemma~\ref{lemma:mlq-type}.]
Here is an alternative way of presenting essentially the same argument, using a different point of view:
Since $\tup{q_1, \ldots, q_{\ell-1}} = \qq$ is a $\sigma$-twisted MLQ of type $\mm$, we know that the sizes $\abs{q_1}, \abs{q_2}, \ldots, \abs{q_{\ell-1}}$ of its queues are a permutation of $p_1(\mm), p_2(\mm), \ldots, p_{\ell-1}(\mm)$.
Thus, $p_1(\mm), p_2(\mm), \ldots, p_{\ell-1}(\mm)$ is the weakly increasing permutation of the sequence $\abs{q_1}, \abs{q_2}, \ldots, \abs{q_{\ell-1}}$.
Now, define the \defn{p-sequence} of a type $\nn$ to be the weakly increasing infinite sequence $\tup{p_0(\nn), p_1(\nn), p_2(\nn), \ldots}$ of integers (which uniquely determines $\nn$).
Furthermore, if $u$ is a word of type $\nn$, then we define the p-sequence of $u$ to be the p-sequence of $\nn$.
Then,~\eqref{eq:queue_type_change} shows the following: If $q$ is a queue and $u$ is a word, then the p-sequence of $q(u)$ is obtained from the p-sequence of $u$ by inserting $\abs{q}$ into the p-sequence at the appropriate position (appropriate in the sense that the resulting sequence is still weakly increasing).
Hence, the p-sequence of the word $\qq (1^n)$ is obtained from the p-sequence $\tup{0, n, n, n, \ldots}$ of the word $1^n$ by inserting $\abs{q_1}, \abs{q_2}, \ldots, \abs{q_{\ell-1}}$ at the appropriate positions.
In other words, the p-sequence of the word $\qq (1^n)$ is $\tup{0, p_1(\mm), p_2(\mm), \ldots, p_{\ell-1}(\mm), n, n, \ldots, n}$ (since $p_1(\mm), p_2(\mm), \ldots, p_{\ell-1}(\mm)$ is the weakly increasing permutation of the sequence $\abs{q_1}, \abs{q_2}, \ldots, \abs{q_{\ell-1}}$).
In other words, this p-sequence is $\tup{p_0(\mm), p_1(\mm), p_2(\mm), \ldots}$.
Hence, the type of the word $\qq (1^n)$ is $\mm$.
\end{proof}
\end{verlong}

\begin{example}
Let $\ell \geq 1$.
Let $u$ be a packed word with $\ell$ classes and type $\mm$.
For each $k \in \ive{\ell-1}$, let $q_k$ be the set of all sites $i$ such that $u_i \leq k$.
It is then easy to check that $\qq := \tup{q_1, q_2, \ldots, q_{\ell-1}}$ is an
MLQ of type $\mm$ satisfying $\qq(1^n) = u$ and has weight
$\wt(\qq) = \prod_{j \in \ZZ / n \ZZ} x_j^{\ell - u_j}$.
Hence, $\swt{u} \neq 0$ (as a polynomial over $\ZZ$).
\end{example}

%%%%%%%%%%
\subsection{Symmetric polynomials}

We also need the \defn{elementary symmetric polynomials} and \defn{complete homogeneous symmetric polynomials}.
Recall that they are defined for each $N \in \NN$ and $k \geq 0$ and any $N$ indeterminates $y_1, y_2, \ldots, y_N$ by
\begin{align*}
e_k(y_1, y_2, \dotsc, y_N) & = \sum_{1 \leq i_1 < \cdots < i_k \leq N} y_{i_1} \dotsm y_{i_k},
\\ h_k(y_1, y_2, \dotsc, y_N) & = \sum_{1 \leq i_1 \leq \cdots \leq i_k \leq N} y_{i_1} \dotsm y_{i_k},
\end{align*}
respectively.
We define $e_k(y_1, \dotsc, y_N) = 0$ and $h_k(y_1, \dotsc, y_N) = 0$ for $k < 0$.
For more details on symmetric polynomials, we refer the reader to~\cite[Ch.~7]{Stanley-EC2}.

%=====================================================================
\section{Main results}
\label{sec:result}

In this section, we state our main results and use them to prove the commutativity conjecture of~\cite{AAMP} and~\cite[Conj.~3.10]{AasLin17}.

\subsection{\texorpdfstring{$\sigma$}{sigma}-independence of the spectral weight}

\begin{thm}
\label{thm:permutation}
  Let $u$ be a packed word of type $\mm$ with $\ell$ classes.
  For any $\sigma \in \SymGp{\ell-1}$, we have 
  \[
  \swt{u} = \swt{u}_{\sigma}.
  \]
\end{thm}

We will give the proof of Theorem~\ref{thm:permutation} in Section~\ref{sec:thm_proof}.
Note that we have verified a particular case of Theorem~\ref{thm:permutation} in Example~\ref{ex:spectral_weight_example}.

\subsection{Queue action and merges}
We shall next study the interplay between the action of queues on words and the merging of adjacent classes.

If $w$ is a word of type $\mm$, and if $j \geq 0$, then $p_j(\mm)$ is the number of letters of $w$ that are at most $j$.

For a word $w$ and a nonnegative integer $k$, we let \defn{$\vee^{(k)} w$} be the word obtained from $w$ by decrementing (by $1$) all but the $k$ smallest letters of $w$.
This is only well-defined if these $k$ smallest letters are determined uniquely and the remaining $n-k$ letters are $> 1$.
In other words, this is only well-defined if $k \in \set{ p_j(\mm) \mid j \geq 1 }$, where $\mm$ is the type of $w$.
Note that $\vee^{(k)} w = \merge{j} w$, where $j$ is such that $k = p_j(\mm)$.
% Note that $j$ is not necessarily unique, in which case $w$ is not a packed word, but the result is independent of the choice of $j$.

\begin{lemma}
\label{lemma:queue_merge_commute}
Let $u \in \mcW_n$ be a word of type $\mm$.
Let $k \in \set{ p_j(\mm) \mid j \geq 1 }$.
If $q$ is a queue, then
\[
\vee^{(k)} q(u) = q(\vee^{(k)} u).
\]
In particular, both $\vee^{(k)} q(u)$ and $\vee^{(k)} u$ are well-defined.
\end{lemma}

\begin{proof}
The word $\vee^{(k)} u$ is well-defined since $k = p_j(\mm)$ for some $j \geq 1$;
furthermore, $\vee^{(k)} q(u)$ is well-defined since the type $\nn$ of $q(u)$
satisfies
\[
k \in \set{ p_j(\mm) \mid j \geq 1 } \subseteq \set{ p_j(\mm) \mid j \geq 1 } \cup \set{\abs{q}} = \set{ p_j(\nn) \mid j \geq 1 }.
\]
Thus, it remains to prove $\vee^{(k)} q(u) = q(\vee^{(k)} u)$.
The permutation $\tup{i_1, i_2, \dotsc, i_n}$ in the construction of $q(u)$
also works for the construction of $q(\vee^{(k)} u)$,
since $(\vee^{(k)} u)_a \leq (\vee^{(k)} u)_b$ whenever $u_a \leq u_b$.
Consequently, the construction of $q(\vee^{(k)} u)$ proceeds exactly like
the construction of $q(u)$ (with the same entries being set in the same
order), except that all but the $k$ smallest letters are now smaller by $1$.
Hence, $q(\vee^{(k)} u)$ is obtained from $q(u)$ by decrementing (by $1$)
all but the $k$ smallest letters of $q(u)$.
% (since $\vee^{(k)} u$ is obtained in this way from $u$).
However, the word $\vee^{(k)} q(u)$ is obtained from $q(u)$ in exactly the same way.
Therefore, we have $\vee^{(k)} q(u) = q(\vee^{(k)} u)$ as claimed.
\end{proof}

\begin{example}
Consider $q = \set{1,4,5,9,10}$ and $v = 3455313321$:
\[
\begin{tikzpicture}[baseline=-7]
  \def\ll{0.65}   % level 2
  \foreach \i in {6,10} { \node at (\i,\ll) {$1$}; }
  \foreach \i in {9} { \node at (\i,\ll) {$2$}; }
  \foreach \i in {1,5,7,8} { \node at (\i,\ll) {$3$}; }
  \foreach \i in {2} { \node at (\i,\ll) {$4$}; }
  \foreach \i in {3,4} { \node at (\i,\ll) {$5$}; }
  \foreach \i in {1,4,5,9,10} { \draw (\i,0) circle (0.3); }
  \foreach \i in {2,3,6,7,8} { \draw (\i-.3,0-.3) rectangle +(0.6,+0.6); }
  \foreach \i in {9,10} { \node at (\i,0) {$1$}; }
  \foreach \i in {1} { \node at (\i,0) {$2$}; }
  \foreach \i in {4,5} { \node at (\i,0) {$3$}; }
  \foreach \i in {6,7} { \node at (\i,0) {$4$}; }
  \foreach \i in {8} { \node at (\i,0) {$5$}; }
  \foreach \i in {2,3} { \node at (\i,0) {$6$}; }
\end{tikzpicture}\ .
\]
Let $u = 3566413321$, and note that $v = \merge{3} u = \vee^{(6)} u$ and
\[
\begin{tikzpicture}[baseline=-7]
  \def\ll{0.65}   % level 2
  \foreach \i in {6,10} { \node at (\i,\ll) {$1$}; }
  \foreach \i in {9} { \node at (\i,\ll) {$2$}; }
  \foreach \i in {1,7,8} { \node at (\i,\ll) {$3$}; }
  \foreach \i in {5} { \node at (\i,\ll) {$4$}; }
  \foreach \i in {2} { \node at (\i,\ll) {$5$}; }
  \foreach \i in {3,4} { \node at (\i,\ll) {$6$}; }
  \foreach \i in {1,4,5,9,10} { \draw (\i,0) circle (0.3); }
  \foreach \i in {2,3,6,7,8} { \draw (\i-.3,0-.3) rectangle +(0.6,+0.6); }
  \foreach \i in {9,10} { \node at (\i,0) {$1$}; }
  \foreach \i in {1} { \node at (\i,0) {$2$}; }
  \foreach \i in {4,5} { \node at (\i,0) {$3$}; }
  \foreach \i in {6} { \node at (\i,0) {$4$}; }
  \foreach \i in {7} { \node[red] at (\i,0) {$5$}; }
  \foreach \i in {8} { \node[dgreencolor] at (\i,0) {$6$}; }
  \foreach \i in {2,3} { \node[dgreencolor] at (\i,0) {$7$}; }
\end{tikzpicture}\ ,
\]
where we have $2663344511 = \merge{4} 2773345611 = \vee^{(6)} 2773345611$.
Similarly, let $u' = 4566413421$, and note that $v = \merge{3} u'$ and
\[
\begin{tikzpicture}[baseline=-7]
  \def\ll{0.65}   % level 2
  \foreach \i in {6,10} { \node at (\i,\ll) {$1$}; }
  \foreach \i in {9} { \node at (\i,\ll) {$2$}; }
  \foreach \i in {7} { \node at (\i,\ll) {$3$}; }
  \foreach \i in {1,5,8} { \node at (\i,\ll) {$4$}; }
  \foreach \i in {2} { \node at (\i,\ll) {$5$}; }
  \foreach \i in {3,4} { \node at (\i,\ll) {$6$}; }
  \foreach \i in {1,4,5,9,10} { \draw (\i,0) circle (0.3); }
  \foreach \i in {2,3,6,7,8} { \draw (\i-.3,0-.3) rectangle +(0.6,+0.6); }
  \foreach \i in {9,10} { \node at (\i,0) {$1$}; }
  \foreach \i in {1} { \node at (\i,0) {$2$}; }
  \foreach \i in {4} { \node[blue] at (\i,0) {$3$}; }
  \foreach \i in {5} { \node[dgreencolor] at (\i,0) {$4$}; }
  \foreach \i in {6,7} { \node[dgreencolor] at (\i,0) {$5$}; }
  \foreach \i in {8} { \node[dgreencolor] at (\i,0) {$6$}; }
  \foreach \i in {2,3} { \node[dgreencolor] at (\i,0) {$7$}; }
\end{tikzpicture}\ ,
\]
where we have $2663344511 = \merge{4} 2773455611$.
\end{example}

\begin{lemma}
\label{lemma:queue_merge}
  Let $q'$ be a $p_i(\mm)$-queue for some type $\mm$ and some $i \geq 1$.
  Let the type $\mm$ have $\ell$ classes, and let $\sigma \in \SymGp{\ell-1}$.
  Let $\qq$ be a $\sigma$-twisted MLQ of type $\merge{i}\mm$.
  For the word
  $
  u = \qq\bigl( q'(1^n) \bigr),
  $
  we have
  \[
  \qq(1^n) = \merge{i} u.
  \]
\end{lemma}

\begin{proof}
Write the $\sigma$-twisted MLQ $\qq$ as $\tup{q_1, \ldots, q_{\ell-1}}$.
It has type $\merge{i} \mm$.

Set $\qq' := \tup{q', q_1, \ldots, q_{\ell-1}}$; this is easily
seen to be a $\zeta$-twisted MLQ of type $\mm$, for some
$\zeta \in \SymGp{\ell}$
(since the multiset of the $p_k(\merge{i}\mm)$ for $k \geq 1$
is precisely the multiset of the $p_k(\mm)$ for $k \geq 1$
with one copy of $p_i(\mm) = \abs{q'}$ removed).
Furthermore, the definition of $u$ becomes $u = \qq' (1^n)$.
Hence, the type of $u$ is $\mm$ (by Lemma~\ref{lemma:mlq-type}).

Set $k = p_i(\mm)$. Then, $q'$ is a $k$-queue, so that the type
$\nn$ of the word $q'(1^n)$ satisfies $p_1(\nn) = k$.
Hence, any word obtained by actions of queues on $q'(1^n)$
will have its type $\nn$ satisfy
$k \in \set{ p_j(\nn) \mid j \geq 1 } $.

Since the type of $u$ is $\mm$, and since $k = p_i(\mm)$, we have
\begin{equation}
\label{pf.lemma:queue_merge.4}
 \merge{i} u = \vee^{(k)} (u) = \vee^{(k)} \qq\bigl( q' (1^n) \bigr) = \qq\bigl( \vee^{(k)} q' (1^n) \bigr)
\end{equation}
by repeated use of Lemma~\ref{lemma:queue_merge_commute}
(since any word obtained by actions of queues on $q'(1^n)$
will have its type $\nn$ satisfy
$k \in \set{ p_j(\nn) \mid j \geq 1 } $).
It is clear that $\vee^{(k)} q'(1^n) = 1^n$,
and so~\eqref{pf.lemma:queue_merge.4} becomes $\qq(1^n) = \merge{i} u$.
\end{proof}

\subsection{Spectral weights of merged words}

The preceding lemmas will help us establish a rule for products of spectral weights with elementary symmetric polynomials (somewhat similar to the dual Pieri rule, \textit{e.g.}, \cite[Section 7.15]{Stanley-EC2}):

\begin{thm}
\label{thm:merge}
  Let $\mm$ be a type with $m_i \neq 0$ and $m_{i+1} \neq 0$.
  Let $v$ be a packed word of type $\merge{i}\mm$.
  Then,
  \[
  \swt{v} e_{p_i(\mm)}(x_1, \dotsc, x_n) = \sum_u \swt{u},
  \]
where we sum over all $u$ of type $\mm$ such that $v = \merge{i} u$.
\end{thm}

\begin{proof}
  The type $\mm$ is packed (since $m_i \neq 0$ and $m_{i+1} \neq 0$,
  and since $\merge{i}\mm$ is packed).
  Let $\merge{i}\mm$ have $\ell$ classes; then, $\mm$ has $\ell+1$
  classes.
  First note that
  \[
  \swt{v} e_{p_i(\mm)}(x_1, \dotsc, x_n) = \sum_{(\qq,q')} \wt(\qq) \wt(q'),
  \]
  where we sum over all pairs $(\qq, q')$ such that
  \begin{itemize}
  \item $\qq = (q_1, \dotsc, q_{\ell-1})$ is an MLQ of type $\merge{i}\mm$ such that $v = \qq(1^n)$ and
  \item $q'$ is a $p_i(\mm)$-queue.
  \end{itemize}
  Given any such pair $\tup{\qq, q'}$, we set
  \[
  \theta(\qq, q') := (q', q_1, \dotsc, q_{\ell-1});
  \]
  the result is a $(s_{i-1} \dotsm s_2 s_1)$-twisted MLQ of type $\mm$ with weight
  $\wt\bigl( \theta(\qq, q') \bigr) = \wt(\qq) \wt(q')$.
  But recall that $v = \qq(1^n$); thus, by Lemma~\ref{lemma:queue_merge}, we have
  \[
  v = \qq(1^n) = \merge{i} \qq\bigl( q'(1^n) \bigr) = \merge{i} \bigl( \theta(\qq, q')(1^n) \bigr).
  \]
  Thus, we have defined a weight preserving bijection $\theta$ from the set of all pairs $(\qq, q')$ as above
  to the set of all $(s_{i-1} \dotsm s_1)$-twisted MLQs $\widetilde{\qq}$ of type $\mm$ satisfying $v = \merge{i} \widetilde{\qq} (1^n)$.
  Hence, we have
  \[
  \sum_{(\qq,q')} \wt(\qq) \wt(q')
  = \sum_{\widetilde{\qq}} \wt(\widetilde{\qq})
  = \sum_u \swt{u}_{s_{i-1} \dotsm s_1} ,
  \]
  where the last sum is over all words $u$ of type $\mm$ satisfying $v = \merge{i} u$.
  Finally, we have $\swt{u}_{s_{i-1} \dotsm s_1} = \swt{u}$ for all such $u$ by Theorem~\ref{thm:permutation}.
  Combining all the above equalities, we find $\swt{v} e_{p_i(\mm)}(x_1, \dotsc, x_n) = \sum_u \swt{u}$.
\end{proof}

\begin{remark}
\label{rmk:bijective_proof}
Our proof of Theorem~\ref{thm:permutation} is by constructing a bijection $\omega$, and hence, we can give a bijective proof of Theorem~\ref{thm:merge} by the composition $\omega \circ \theta$.
\end{remark}

\begin{example}
\label{ex:checking_merge_thm}
Suppose $n = 5$.
Let $v = 13234$, and we have that $v = \merge{3} u$ if and only if $u \in \set{13245, 14235}$.
By examining all possible MLQs for these words, we obtain
\begin{align*}
\swt{13234} & = x_1 x_2 x_3^2 x_4 (x_1^2 + x_1 x_4 + x_1 x_5 + x_4 x_5 + x_5^2),
\\ \swt{13245} & = x_1 x_2 x_3^2 x_4 (x_1^2 + x_1x_4 + x_1x_5 + x_4^2 + x_4x_5 + x_5^2)
\\ & \hspace{20pt} \times (x_1x_2x_3 + x_1x_2x_5+x_1x_3x_5+x_2x_3x_5),
\\ \swt{14235} & = x_1x_2x_3^2x_4^2 (x_1^3x_2 + x_1^3x_3 + x_1^3x_5 + x_1^2x_2x_3 + x_1^2x_2x_4 + 2x_1^2x_2x_5
\\ & \hspace{60pt} + x_1^2x_3x_4 + 2x_1^2x_3x_5 + x_1^2x_4x_5 + x_1^2x_5^2 + x_1x_2x_3x_5
\\ & \hspace{60pt} + x_1x_2x_4x_5 + 2x_1x_2x_5^2 + x_1x_3x_4x_5 + 2x_1x_3x_5^2 + x_1x_4x_5^2
\\ & \hspace{60pt} + x_1x_5^3 + x_2x_3x_5^2 + x_2x_4x_5^2 + x_2x_5^3 + x_3x_4x_5^2 + x_3x_5^3).
\end{align*}
(We have factored the expressions for readability only.)
We verify Theorem~\ref{thm:merge} in this case by computing $\swt{13234} e_3(x_1, x_2, x_3, x_4, x_5) = \swt{13245} + \swt{14235}$.
\end{example}

By repeated applications of Theorem~\ref{thm:merge}, we obtain the following:

\begin{cor}
\label{cor:merges}
  Let $T$ be a finite set of positive integers.
  Let $\mm$ be a type such that every $t \in T$ satisfies
  $m_t \neq 0$ and $m_{t+1} \neq 0$.
  Let $v$ be a packed word of type $\bigvee_T \mm$.
  Then,
  \[
  \swt{v} \prod_{t \in T} e_{p_t(\mm)}(x_1, \dotsc, x_n) = \sum_u \swt{u},
  \]
where we sum over all $u$ of type $\mm$ such that $v = \bigvee_T u$.
\end{cor}

\begin{example}
Let $\mm = (1,1,1,1,1,0,0,\ldots)$ and $T = \{1,3\}$.
Then we have $\bigvee_T \mm = (2,2,1,\ldots)$, and we take $v = 12123$.
Then the words $u$ of type $\mm$ such that $v = \bigvee_T u = \vee_1 \vee_3 u$ are
\[
13245,
\qquad
14235,
\qquad
23145,
\qquad
24135.
\]
Note that $v = \vee_1 u'$ for $u' \in \{13234, 23134\}$, both of these possible $u'$ having type $(1, 1, 2, 1, 0, 0, \ldots)$.
From Theorem~\ref{thm:merge} (see also Example~\ref{ex:checking_merge_thm}), we have
\begin{align*}
\swt{12223} e_1(x_1, x_2, x_3, x_4, x_5) & = \swt{13234} + \swt{23134},
\\
\swt{13234} e_3(x_1, x_2, x_3, x_4, x_5) & = \swt{13245} + \swt{14235},
\\
\swt{23134} e_3(x_1, x_2, x_3, x_4, x_5) & = \swt{23145} + \swt{24135}.
\end{align*}
Hence, we verify Corollary~\ref{cor:merges} in this case since
\begin{gather*}
\swt{13245} + \swt{14235} + \swt{23145} + \swt{24135}
 = (\swt{13234} + \swt{23134}) e_3(x_1, x_2, x_3, x_4, x_5)
\\ = \swt{12223} e_3(x_1, x_2, x_3, x_4, x_5) e_1(x_1, x_2, x_3, x_4, x_5).
\end{gather*}
\end{example}

\subsection{A Jacobi-Trudi-like formula for special \texorpdfstring{$u$}{u}}
\label{subsec:JT_formula}

Throughout this subsection, we shall regard the sites $1, 2, \ldots, n$ as elements of $\set{1, 2, \ldots, n}$ rather than of $\ZZ/n\ZZ$.
In particular, they are totally ordered by $1 < 2 < \cdots < n$.

Let $\ell$ and $r$ be positive integers.
Let $B = \set{ b_1 < b_2 < \cdots < b_r} \subseteq \ive{n}$.
Let $v = \tup{v_1, v_2, \ldots, v_r}$ be an $r$-tuple of elements of $\ive{\ell-1}$.
Define a word $u(v) \in \mcW_n$ with $\ell$ classes by
\begin{equation}
\label{eq:weakly_decreasing_construction}
\bigl( u(v) \bigr)_i = \begin{cases}
v_j & \text{if $i = b_j$ for some $j$}, \\
\ell & \text{otherwise}.
\end{cases}
\end{equation}

We shall now state a determinantal formula for the special case of $\swt{u(v)}$
when $v$ is a weakly decreasing $r$-tuple whose entries cover $\ive{\ell-1}$:
%For example, $u = {\color{gray}4}3{\color{gray}44}33{\color{gray}4}22{\color{gray}4}211{\color{gray}4}$ is such a word, with $\ell=4$.

\begin{thm}
\label{thm:determinant_form}
Let $\ell$ be a positive integer.
Let $B = \set{ b_1 < b_2 < \cdots < b_r} \subseteq \ive{n}$.
Let $\tup{v_1 \geq v_2 \geq \cdots \geq v_r}$ be a weakly decreasing $r$-tuple of integers such that we have $\set{v_1, v_2, \ldots, v_r} = \ive{\ell-1}$.
For each $j \in \ive{r}$, set $\gamma_j = \ell - v_j$.
(The value $\gamma_j$ can also be described as the number of distinct letters among $v_1, v_2, \dotsc, v_j$.)
Then
\[
\swt{u(v)} = \left(  \prod_{b \in B} x_b \right)  \det \left(  h_{\gamma_j+i-j-1}(x_1,\dotsc,x_{b_j}) \right)_{i, j \in \ive{r}}.
\]
\end{thm}

The proof of this theorem is given in Section~\ref{sec:JT-proofs}.

%%%%%%%%%%
\subsection{Conclusions for Aas--Linusson MLQs}

We can use Theorem~\ref{thm:determinant_form} to settle two conjectures from~\cite{AasLin17}.

Fix a set of sites $B = \set{b_1 < b_2 < \cdots < b_r}$.
We say that a set $S$ of integers is \defn{lacunar} if $i\in S$ implies $i+1 \notin S$.
Let $S \subseteq \ive{r-1}$ be lacunar, and define the permutation $\sigma_S := \left( \prod_{i \in S} s_i \right) w_0$, where $s_i, w_0 \in \SymGp{r}$.
Note that the elements $\{s_i \mid i \in S\}$ all commute, so their product, and hence $\sigma_S$, is well-defined.
In one-line notation, $\sigma_S$ is the list of all elements of $\ive{r}$ in decreasing order, except that for each $i \in S$, the pair $\tup{i, i+1}$ is sorted back into increasing order.

For each permutation $\tau \in \SymGp{r}$, we let $v_\tau$ be the $r$-tuple $\tup{\tau_1, \tau_2, \ldots, \tau_r}$ (that is, the one-line notation of $\tau$).
To ease notation, we shall identify any $\tau \in \SymGp{r}$ with the corresponding word $u(v_{\tau}) \in \mcW_n$ defined by~\eqref{eq:weakly_decreasing_construction}, where we set $\ell=r+1$.

In~\cite[Conj.~3.10]{AasLin17}, a formula for the spectral weight $\swt{\sigma_S}$ at $x_1 = \cdots = x_n = 1$ is conjectured.
We now prove a version of this conjecture for general $x_i$.

\begin{cor}
\label{cor:special_weight_lacunar}
Let $B = \set{b_1 < b_2 < \cdots < b_r} \subseteq \ive{n}$.
Let $T \subseteq \ive{r-1}$ be lacunar, and let $\psi(T) = \sum_{S \subseteq T} \swt{\sigma_S}$.
Then we have
\begin{align*}
  \psi(T) & = \left(\prod_{t \in T} e_t(x_1, \dotsc, x_n) \right) \swt{ \bigvee_T w_0}
  \\ & = \left( \prod_{b \in B} x_b \right) \left(\prod_{t \in T} e_t(x_1, \dotsc, x_n) \right) \det\bigl(h_{\gamma_j+i-j-1}(x_1, \dotsc, x_{b_j})\bigr)_{i, j \in \ive{r}},
\end{align*}
where $\gamma_j = j - \abs{ \{t \in T \mid t > r - j \} }$.
\end{cor}

\begin{proof}
Recall that we identify the permutation $w_0 \in \SymGp{r}$ with the word $u(v_{w_0}) \in \mcW_n$;
the latter word has type
$\mm := \tup{1,1,\ldots,1,n-r,0,0,\ldots}$ (starting with $r$
ones).
Hence, the word $\bigvee_T w_0$ has type $\bigvee_T \mm$.
By Corollary~\ref{cor:merges} (applied to $v = \bigvee_T w_0$), we thus have
\[
\left(\prod_{t \in T} e_t(x_1, \dotsc, x_n) \right) \swt{ \bigvee_T w_0} = \sum \swt{u'} ,
\]
where the sum ranges over all words $u' \in \mcW_n$
of type $\mm$ satisfying
$\bigvee_T u' = \bigvee_T w_0$.
But the words $u' \in \mcW_n$ of type $\mm$ satisfying
$\bigvee_T u' = \bigvee_T w_0$ are precisely the words of the form $\sigma_S$
\begin{verlong}
(recall that this stands for $u(v_{\sigma_S})$)
\end{verlong}
for $S \subseteq T$.
Hence, the $\sum \swt{u'}$ in the above equality rewrites as $\psi(T)$.
This proves the first equality of the corollary.
To prove the second equality, we consider $\ell = r - \abs{T} + 1$ and
the $r$-tuple
\[
v = \bigl(v_j = \ell - j + \abs{\set{t \in T \mid t > r-j}} = \ell - \gamma_j \bigr)_{j=1}^r.
\]
Note that $v_1 = r - \abs{T} = \ell-1$ and $v_r = 1$ and
\[
v_{i+1} = v_i - \begin{cases} 0 & \text{if } r-i \in T, \\ 1 & \text{if } r-i \notin T, \end{cases}
\qquad
\text{for each $i \in \ive{r-1}$.}
\]
Thus, $v$ is weakly decreasing and $\set{v_1, v_2, \ldots, v_r} = \ive{\ell-1}$ and $u(v) = \bigvee_T w_0$.
Hence, the second equality follows from Theorem~\ref{thm:determinant_form}.
\end{proof}

\begin{remark}
We note that $\bigvee_T w_0$ considered in $\mcW_r$ is a weakly decreasing packed word whose descents are in $[r-1] \setminus T$.
When we lift $\bigvee_T w_0$ to $\mcW_n$ under the identification with $\bigvee_T u(v_{w_0}) \in \mcW_n$, we add in the values $r-\abs{T}+1$ at the positions specified by $B$.
\end{remark}

\begin{example}
Consider $n = 8$, $B = \set{1,2,5,6,8}$, and $T = \set{1,4}$.
Then we have $\bigvee_T w_0 = 33{\color{gray}44}21{\color{gray}4}1$, and
\[
\psi(T) = \swt{54{\color{gray}66}32{\color{gray}6}1} + \swt{45{\color{gray}66}32{\color{gray}6}1} + \swt{54{\color{gray}66}31{\color{gray}6}2} + \swt{45{\color{gray}66}31{\color{gray}6}2}.
\]
Applying Theorem~\ref{thm:merge} to $t = 4$ twice and then to $t=1$, we obtain
\begin{align*}
\psi(T) & = e_4(x_1, \dotsc, x_8) \, \bigl( \swt{44{\color{gray}55}32{\color{gray}5}1} + \swt{44{\color{gray}55}31{\color{gray}5}2} \bigr)
\\ & = e_1(x_1, \dotsc, x_8) \, e_4(x_1, \dotsc, x_8) \swt{33{\color{gray}44}21{\color{gray}4}1}.
\end{align*}
Using Theorem~\ref{thm:determinant_form}, this further becomes
\begin{align*}
\psi(T) & = x_1 x_2 x_5 x_6 x_8 \, e_1(x_1, \dotsc, x_8) \, e_4(x_1, \dotsc, x_8)
\\ & \hspace{20pt} \times \det \begin{pmatrix}
1 & 0 & 0 & 0 & 0 \\
h_1[1] & 1 & 1 & 1 & 0 \\
h_2[1] & h_1[2] & h_1[5] & h_1[6] & 1 \\
h_3[1] & h_2[2] & h_2[5] & h_2[6] & h_1[8] \\
h_4[1] & h_3[2] & h_3[5] & h_3[6] & h_2[8] \\
\end{pmatrix},
\end{align*}
where $h_i[k] := h_i(x_1, \dotsc, x_k)$.
\end{example}

Note that the well-known formula for M\"obius inversion on the boolean lattice yields
\[
\swt{\sigma_S} = \sum_{T\subseteq S} (-1)^{|S|-|T|} \psi(T)
\]
for any $S \subseteq \ive{r-1}$.
Therefore, by specializing Corollary~\ref{cor:special_weight_lacunar} at $x_1 = \cdots = x_n = 1$, we obtain~\cite[Conj.~3.10]{AasLin17} (which is a generalization of~\cite[Conj.~3.9]{AasLin17}).
Additionally, this proves~\cite[Conj.~3.6]{AasLin17} (which is a generalization of~\cite[Conj.~3.4]{AasLin17}).

%=====================================================================
\section{The TASEP connection}
\label{sec:tasep}

We now explain how our proof of Theorem~\ref{thm:permutation} gives a proof of the commutativity conjecture of~\cite{AAMP}.

There are $2^{n-1}$ packed types for words of length $n$, as they are the compositions of $n$ (see \cite[Section 1.2]{Stanley-EC1}).
We let the subset $S \subseteq [n-1]$ correspond to the type of the word obtained by merging $i$ and $i+1$ in $12 \dotsm n$ for each $i \in S$; \textit{i.e.}, $\bigvee_S \mm$ for $\mm = \tup{1, \dotsc, 1, 0, \ldots}$ with $n$ occurrences of $1$.
Denote this type by $\mm_S$.
Note that $\set{p_1(\mm_S), \dotsc, p_{\ell-1}(\mm_S)} = [n-1] \setminus S$, where $\mm_S$ has $\ell$ classes; this is the complement of the usual bijection between subsets of $\ive{n-1}$ and compositions of $n$.
Let $\mcW_S$ denote the set of words of type $\mm_S$.
Let $V_S$ be the vector space over $\RR$ with basis $\set{\epsilon_w \mid w \in \mcW_S}$.

\begin{example}
  For $n = 4$, we have
  
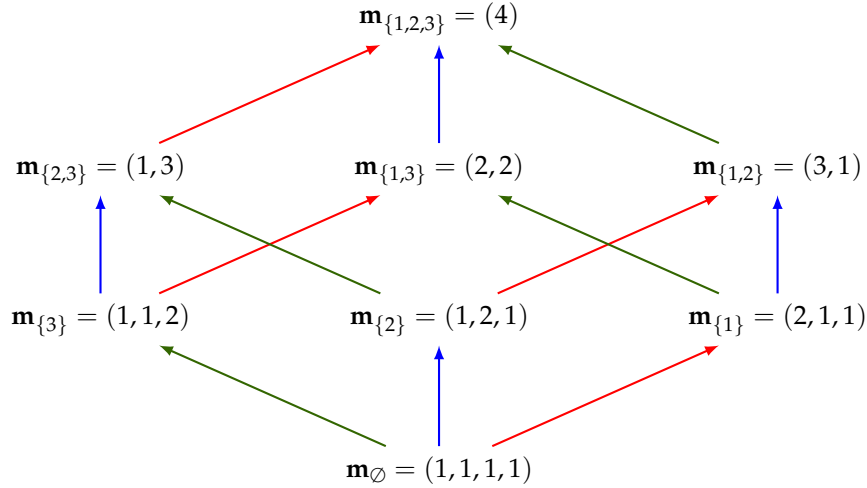
\[
  \begin{tikzpicture}[xscale=4.5,yscale=2,thick,>=latex]
  \node (E) at (1,0) {$\mbf{m}_{\emptyset} = (1,1,1,1)$};
  \node (1) at (2,1) {$\mbf{m}_{\{1\}} = (2,1,1)$};
  \node (2) at (1,1) {$\mbf{m}_{\{2\}} = (1,2,1)$};
  \node (3) at (0,1) {$\mbf{m}_{\{3\}} = (1,1,2)$};
  \node (12) at (2,2) {$\mbf{m}_{\{1,2\}} = (3,1)$};
  \node (13) at (1,2) {$\mbf{m}_{\{1,3\}} = (2,2)$};
  \node (23) at (0,2) {$\mbf{m}_{\{2,3\}} = (1,3)$};
  \node (123) at (1,3) {$\mbf{m}_{\{1,2,3\}} = (4)$};
  \draw[->,red] (E) -- (1);
  \draw[->,blue] (E) -- (2);
  \draw[->,dgreencolor] (E) -- (3);
  \draw[->,red] (2) -- (12);
  \draw[->,red] (3) -- (13);
  \draw[->,blue] (3) -- (23);
  \draw[->,blue] (1) -- (12);
  \draw[->,dgreencolor] (1) -- (13);
  \draw[->,dgreencolor] (2) -- (23);
  \draw[->,red] (23) -- (123);
  \draw[->,blue] (13) -- (123);
  \draw[->,dgreencolor] (12) -- (123);
  \end{tikzpicture}
  \]
  where we have drawn an arrow $\mm_S \to \mm_{S \cup \set{i}}$ for each $S \subseteq [n-1]$ and each $i \in [n-1] \setminus S$ (this is the Hasse diagram of the Boolean lattice of subsets of $[n-1]$).
  As an example of the vector space $V_S$, we have $V_{\{2,3\}} = \RR \{ \epsilon_{1222}, \epsilon_{2122}, \epsilon_{2212}, \epsilon_{2221} \}$.
\end{example}

The \defn{totally asymmetric simple exclusion process} (TASEP) is a Markov chain on $\mcW_S$, where $S \subseteq[n-1]$, as follows.
For a state $u \in \mcW_S$, we move to a new state by picking a random $i \in [n]$ and either
\begin{itemize}
\item if $u_i > u_{i+1}$, swap the positions $u_i$ and $u_{i+1}$, or
\item do nothing (\textit{i.e.} stay at $u$).
\end{itemize}
Let $M_S \colon V_S \to V_S$ be the transition matrix of this Markov chain.
Note that these moves preserve the type of the words; thus we could consider this as a Markov chain on $\mcW_n$, where $\mcW_S$ becomes an irreducible component.
For $i \notin S$, we have the merging map $\Phi_i \colon \mcW_S \to \mcW_{S\cup\{i\}}$ given by
$\Phi_i(\epsilon_u) = \epsilon_{\vee^{(i)} u}$.
%$\Phi_i(\epsilon_u) = \epsilon_{\merge{t}u}$, where $t = \min\set{k \mid p_k(\mm_S) \geq i}$.
It is easy to see that $\Phi_i M_S = M_{S\cup \{i\}} \Phi_i$. 

\begin{figure}
\[
\begin{tikzpicture}[>=stealth,thick,scale=0.8]
\node (321) at (3,8) {$321$};
\node (231) at (0,6) {$231$};
\node (312) at (6,6) {$312$};
\node (213) at (0,2) {$213$};
\node (132) at (6,2) {$132$};
\node (123) at (3,0) {$123$};
\draw[->,blue] (123) -- (321); % node[right,pos=0.8] {$1/3$};
\draw[->,dgreencolor] (132) -- (123); % node[above=3pt,pos=0.5] {$1/3$};
\draw[->,red] (132) -- (231); % node[below=4pt,pos=0.8] {$1/3$};
\draw[->,red] (213) -- (123); % node[above=3pt,pos=0.5] {$1/3$};
\draw[->,dgreencolor] (213) -- (312); % node[below=4pt,pos=0.8] {$1/3$};
\draw[->,blue] (231) -- (213); % node[left,pos=0.5] {$1/3$};
\draw[->,blue] (312) -- (132); % node[right,pos=0.5] {$1/3$};
\draw[->,dgreencolor] (321) -- (231); % node[above=3pt,pos=0.5] {$1/3$};
\draw[->,red] (321) -- (312); % node[above=3pt,pos=0.5] {$1/3$};
\end{tikzpicture}
\hspace{35pt}
\begin{tikzpicture}[>=stealth,thick,scale=0.8]
\node (1223) at (0,-2) {$1223$};
\node (3221) at (0,8) {$3221$};
\node (2321) at (2,6) {$2321$};
\node (3212) at (-2,6) {$3212$};
\node (2231) at (3,4) {$2231$};
\node (2312) at (0,4) {$2312$};
\node (3122) at (-3,4) {$3122$};
\node (2213) at (3,2) {$2213$};
\node (2132) at (0,2) {$2132$};
\node (1322) at (-3,2) {$1322$};
\node (1232) at (-2,0) {$1232$};
\node (2123) at (2,0) {$2123$};
\draw[->,blue] (1223) .. controls (2,3) and (1,6) .. (3221);
\draw[->,red] (3221) -- (3212);
\draw[->,dgreencolor] (3221) -- (2321);
\draw[->,dgreencolor] (2321) -- (2231);
\draw[->,red] (2321) -- (2312);
\draw[->,red] (3212) -- (3122);
\draw[->,dgreencolor] (3212) -- (2312);
\draw[->,blue] (3122) -- (1322);
\draw[->,blue] (2312) -- (2132);
\draw[->,blue] (2231) -- (2213);
\draw[->,dgreencolor] (1322) -- (1232);
\draw[->,red] (1322)  .. controls (-1,3) and (-1,6) .. (2321);
\draw[->,dgreencolor] (2132) -- (2123);
\draw[->,red] (2132) -- (1232);
\draw[->,dgreencolor] (2213) .. controls (1,3) and (1,6) .. (3212);
\draw[->,red] (2213) -- (2123);
\draw[->,dgreencolor] (1232) -- (1223);
\draw[->,red] (1232) .. controls (1,0) and (1,3) .. (2231);
\draw[->,red] (2123) -- (1223);
\draw[->,dgreencolor] (2123) .. controls (-1,0) and (-1,3) .. (3122);
\end{tikzpicture}
\]
\caption{The states and transitions for $\mcW_{\emptyset}$ for $n = 3$ (left) and $\mcW_{\set{2}}$ for $n = 4$ (right).
All probabilities of the drawn transitions are $1/n$.}
\end{figure}
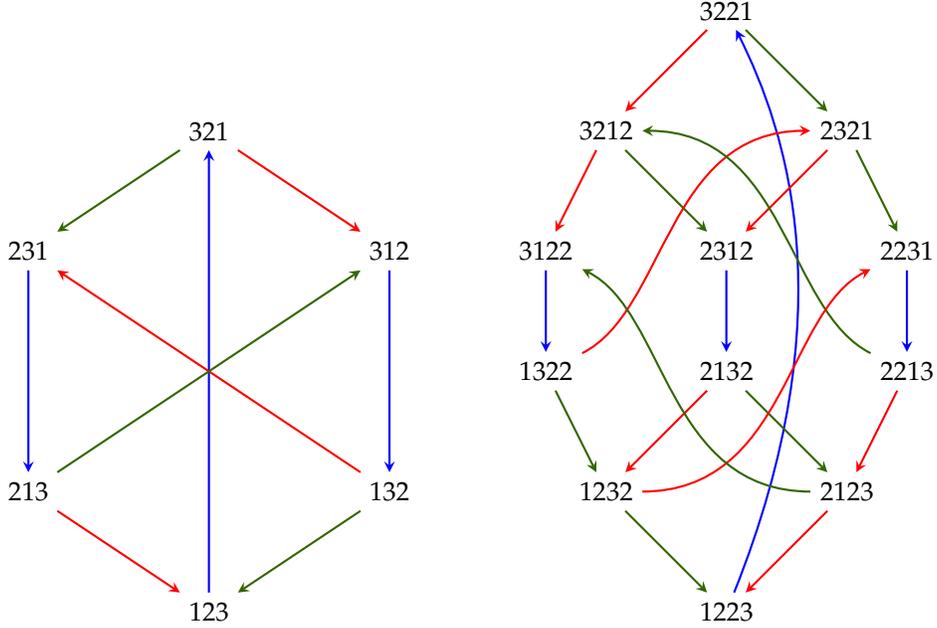

Building on work by Ferrari and Martin~\cite{FM06,FM07}, the paper~\cite{AAMP} introduced opposite operators $\Psi_i \colon \mcW_S \to \mcW_{S \setminus \{i\}}$ given by $\Psi_i(\epsilon_u) = \sum_{q} \epsilon_{q(u)}$, where the sum is taken over all $i$-queues $q$, and showed that $\Psi_i M_S = M_{S \setminus \{i\}} \Psi_i$.
Furthermore they proposed the \defn{commutativity conjecture}: that $\Psi_i \Psi_j = \Psi_j \Psi_i$.
By looking at the $(u,v)$ entry of both sides of this equation, the commutativity conjecture is asking whether the number of $(i,j)$-configurations $C$ such that $v = C(u)$ equals the number of $(j,i)$-configurations $C'$ such that $v = C'(u)$. 
Thus, our proof of Theorem~\ref{thm:permutation} shows that $\widetilde{\Psi}_i \widetilde{\Psi}_j = \widetilde{\Psi}_j \widetilde{\Psi}_i$ for the weighted operators $\widetilde{\Psi}_i$ given by
\[
\widetilde{\Psi}_i(\epsilon_u) = \sum_q \wt(q) \epsilon_{q(u)},
\]
where we also sum over all $i$-queues $q$.
Note that $\widetilde{\Psi}_i = \Psi_i$ when we specialize $x_1 = \cdots = x_n = 1$, giving the connection between our MLQs and the multi-species TASEP.
We note that the proof of interlacing given in~\cite{AAMP} is significantly different from our approach.

We have not managed to find a process similar to the TASEP whose transition matrix $\widetilde{M}_S$ would satisfy $\widetilde{M}_S \widetilde{\Psi}_i = \widetilde{\Psi}_i \widetilde{M}_{S \setminus \set{i}}$ for our $\widetilde{\Psi}_i$ operators.
Note however that queues give us \emph{a} random process with this property: for a word $u \in \mcW_S$, a move in the chain is given by
\begin{enumerate}
\item picking a random $i$-queue $q$
\item going to the state $\merge{t} q(u) \in \mcW_S$, where $t = \min\set{k \mid p_k(\mm_S) \geq i}$.
\end{enumerate}

%=====================================================================

\section{Proof of Theorem~\ref{thm:determinant_form}}
\label{sec:JT-proofs}

The goal of this section is to prove Theorem~\ref{thm:determinant_form}.
Along the way, we shall derive a number of intermediate results, some of which may be of independent interest.

As in Subsection~\ref{subsec:JT_formula}, we shall consider the sites as elements of the totally ordered set $\set{1, 2, \ldots, n}$ (ordered by $1 < 2 < \cdots < n$) throughout this section.
Moreover, we shall use infinitely many distinct commuting indeterminates $\ldots, x_{-2}, x_{-1}, x_0, x_1, x_2, \ldots$ instead of those indexed by $\ZZ/n\ZZ$; thus we do not have $x_{n+k} = x_k$ in this section.

%%%%%%%%%%
\subsection{Lattice paths and the Lindstr\"{o}m--Gessel--Viennot theorem}

Our arguments will rely on the \defn{Lindstr\"om--Gessel--Viennot (LGV) Lemma}~\cite{GV85,Lindstrom73} and on a re-interpretation of MLQs as a certain kind of semistandard tableaux (of non-partition shape).
This takes inspiration from the ``bully paths'' of~\cite{AasLin17} as well as from the standard proof of the Jacobi--Trudi identities for Schur functions~\cite[First proof of Theorem 7.16.1]{Stanley-EC2}.
We begin with basic definitions.

The \defn{lattice} shall mean the (infinite) directed graph whose vertices are
all pairs of integers (that is, its vertex set is $\ZZ^2$), and whose arcs are
\begin{align*}
(i,j) & \to (i,j+1) & & \text{for all } (i,j) \in \ZZ^2, \qquad \text{and} \\
(i,j) & \to (i+1,j) & & \text{for all } (i,j) \in \ZZ^2.
\end{align*}
The arcs of the first kind are called \defn{north-steps}, whereas the arcs of
the second kind are called \defn{east-steps}.
\begin{verlong}
The vertices of the lattice will just be called \defn{vertices}.
\end{verlong}
We consider the lattice as the usual integer lattice in the Cartesian plane.

For each vertex $v = (i,j) \in \ZZ^2$, we set $\xcoord(v) = i$ and $\ycoord(v) = j$.
We refer to $\xcoord(v)$ (resp.~$\ycoord(v)$) as the \defn{$x$-coordinate} (resp.~\defn{$y$-coordinate}) of $v$.
The \defn{$y$-coordinate} of an east-step $(i,j) \to (i+1, j)$ is defined to be $j$.

For each arc $a$ of the lattice, we define the \defn{weight} of $a$ as the monomial
\[
\wt(a) :=
\begin{cases}
x_j & \text{if $a$ is an east-step } (i,j) \to (i+1,j), \\
1 & \text{if $a$ is a north-step } (i,j) \to (i,j+1).
\end{cases}
\]
Thus, all north-steps have weight $1$, while east-steps with $y$-coordinate $j$ have weight $x_j$.

Fix $k \in \NN$.
A $k$-tuple of vertices of the lattice will be called a \defn{$k$-vertex}.
\begin{verlong}
If $\vv = \tup{A_1, A_2, \dotsc, A_k}$ is a $k$-vertex, and if $\sigma \in \SymGp{k}$ is a permutation, then $\sigma(\vv)$ denotes the $k$-vertex $\tup{A_{\sigma(1)}, A_{\sigma(2)}, \dotsc, A_{\sigma(k)}}$.
\end{verlong}
A \defn{path} simply means a (directed) path in the lattice.
The \defn{weight} of a path $p$, denoted $\wt(p)$, is defined as the product of the weights of all arcs of this path; this weight is a monomial.
If $A$ and $B$ are two vertices, then \defn{$N(A,B)$} shall denote the set of all paths from $A$ to $B$.

It is easy to see (see, \textit{e.g.}, \cite[(2.36)]{Stanley-EC1}) that any two vertices $A = (a,b)$ and $B = (c,d)$ satisfy
\begin{equation}
\label{eq.LGV.single-paths}
\sum_{p \in N(A,B)} \wt(p) = h_{c-a}(x_{b}, x_{b+1}, \dotsc, x_{d}).
\end{equation}

\begin{comment}
We begin with a proof of~\eqref{eq.LGV.single-paths}, which is a simple and classical fact that appears implicitly in various texts on enumerative combinatorics (\textit{e.g.}, \cite[(2.36)]{Stanley-EC1}).

\begin{proof}[Proof of~\eqref{eq.LGV.single-paths}.]
Let $A = (a,b)$ and $B = (c,d)$ be two vertices.
If $p$ is a path from $A$ to $B$, then $p$ must have exactly $c-a$ east-steps, and the $y$-coordinates of these
east-steps must belong to the interval $\set{b,b+1,\dotsc,d}$.
Let $\bigl( y_1(p), y_2(p), \dotsc, y_{c-a}(p) \bigr) \in \set{b,b+1, \dotsc, d}^{c-a}$ be the weakly increasing $(c-a)$-tuple consisting of the $y$-coordinates of the $c-a$ east-steps of $p$ (from left to right).
Moreover, $p$ can be uniquely reconstructed from this $(c-a)$-tuple (since this $(c-a)$-tuple determines the east-steps of $p$).
Conversely, any weakly increasing $(c-a)$-tuple of elements of $\set{b, b+1, \dotsc,d}$ has the form $\bigl( y_1(p), y_2(p), \dotsc, y_{c-a}(p) \bigr)$ for a unique path $p$ from $A$ to $B$.
Thus, there is a bijection between the paths $p$ from $A$ to $B$ and the weakly increasing $(c-a)$-tuples of elements of $\set{b, b+1, \dotsc, d}$.
This yields
\begin{align*}
\sum_{p \in N(A,B)} \wt(p) & = \sum_{\substack{ \tup{k_1,k_2,\ldots,k_{c-a}} \text{ is a weakly increasing} \\ (c-a)\text{-tuple of elements of }\set{b,b+1,\dotsc,d} }} x_{k_1} x_{k_2} \cdots x_{k_{c-a}}
\\ & = h_{c-a}(x_{b}, x_{b+1}, \dotsc, x_{d}).
\end{align*}
\end{proof}
\end{comment}

If $\tup{A_1, A_2, \dotsc, A_k}$ and $\tup{B_1, B_2, \dotsc, B_k}$ are two $k$-vertices, then a \defn{non-intersecting lattice path tuple (NILP)} from $\tup{A_1, A_2, \dotsc, A_k}$ to $\tup{B_1, B_2, \dotsc, B_k}$ shall mean a $k$-tuple $\tup{p_1, p_2, \dotsc, p_k}$ of paths such that
\begin{itemize}
\item each $p_i$ is a path from $A_i$ to $B_i$;
\item no two of the paths $p_1, p_2, \dotsc, p_k$ have a vertex in common.
\end{itemize}
(Visually speaking, the paths must neither cross nor touch.)

The \defn{weight} of a NILP $\pp = \tup{p_1, p_2, \dotsc, p_k}$ is the monomial $\wt(\pp)$ defined by
\[
\wt(\pp) := \prod_{i=1}^{k} \wt(p_i).
\]
See Figure~\ref{fig:NILP_example} for an illustration.

If $\uu$ and $\vv$ are two $k$-vertices, then \defn{$N(\uu,\vv)$} denotes the set of all NILPs from $\uu$ to $\vv$.

\begin{figure}[t]
\[
\begin{tikzpicture}
  \draw[densely dotted] (0,0) grid (7.2,7.2);
  % axes:
  \draw[->] (0,0) -- (0,7.2);
  \draw[->] (0,0) -- (7.2,0);
  \foreach \x/\xtext in {0, 1, 2, 3, 4, 5, 6, 7}
     \draw (\x cm,1pt) -- (\x cm,-1pt) node[anchor=north] {$\xtext$};
  \foreach \y/\ytext in {0, 1, 2, 3, 4, 5, 6, 7}
     \draw (1pt,\y cm) -- (-1pt,\y cm) node[anchor=east] {$\ytext$};

  \node[circle,fill=white,draw=black,text=UMNmaroon,inner sep=1pt] (A3) at (2,3) {$A_3$};
  \node[circle,fill=white,draw=black,text=UMNmaroon,inner sep=1pt] (A2) at (3,1) {$A_2$};
  \node[circle,fill=white,draw=black,text=UMNmaroon,inner sep=1pt] (A1) at (5,1) {$A_1$};

  \node[circle,fill=white,draw=black,text=black,inner sep=1pt] (B3) at (4,6) {$B_3$};
  \node[circle,fill=white,draw=black,text=black,inner sep=1pt] (B2) at (5,6) {$B_2$};
  \node[circle,fill=white,draw=black,text=black,inner sep=1pt] (B1) at (6,4) {$B_1$};

  \begin{scope}[thick,>=stealth,darkred]
      % $p_3$:
      \draw (A3) edge[->] (2,4);
      \draw (2,3.6) node[anchor=east] {$1$};
      \draw (2,4) edge[->] (3,4);
      \draw (2.5,4) node[anchor=north] {$x_4$};
      \draw (3,4) edge[->] (3,5);
      \draw (3,4.5) node[anchor=east] {$1$};
      \draw (3,5) edge[->] (4,5);
      \draw (3.5,5) node[anchor=north] {$x_5$};
      \draw (4,5) edge[->] (B3);
      \draw (4,5.4) node[anchor=east] {$1$};
  \end{scope}
  \begin{scope}[thick,>=stealth,dbluecolor]
      % $p_2$:
      \draw (A2) edge[->] (4,1);
      \draw (3.6,1) node[anchor=north] {$x_1$};
      \draw (4,1) edge[->] (4,2);
      \draw (4,1.5) node[anchor=east] {$1$};
      \draw (4,2) edge[->] (4,3);
      \draw (4,2.5) node[anchor=east] {$1$};
      \draw (4,3) edge[->] (4,4);
      \draw (4,3.5) node[anchor=west] {$1$};
      \draw (4,4) edge[->] (5,4);
      \draw (4.5,4) node[anchor=south] {$x_4$};
      \draw (5,4) edge[->] (5,5);
      \draw (5,4.5) node[anchor=west] {$1$};
      \draw (5,5) edge[->] (B2);
      \draw (5,5.4) node[anchor=west] {$1$};
  \end{scope}
  \begin{scope}[thick,>=stealth,dgreencolor]
      % $p_1$:
      \draw (A1) edge[->] (5,2);
      \draw (5,1.6) node[anchor=east] {$1$};
      \draw (5,2) edge[->] (6,2);
      \draw (5.5,2) node[anchor=south] {$x_2$};
      \draw (6,2) edge[->] (6,3);
      \draw (6,2.5) node[anchor=west] {$1$};
      \draw (6,3) edge[->] (B1);
      \draw (6,3.4) node[anchor=west] {$1$};
  \end{scope}
\end{tikzpicture}
\]
\caption{A NILP from the $3$-vertex $\tup{A_1, A_2, A_3}$ to the $3$-vertex $\tup{B_1, B_2, B_3}$ of weight $x_1 x_2 x_4^2 x_5$.
  The weights of the edges of the paths are written next to the edges.}
\label{fig:NILP_example}
\end{figure}
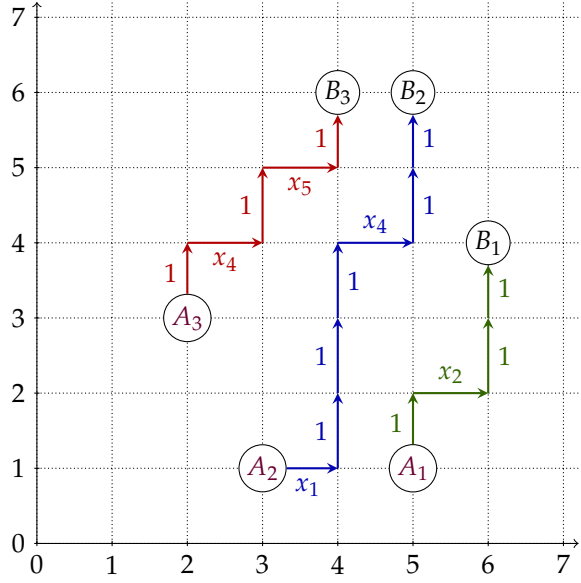

\begin{verlonglong}
A pair $(\uu, \vv)$ of two $k$-vertices $\uu$ and $\vv$ is said to be \defn{nonpermutable} if and only
if every permutation $\sigma \neq \id$ in $\SymGp{k}$ satisfies $N\bigl( \uu,\sigma(\vv) \bigr) = \emptyset$.
Note that we are not requiring that $N(\uu, \vv) \neq \emptyset$ here.

The following fact (a particular case of~\cite[Corollary 2]{GesVie89}) is crucial:

\begin{prop}
\label{prop.LGV.nonper}
Let $k \in \NN$.
Let $(\uu, \vv)$ be a nonpermutable pair of two $k$-vertices $\uu = \tup{A_1, A_2, \dotsc, A_k}$ and $\vv = \tup{B_1, B_2, \dotsc, B_k}$.
Then,
\[
\sum_{\pp \in N(\uu,\vv)} \wt(\pp) = \det\left( \sum_{p \in N(A_i,B_j)} \wt(p) \right)_{i, j \in \ive{k}}.
\]
\end{prop}

Next, we state a simple lemma that will help us show that certain pairs of $k$-vertices are nonpermutable:

\begin{lemma}
\label{lem.LGV.hex}
Let $A$, $B$, $A'$ and $B'$ be four vertices of the lattice such that
\[
\xcoord(A') \leq \xcoord(A), \qquad \ycoord(A') \geq \ycoord(A), \qquad
\xcoord(B') \leq \xcoord(B), \qquad \ycoord(B') \geq \ycoord(B).
\]
Let $p$ be a path from $A$ to $B'$. Let $p'$ be a path from $A'$ to $B$.
Then, $p$ and $p'$ have a vertex in common.
\end{lemma}

\begin{figure}[t]
\[
\begin{tikzpicture}
  \draw[densely dotted] (0,0) grid (7.2,7.2);
  % axes:
  \draw[->] (0,0) -- (0,7.2);
  \draw[->] (0,0) -- (7.2,0);
  \foreach \x/\xtext in {0, 1, 2, 3, 4, 5, 6, 7}
     \draw (\x cm,1pt) -- (\x cm,-1pt) node[anchor=north] {$\xtext$};
  \foreach \y/\ytext in {0, 1, 2, 3, 4, 5, 6, 7}
     \draw (1pt,\y cm) -- (-1pt,\y cm) node[anchor=east] {$\ytext$};

  \node[circle,fill=white,draw=black,text=UMNmaroon,inner sep=1pt] (A') at (1,3) {$A^{\prime}$};
  \node[circle,fill=white,draw=black,text=UMNmaroon,inner sep=1pt] (A) at (2,1) {$A$};
  
  \node[circle,fill=white,draw=black,text=black,inner sep=1pt] (B') at (5,6) {$B^{\prime}$};
  \node[circle,fill=white,draw=black,text=black,inner sep=1pt] (B) at (6,4) {$B$};
  
  \begin{scope}[thick,>=stealth,darkred]
      % $p$:
      \draw (A) edge[->] (3,1);
      \draw (2.6,1) node[anchor=north] {$p$};
      \draw (3,1) edge[->] (3,2);
      \draw (3,2) edge[->] (4,2);
      \draw (4,2) edge[->] (4,3);
      \draw (4,3) edge[->] (4,4);
      \draw (4,4) edge[->] (4,5);
      \draw (4,5) edge[->] (5,5);
      \draw (5,5) edge[->] (B');
      \draw (5,5.4) node[anchor=east] {$p$};
  \end{scope}
  \begin{scope}[thick,>=stealth,dbluecolor]
      % $p^{\prime}$:
      \draw (A') edge[->] (2,3);
      \draw (1.6,3) node[anchor=north] {$p^{\prime}$};
      \draw (2,3) edge[->] (2,4);
      \draw (2,4) edge[->] (3,4);
      \draw (3,4) edge[->] (4,4);
      \draw (4,4) edge[->] (5,4);
      \draw (5,4) edge[->] (B);
      \draw (5.4,4) node[anchor=north] {$p^{\prime}$};
  \end{scope}
\end{tikzpicture}
\]
\caption{Illustration for Lemma~\ref{lem.LGV.hex}.}
\label{fig:LGV.hex.example}
\end{figure}

See Figure~\ref{fig:LGV.hex.example} for an illustration of the situation of this lemma.

\begin{proof}[First proof of Lemma~\ref{lem.LGV.hex}.]
If $q$ is any path, then the \defn{length $\ell\left(q\right)$} of $q$ is defined to be the number of arcs
of $q$.

We shall now prove Lemma~\ref{lem.LGV.hex} by strong induction on $\ell(p) + \ell(p')$:

\textit{Induction step:}
Fix $N \in \NN$. Assume (as the induction hypothesis) that Lemma~\ref{lem.LGV.hex} holds whenever $\ell(p) + \ell(p') < N$.
We must now prove that Lemma~\ref{lem.LGV.hex} holds when $\ell(p) + \ell(p') = N$.

So let $A$, $B$, $A'$, $B'$, $p$ and $p'$ be as in Lemma~\ref{lem.LGV.hex}, and let us assume that $\ell(p) + \ell(p') = N$.
We must prove that $p$ and $p'$ have a vertex in common.

Assume the contrary. Thus, $p$ and $p'$ have no vertex in common.

The vertex $A$ belongs to the path $p$, and thus does not belong to the path
$p'$ (since $p$ and $p'$ have no vertex in common). Similarly,
the vertex $A'$ does not belong to the path $p$.

Recall that each arc of the lattice is either an east-step or a north-step.
Thus, the $x$-coordinates of the vertices of a path are always weakly
increasing, and so are the $y$-coordinates. Hence, the existence of a path $p$
from $A$ to $B'$ shows that $\xcoord(A) \leq \xcoord(B')$ and $\ycoord(A) \leq \ycoord(B')$.
Similarly, the existence of a path $p'$ from $A'$ to $B$ yields $\xcoord(A') \leq \xcoord(B)$ and $\ycoord(A') \leq \ycoord(B)$.

Next, we claim that $\ell(p) \neq 0$.

[\textit{Proof:}
Assume the contrary. Thus, $\ell(p) = 0$.
Hence, the path $p$ has no steps.
Therefore, $A=B'$ (since $p$ is a path from $A$ to $B'$).
Hence, $\ycoord(A) = \ycoord(B') \geq \ycoord(B)$, so that $\ycoord(A') \geq \ycoord(A) \geq \ycoord(B)$.
Combining this with $\ycoord(A') \leq \ycoord(B)$, we obtain $\ycoord(A') = \ycoord(B)$.
Combining $\ycoord(A') \geq \ycoord(A)$ with $\ycoord(A') = \ycoord(B) \leq \ycoord(A)$, we obtain $\ycoord(A')  = \ycoord(A)$.
Thus, $\ycoord(A') = \ycoord(A) = \ycoord(B)$.
Therefore, the vertex $A$ lies on the horizontal line that contains $A'$ and $B$.
Furthermore, this vertex $A$ must lie on the line segment between $A'$ and $B$ (since $\xcoord(A') \leq \xcoord( \underbrace{A}_{=B'} ) = \xcoord(B') \leq \xcoord(B)$).

Recall again that the $y$-coordinates of the vertices a path are always weakly increasing.
Moreover, they increase strictly whenever the path makes a north-step.
Hence, if the path $p'$ would have any north-step, then we would have $\ycoord(A') < \ycoord(B)$ (since $p'$ is a path from $A'$ to $B$).
But this would contradict $\ycoord(A') \geq \ycoord(B)$.
Hence, the path $p'$ has no north-step.
Thus, $p'$ consists entirely of east-steps.
Hence, $p'$ contains every vertex on the line segment between $A'$ and $B$.
Therefore, $p'$ contains the vertex $A$ (since the vertex $A$ lies on the line segment between $A'$ and $B$).
This contradicts the fact that $A$ does not belong to the path $p'$.
This contradiction shows that our assumption was false; hence, $\ell(p) \neq0$ is proven.]

Furthermore, we claim that $\ell(p') \neq 0$.

[\textit{Proof:}
Assume the contrary. Thus, $\ell(p') = 0$.
Hence, the path $p'$ has no steps.
Therefore, $A'=B$ (since $p'$ is a path from $A'$ to $B$).
Hence, $\xcoord(A') = \xcoord(B) \geq \xcoord(B')$, so that $\xcoord(A) \geq \xcoord(A') \geq \xcoord (B')$.
Combining this with $\xcoord(A) \leq \xcoord(B')$, we obtain $\xcoord(A) = \xcoord(B')$.
Combining $\xcoord(A) \geq \xcoord(A')$ with $\xcoord(A') \geq \xcoord(B') = \xcoord(A)$, we obtain $\xcoord(A)  = \xcoord(A')$.
Thus, $\xcoord(A') = \xcoord(A) = \xcoord(B')$.
Therefore, the vertex $A'$ lies on the vertical line that contains $A$ and $B'$.
Furthermore, this vertex $A'$ must lie on the line segment between $A$ and $B'$ (since $\ycoord(A') \geq \ycoord(A)$ and $\ycoord(\underbrace{A'}_{=B}) = \ycoord(B) \leq \ycoord(B')$).

Recall again that the $x$-coordinates of the vertices a path are always weakly increasing.
Moreover, they increase strictly whenever the path makes an east-step.
Hence, if the path $p$ would have any east-step, then we would have $\xcoord(A) < \xcoord(B')$ (since $p$ is a path from $A$ to $B'$).
But this would contradict $\xcoord(A) = \xcoord(B')$.
Hence, the path $p$ has no east-step.
Thus, $p$ consists entirely of north-steps.
Hence, $p$ contains every vertex on the line segment between $A$ and $B'$.
Therefore, $p$ contains the vertex $A'$ (since the vertex $A'$ lies on the line segment between $A$ and $B'$).
This contradicts the fact that $A'$ does not belong to the path $p$.
This contradiction shows that our assumption was false; hence, $\ell(p') \neq 0$ is proven.]

Let $P$ be the second vertex of the path $p$.
(This is well-defined, since $\ell(p) \neq 0$.)
Hence, $P$ lies on a path from $A$ to $B'$ (namely, on the path $p$).
Therefore, $\xcoord(A) \leq \xcoord(P) \leq \xcoord(B')$ (since the $x$-coordinates of the vertices a path are
always weakly increasing) and $\ycoord(A) \leq \ycoord(P) \leq \ycoord(B')$ (since the $y$-coordinates of the vertices a path are always weakly increasing).
Let $r$ be the path from $P$ to $B'$ obtained by removing the first arc from $p$.
Thus, $r$ is a subpath of $p$.
Hence, the paths $r$ and $p'$ have no vertex in common (since $p$ and $p'$ have no vertex in common).
Also, $\ell(r) = \ell(p) - 1 < \ell(p)$ and thus $\underbrace{\ell(r)}_{<\ell(p)} + \ell(p') < \ell(p) + \ell(p') = N$.

Moreover, $\xcoord(A') \leq \xcoord(A) \leq \xcoord(P)$.
If we had $\ycoord(A') \geq \ycoord(P)$, then we could apply Lemma~\ref{lem.LGV.hex} to $P$ and $r$ instead of $A$ and $p$ (by the induction hypothesis, since $\ell(r) + \ell(p') < N$).
We thus would conclude that the paths $r$ and $p'$ have a vertex in common; this would contradict the fact that the paths $r$ and $p'$ have no vertex in common.
Hence, we cannot have $\ycoord(A') \geq \ycoord(P)$.
Thus, $\ycoord(A') < \ycoord(P)$.
Hence, $\ycoord(A') \leq \ycoord(P)  -1$ (since $\ycoord(A')$ and $\ycoord(P)$ are integers).

But $P$ is the next vertex after $A$ on the path $p$.
Hence, there is an arc from $A$ to $P$. If this arc was an east-step, then we would have $\ycoord(P) = \ycoord(A)$, which would contradict $\ycoord(A) \leq \ycoord(A') < \ycoord(P)$.
Hence, this arc cannot be an east-step.
Thus, this arc must be a north-step.
Therefore, $\ycoord(P) = \ycoord(A) + 1$ and $\xcoord(P) = \xcoord(A)$.
Combining $\ycoord(A') \leq \ycoord(P) - 1 = \ycoord(A)$ (since $\ycoord(P) = \ycoord(A) + 1$) with $\ycoord(A') \geq \ycoord(A)$, we obtain $\ycoord(A') = \ycoord(A)$.

Let $P'$ be the second vertex on the path $p'$.
(This is well-defined, since $\ell(p') \neq0 $.)
Hence, $P'$ lies on a path from $A'$ to $B$ (namely, on the path $p'$).
Therefore, $\xcoord(A') \leq \xcoord(P') \leq \xcoord(B)$ (since the $x$-coordinates of the vertices a path are always weakly increasing) and $\ycoord(A') \leq \ycoord(P') \leq \ycoord(B)$ (since the $y$-coordinates of the vertices a path are always weakly increasing).
Let $r'$ be the path from $P'$ to $B$ obtained by removing the first arc from $p'$.
Thus, $r'$ is a subpath of $p'$.
Hence, the paths $p$ and $r'$ have no vertex in common (since $p$ and $p'$ have no vertex in common).
Also, $\ell(r') = \ell(p') - 1 < \ell(p')$, and thus $\ell(p) + \underbrace{\ell(r')}_{<\ell(p')} < \ell(p) + \ell(p') = N$.

Moreover, $\ycoord(P') \geq \ycoord(A') \geq \ycoord(A)$.
If we had $\xcoord(P') \leq \xcoord(A)$, then we could apply Lemma~\ref{lem.LGV.hex} to $P'$ and $r'$ instead of $A'$ and $p'$ (by the induction hypothesis, since $\ell(p) + \ell(r') < N$).
We thus would conclude that the paths $p$ and $r'$ have a vertex in common; this would contradict the fact that the paths $p$ and $r'$ have no vertex in common.
Hence, we cannot have $\xcoord(P') \leq \xcoord(A)$.
Thus, $\xcoord(P') > \xcoord(A)$.
Hence, $\xcoord(P') \geq \xcoord(A) + 1$ (since $\xcoord(P')$ and $\xcoord(A)$ are integers).

But $P'$ is the next vertex after $A'$ on the path $p'$.
Hence, there is an arc from $A'$ to $P'$.
If this arc was a north-step, then we would have $\xcoord(P') = \xcoord(A')$, which would contradict $\xcoord(A') \leq \xcoord(A) < \xcoord(P')$.
Hence, this arc cannot be a north-step.
Thus, this arc must be an east-step.
Therefore, $\ycoord(P') = \ycoord(A')$ and $\xcoord(P') = \xcoord(A') + 1$.
Hence, $\xcoord(A') + 1 = \xcoord(P') \geq \xcoord(A) + 1$, so that $\xcoord(A') \geq \xcoord(A)$.
Combining this with $\xcoord(A') \leq \xcoord(A)$, we obtain $\xcoord(A') = \xcoord(A)$.

Now, the vertices $A$ and $A'$ have the same $x$-coordinate (since $\xcoord(A') = \xcoord(A)$) and the same $y$-coordinate (since $\ycoord(A') = \ycoord(A)$).
Hence, these two vertices are equal.
In other words, $A = A'$.
Hence, the vertex $A$ belongs to the path $p'$ (since the vertex $A'$ belongs to the path $p'$).
This contradicts the fact that the vertex $A$ does not belong to the path $p'$.
This contradiction shows that our assumption was false.
Hence, we have shown that $p$ and $p'$ have a vertex in common.

Now, forget that we fixed $A$, $B$, $A'$, $B'$, $p$ and $p'$.
We thus have proven that if $A$, $B$, $A'$, $B'$, $p$ and $p'$ are as in Lemma~\ref{lem.LGV.hex}, and if $\ell(p) + \ell(p') = N$, then $p$ and $p'$ have a vertex in common.
In other words, Lemma~\ref{lem.LGV.hex} holds when $\ell(p) + \ell(p') = N$.
This completes the induction step.
Hence, Lemma~\ref{lem.LGV.hex} is proven.
\end{proof}

\begin{proof}[Second proof of Lemma~\ref{lem.LGV.hex} (sketched).]
Recall that each arc of the lattice is either an east-step or a north-step.
Thus, the $x$-coordinates of the vertices of a path are always weakly increasing, and so are the $y$-coordinates.
Hence, the existence of a path $p$ from $A$ to $B'$ shows that $\xcoord(A) \leq \xcoord(B')$ and $\ycoord(A) \leq \ycoord(B')$.
Similarly, the existence of $p'$ yields $\xcoord(A') \leq \xcoord(B')$ and $\ycoord(A') \leq\ycoord(B)$.

Thus,
\begin{align}
\xcoord(A')  & \leq \xcoord(A)  \leq \xcoord(B') \leq \xcoord(B) \qquad \text{and} \label{pf.lem.LGV.hex.1} \\
\ycoord(A)  & \leq \ycoord(A') \leq \ycoord(B) \leq \ycoord(B'). \label{pf.lem.LGV.hex.2}
\end{align}
Now, consider the rectangle whose four sides are given by the equations
\[
x = \xcoord(A'), \qquad
y = \ycoord(A), \qquad
x = \xcoord(B), \qquad
y = \ycoord(B'),
\]
respectively (all in Cartesian coordinates).\footnote{See Figure~\ref{fig:LGV.hex.pf2} for this rectangle.}
Then, the path $p$ joins two opposite sides of this rectangle (namely, the second and the fourth), whereas
the path $p'$ joins the other two sides of this rectangle; moreover, both paths stay fully within the rectangle (due to~\eqref{pf.lem.LGV.hex.1} and~\eqref{pf.lem.LGV.hex.2}).
Hence, it is geometrically obvious that $p$ and $p'$ must meet; in other words, $p$ and $p'$ have a vertex in common.

\begin{figure}[t]
\[
\begin{tikzpicture}
  \draw[densely dotted] (-1.2,-1.2) grid (8.2,8.2);
  
  \draw (1,-1.2) -- (1,8.2);
  \draw (1,4.3) node[anchor=east] {$x = \xcoord(A')$};
  \draw (-1.2,1) -- (8.2,1);
  \draw (4,1) node[anchor=north] {$y = \ycoord(A)$};
  \draw (6,-1.2) -- (6,8.2);
  \draw (6,2.7) node[anchor=west] {$x = \xcoord(B)$};
  \draw (-1.2,6) -- (8.2,6);
  \draw (3,6) node[anchor=south] {$y = \ycoord(B')$};
  
  \node[circle,fill=white,draw=black,text=UMNmaroon,inner sep=1pt] (A') at (1,3) {$A^{\prime}$};
  \node[circle,fill=white,draw=black,text=UMNmaroon,inner sep=1pt] (A) at (2,1) {$A$};
  
  \node[circle,fill=white,draw=black,text=black,inner sep=1pt] (B') at (5,6) {$B^{\prime}$};
  \node[circle,fill=white,draw=black,text=black,inner sep=1pt] (B) at (6,4) {$B$};
  
  \begin{scope}[thick,>=stealth,darkred]
      % $p$:
      \draw (A) edge[->] (3,1);
      \draw (2.6,1) node[anchor=north] {$p$};
      \draw (3,1) edge[->] (3,2);
      \draw (3,2) edge[->] (4,2);
      \draw (4,2) edge[->] (4,3);
      \draw (4,3) edge[->] (4,4);
      \draw (4,4) edge[->] (4,5);
      \draw (4,5) edge[->] (5,5);
      \draw (5,5) edge[->] (B');
      \draw (5,5.4) node[anchor=east] {$p$};
  \end{scope}
  \begin{scope}[thick,>=stealth,dashed,darkred]
      % $\widetilde{p}$:
      \draw (2,0) edge[->] (A);
      \draw (2,-1) edge[->] (2,0);
      \draw (2,-1.2) edge[->] (2,-1);
      \draw (2,-0.3) node[anchor=west] {$\widetilde{p}$};
      \draw (B') edge[->] (5,7);
      \draw (5,7) edge[->] (5,8);
      \draw (5,8) edge (5,8.2);
      \draw (5,7.3) node[anchor=west] {$\widetilde{p}$};
  \end{scope}
  \begin{scope}[thick,>=stealth,dbluecolor]
      % $p^{\prime}$:
      \draw (A') edge[->] (2,3);
      \draw (1.6,3) node[anchor=north] {$p^{\prime}$};
      \draw (2,3) edge[->] (2,4);
      \draw (2,4) edge[->] (3,4);
      \draw (3,4) edge[->] (4,4);
      \draw (4,4) edge[->] (5,4);
      \draw (5,4) edge[->] (B);
      \draw (5.4,4) node[anchor=north] {$p^{\prime}$};
  \end{scope}
  \begin{scope}[thick,>=stealth,dashed,dbluecolor]
      % $\widetilde{p^{\prime}}$:
      \draw (0,3) edge[->] (A');
      \draw (-1,3) edge[->] (0,3);
      \draw (-1.2,3) edge[->] (-1,3);
      \draw (-0.3,3) node[anchor=north] {$\widetilde{p^{\prime}}$};
      \draw (B) edge[->] (7,4);
      \draw (7,4) edge[->] (8,4);
      \draw (8,4) edge (8.2,4);
      \draw (7.3,4) node[anchor=south] {$\widetilde{p^{\prime}}$};
  \end{scope}
\end{tikzpicture}
\]
\caption{Illustration for the Second proof of Lemma~\ref{lem.LGV.hex}.}
\label{fig:LGV.hex.pf2}
\end{figure}

[This last geometric argument can also be substituted by a purely combinatorial one.
Namely: Assume the contrary. Hence, the paths $p$ and $p'$ have no vertex in common.

We extend the path $p$ to a bidirectionally infinite path $\widetilde{p}$ by vertical steps (\textit{i.e.}, arcs of the form $(i,j) \to (i,j+1)$) in both directions (so the path $\widetilde{p}$ first reaches $A$ through infinitely many north-steps, then proceeds to $B'$ along $p$, and then leaves $B'$ along infinitely many north-steps).\footnote{These new steps are the dashed red steps in Figure~\ref{fig:LGV.hex.pf2}.}
We extend the path $p'$ to a bidirectionally infinite path $\widetilde{p'}$ by horizontal steps (\textit{i.e.}, arcs of the form $(i,j) \to (i+1,j)$) in both directions.\footnote{These new steps are the dashed blue steps in Figure~\ref{fig:LGV.hex.pf2}.}
The resulting infinite paths $\widetilde{p}$ and $\widetilde{p'}$ have no vertex in common (indeed, the paths $p$ and $p'$ stay within the rectangle discussed above, and have no vertex in common; meanwhile, the new
steps we have added to them to obtain $\widetilde{p}$ and $\widetilde{p'}$ escape this rectangle normally in four different directions, whence they intersect neither each other nor $p$ or $p'$).
Recall that each arc of the lattice is either an east-step or a north-step.
Thus, if a vertex $C$ runs through a path, the value $\xcoord(C) + \ycoord(C)$ is incremented by $1$ with each step.
Hence, if $q$ is a bidirectionally infinite path, then, for each $N \in \ZZ$, there is a unique vertex $C$ of $q$ satisfying $\xcoord(C) + \ycoord(C) = N$.
Denote this vertex $C$ by $h_N(q)$.
Thus, the vertices of $q$ are $\ldots, h_{-1}(q), h_{0}(q), h_1(q), \ldots$ in this order.

All sufficiently low $N \in \ZZ$ satisfy $\xcoord\bigl( h_N(\widetilde{p'}) \bigr) < \xcoord\bigl( h_N(\widetilde{p}) \bigr)$, whereas all sufficiently high $N \in \ZZ$ satisfy $\xcoord\bigl( h_N(\widetilde{p'}) \bigr) > \xcoord\bigl( h_N(\widetilde{p}) \bigr)$.
Hence, the set of all $N \in \ZZ$ that satisfy $\xcoord\bigl( h_N(\widetilde{p'}) \bigr) < \xcoord\bigl( h_N(\widetilde{p}) \bigr)$ is a nonempty set of integers that is bounded from above.
Therefore, this set has a maximum.
Let $M$ be this maximum.
Thus, $M \in \ZZ$ satisfies $\xcoord\bigl( h_{M}(\widetilde{p'}) \bigr) < \xcoord\bigl( h_{M}(\widetilde{p}) \bigr)$ but $\xcoord\bigl( h_{M+1}(\widetilde{p'}) \bigr) \geq \xcoord\bigl( h_{M+1}(\widetilde{p}) \bigr)$.

But the point $h_{M+1}(\widetilde{p'})$ is either the eastern or the northern neighbor of $h_{M}(\widetilde{p'})$; hence, $\xcoord\bigl( h_{M+1}(\widetilde{p'}) \bigr) \leq \xcoord\bigl( h_{M}(\widetilde{p'}) \bigr) + 1$.
Also, the point $h_{M+1}(\widetilde{p})$ is either the eastern or the northern neighbor of $h_{M}(\widetilde{p})$; thus, $\xcoord\bigl( h_{M+1}(\widetilde{p}) \bigr) \geq \xcoord\bigl( h_{M}(\widetilde{p}) \bigr)$.
Hence,
\[
\xcoord\bigl( h_{M+1}(\widetilde{p'}) \bigr) \leq \underbrace{\xcoord\bigl( h_{M}(\widetilde{p'}) \bigr)}_{<\xcoord\bigl( h_{M}(\widetilde{p}) \bigr)} + 1 < \underbrace{\xcoord\bigl( h_{M}(\widetilde{p}) \bigr)}_{\leq\xcoord\bigl( h_{M+1}(\widetilde{p}) \bigr)} + 1 \leq \xcoord\bigl( h_{M+1}(\widetilde{p}) \bigr) + 1.
\]
Therefore, $\xcoord\bigl( h_{M+1}(\widetilde{p'}) \bigr) \leq \xcoord\bigl( h_{M+1}(\widetilde{p}) \bigr)$ (since both sides are integers).
Combining this with $\xcoord\bigl( h_{M+1}(\widetilde{p'}) \bigr) \geq \xcoord\bigl( h_{M+1}(\widetilde{p}) \bigr)$, we obtain $\xcoord\bigl( h_{M+1}(\widetilde{p'}) \bigr) = \xcoord\bigl( h_{M+1}(\widetilde{p}) \bigr)$.
Hence, $h_{M+1}(\widetilde{p'}) = h_{M+1}(\widetilde{p})$ (since $\xcoord\bigl( h_{M+1}(\widetilde{p'}) \bigr) -\ycoord\bigl( h_{M+1}(\widetilde{p'}) \bigr) = M + 1 = \xcoord\bigl( h_{M+1}(\widetilde{p}) \bigr) - \ycoord( h_{M+1}(\widetilde{p}) \bigr)$ as well).
Thus, the paths $\widetilde{p}$ and $\widetilde{p'}$ have a vertex in common (namely, $h_{M+1}(\widetilde{p'}) = h_{M+1}(\widetilde{p})$).
This contradicts the fact that they don't.
This contradiction shows that our assumption was false.
Hence, the paths $p$ and $p'$ have a vertex in common.]
\end{proof}

We can now specialize Proposition~\ref{prop.LGV.nonper} to a form that is most
suited for our applications:\footnote{%
Proposition~\ref{prop.LGV.concrete} is also the particular case of Corollary~4 in \url{https://math.stackexchange.com/questions/2870640} (applied to $\omega_{a} = \wt(a)$).
The proof we are giving here is precisely the proof from \url{https://math.stackexchange.com/questions/2870640}.}
\end{verlonglong}

We now shall state a folklore result, which follows from the Lindstr\"om--Gessel--Viennot lemma:\footnote{%
See the proof of Corollary~4 in \url{https://math.stackexchange.com/questions/2870640} (applied to $\omega_{a} = \wt(a)$) for a detailed derivation of Proposition~\ref{prop.LGV.concrete}.}

\begin{prop}
\label{prop.LGV.concrete}
Let $k \in \NN$.
Let $\uu = \tup{A_1, A_2, \dotsc, A_k}$ and $\vv = \tup{B_1,B_2, \dotsc, B_k}$ be two $k$-vertices such that
\begin{align*}
\xcoord(A_1) & \geq \xcoord(A_2) \geq \cdots \geq \xcoord(A_k), \\
\ycoord(A_1) & \leq \ycoord(A_2) \leq \cdots \leq \ycoord(A_k), \\
\xcoord(B_1) & \geq \xcoord(B_2) \geq \cdots \geq \xcoord(B_k), \\
\ycoord(B_1) & \leq \ycoord(B_2) \leq \cdots \leq \ycoord(B_k).
\end{align*}
Then,
\[
\sum_{\pp \in N(\uu,\vv)} \wt(\pp) = \det\left( \sum_{p\in N(A_i,B_j) } \wt(p) \right)_{i, j \in \ive{k}} .
\]
\end{prop}

The situation of this proposition is illustrated in Figure~\ref{fig:NILP_example}.

\begin{verlonglong}
\begin{proof}[Proof of Proposition~\ref{prop.LGV.concrete}.]
This will follow immediately from Proposition~\ref{prop.LGV.nonper}
once we have shown that $\left( \uu,\vv\right)  $ is nonpermutable.
Thus, it remains to show the latter fact.

Let $\sigma\neq\id$ be a permutation in $\SymGp{k}$.
We must show that $N\bigl(\uu, \sigma(\vv) \bigr) = \emptyset$.
Indeed, let $\pp\in N\bigl( \uu,\sigma(\vv) \bigr)$.

The permutation $\sigma$ has an inversion (since $\sigma \neq \id$).
In other words, there exist some $i<j$ in $\ive{k}$ such that $\sigma(i) > \sigma(j)$.
Consider such $i$ and $j$.

Write $\pp$ in the form $\pp = \tup{p_1, p_2, \dotsc, p_k}$.
Thus, $\tup{p_1, p_2, \dotsc, p_k}$ is a NILP from $\uu$ to $\sigma(\vv)$ (since $\tup{p_1, p_2, \dotsc, p_k} = \pp\in N\bigl( \uu, \sigma(\vv) \bigr)$).

By the definition of a NILP, this implies that
\begin{itemize}
\item $p_i$ is a path from $A_i$ to $B_{\sigma(i)}$;

\item $p_j$ is a path from $A_j$ to $B_{\sigma(j)}$;

\item the paths $p_i$ and $p_j$ have no vertex in common.
\end{itemize}

However, the assumptions of Proposition~\ref{prop.LGV.concrete} yield
$\xcoord(A_j) \leq \xcoord(A_i)$,
$\ycoord(A_j) \geq \ycoord(A_i)$,
$\xcoord(B_{\sigma(i)}) \leq \xcoord(B_{\sigma(j)})$, and
$\ycoord(B_{\sigma(i)}) \geq \ycoord(B_{\sigma(j)})$.
Hence, Lemma~\ref{lem.LGV.hex} (applied to $A = A_i$, $B = B_{\sigma(j)}$, $A' = A_j$, $B' = B_{\sigma(i)}$, $p = p_i$ and $p' = p_j$) shows that $p_i$ and $p_j$ have a vertex in common.
This contradicts the fact that the paths $p_i$ and $p_j$ have no vertex in common.

Thus we have found a contradiction for each $\pp\in N\bigl( \uu,\sigma(\vv) \bigr)$.
Hence, $N\bigl(\uu, \sigma(\vv) \bigr) = \emptyset$, as we desired.
\end{proof}
\end{verlonglong}

%%%%%%%%%%
\subsection{Pseudo-partitions and tableaux}

We shall next introduce the concepts of pseudo-partitions and their corresponding semistandard tableaux;
we will then express a generating function for these tableaux by a determinantal formula (Theorem~\ref{thm.tableau.jt}) akin to the Jacobi--Trudi formula for Schur functions (and, like the latter, the proof will rely on Proposition~\ref{prop.LGV.concrete}).
A pseudo-partition is similar to the concept of a partition, except it allows entries to increase by $1$.
The semistandard tableaux of a pseudo-partition shape are defined just like for partitions.
The generating function in our determinantal formula is going to be a sum over the semistandard tableaux of a fixed pseudo-partition shape with given rightmost entries in each row.
Later we will translate these tableaux into MLQs that yield a specific word when applied to $1^n$.

Let us formalize these definitions.
A \defn{pseudo-partition} shall mean a $k$-tuple $\lambda = \tup{\lambda_1, \lambda_2, \dotsc, \lambda_k}$ of positive integers (for some $k \in \NN$) such that
\[
\text{each } i \in \ive{k-1} \text{ satisfies } \lambda_i + 1 \geq \lambda_{i+1}.
\]
For example, both $\tup{5,3,4,2,2}$ and $\tup{6,2,3,4,1}$ are pseudo-partitions.

The \defn{diagram $\diag{\lambda}$} of a pseudo-partition $\lambda = \tup{\lambda_1, \lambda_2, \dotsc,\lambda_k}$ is defined as the set
$\set{  \left(  i,j\right)  \in \ive{k} \times \set{ 1,2,3,\ldots }  \mid  j \leq \lambda_i }$.
This diagram is drawn in the plane like usual Young diagrams, in English notation.
\begin{verlong}
For example, the diagram of the pseudo-partition $\tup{3,2,3,1}$ is drawn as
\[
\ydiagram{3,2,3,1}.
\]
\end{verlong}

If $\lambda = \tup{\lambda_1, \lambda_2, \dotsc, \lambda_k}$ is a pseudo-partition, then a \defn{tableau of shape $\lambda$} is a map $T \colon \diag{\lambda} \to \set{1,2,3,\ldots}$.
For each such tableau $T$ and each $(i,j) \in \diag{\lambda}$, we refer to the value $T(i,j)$ as the \defn{entry of $T$ in cell $\left(i,j\right)$}.
As usual, we represent a tableau $T$ of shape $\lambda$ by placing the entry $T(i,j)$ into the box corresponding to $(i,j) \in \diag{\lambda}$.
\begin{verlong}
For example, the following are two tableaux of shape $\tup{3,2,3,1}$:
\begin{equation}
\label{eq.tableau.exas-long}
\begin{ytableau} 2 & 5 & 3 \\ 1 & 1 \\ 2 & 4 & 1 \\ 7 \end{ytableau}\qquad
\text{and}\qquad
\begin{ytableau} 1 & 1 & 3 \\ 2 & 3 \\ 4 & 5 & 5 \\ 5 \end{ytableau}.
\end{equation}
\end{verlong}

A tableau $T$ of shape $\lambda$ is said to be \defn{semistandard} if and only if
\begin{itemize}
\item the entries of $T$ weakly increase along each row (\textit{i.e.}, we have $T(i,j_1) \leq T(i,j_2)$ whenever $(i,j_1) \in \diag{\lambda}$ and $(i,j_2) \in \diag{\lambda}$ satisfy $j_1 < j_2$);

\item the entries of $T$ strictly increase down each column (\textit{i.e.}, we have $T(i_1,j) < T(i_2,j)$ whenever $(i_1,j) \in \diag{\lambda}$ and $(i_2,j) \in \diag{\lambda}$ satisfy $i_1 < i_2$).
\end{itemize}

\begin{verlong}
For example, the second tableau in~\eqref{eq.tableau.exas-long} is semistandard, while the first is not.
\end{verlong}

If $\lambda = \tup{\lambda_1, \lambda_2, \dotsc, \lambda_k}$ is a pseudo-partition, and if $T$ is a tableau of shape $\lambda$,
then the \defn{surface} of $T$ is defined to be the $k$-tuple $(s_1, s_2, \dotsc, s_k)$, where $s_i$ is the rightmost entry of the $i$-th row of $T$ (that is, $s_i = T(i,\lambda_i)$).

If $T$ is a tableau of shape $\lambda$ whose entries belong to $\ive{n}$, then the \defn{weight} of $T$ is defined as the monomial
\[
\wt(T) := \prod_{(i,j) \in \diag{\lambda}} x_{T(i,j)}
\]
(that is, the product of $x_{d}$ for $d$ ranging over all entries of $T$).
\begin{verlong}
For example, the two tableaux in~\eqref{eq.tableau.exas-long} have weights $x_1^{3}x_2^{2}x_{3}x_{4}x_{5}x_{7}$ and $x_1^{2}x_2x_{3}^{2}x_{4}x_{5}^{3}$, respectively (assuming that $n \geq 7$).
\end{verlong}
Let $\SST(\lambda, s)$ denote the set of all semistandard tableaux of shape $\lambda$ and surface $s$.

\begin{example}
The diagram of the pseudo-partition $\tup{3,2,3,1}$ is drawn as
\[
\ydiagram{3,2,3,1}\,.
\]
The following are two tableaux of shape $\tup{3,2,3,1}$:
\[
\label{eq.tableau.exas}
\begin{ytableau} 1 & 2 & 5 \\ 2 & 3 \\ 3 & 4 & 5 \\ 7 \end{ytableau}\qquad
\text{and}\qquad
\begin{ytableau} 1 & 1 & 4 \\ 2 & 3 \\ 4 & 5 & 5 \\ 5 \end{ytableau}\,.
\]
The right tableau is semistandard, while the left one is not (the two $5$'s in the rightmost column).
These tableaux have surfaces $(5,3,5,7)$ and $(4,3,5,5)$, respectively, and weights $x_1 x_2^2 x_3^2 x_4 x_5^2 x_7$ and $x_1^2 x_2 x_3 x_4^2 x_5^3$, respectively.
\end{example}

We now state the main result of this subsection.

\begin{thm}
\label{thm.tableau.jt}
Let $\lambda = \tup{\lambda_1, \lambda_2, \dotsc, \lambda_k}$ be a pseudo-partition.
Let $s = \tup{s_1, s_2, \dotsc, s_k}$ be a strictly increasing sequence of elements of $\ive{n}$.
Then,
\[
\sum_{T \in \SST(\lambda, s)} \wt(T) = \left(  \prod_{i=1}^{k} x_{s_i} \right)  \det\left(  h_{\lambda_j-j+i-1}(  x_1,x_2,\ldots,x_{s_j})  \right)_{i, j \in \ive{k}} .
\]
\end{thm}

We remark that the determinant in Theorem~\ref{thm.tableau.jt} is an instance of a \defn{multi-Schur function} as defined in~\cite[(SCHUR.2.2)]{LLPT18}.

\begin{figure}[t]
\[
\begin{ytableau} 1 & 1 & 2 \\ 2 & 3 \\ 4 & 5 & 5 \\ 6 \end{ytableau}
\quad \longleftrightarrow \quad
\begin{tikzpicture}[baseline=3.5cm]
  \draw[densely dotted] (-1,0) grid (6.2,6.2);
  % axes:
  \draw[->] (-1,0) -- (-1,6.2);
  \draw[->] (-1,0) -- (6.2,0);
  \foreach \x/\xtext in {-1, 0, 1, 2, 3, 4, 5, 6}
     \draw (\x cm,1pt) -- (\x cm,-1pt) node[anchor=north] {$\xtext$};
  \foreach \y/\ytext in {0, 1, 2, 3, 4, 5, 6}
     \draw (-1cm+1pt,\y cm) -- (-1cm-1pt,\y cm) node[anchor=east] {$\ytext$};
  \foreach \x in {1, 2, 3, 4}
    \node[circle,fill=white,draw=black,text=UMNmaroon,inner sep=1pt] (A\x) at (4-\x,1) {$A_{\x}$};

  \node[circle,fill=white,draw=black,text=black,inner sep=1pt] (B4p) at (1+4-4-1,6) {$B_4$};
  \node[circle,fill=black,inner sep=1.5pt] (B4) at (1+4-4,6) {};
  \node[circle,fill=white,draw=black,text=black,inner sep=1pt] (B3p) at (3+4-3-1,5) {$B_3$};
  \node[circle,fill=black,inner sep=1.5pt] (B3) at (3+4-3,5) {};
  \node[circle,fill=white,draw=black,text=black,inner sep=1pt] (B2p) at (2+4-2-1,3) {$B_2$};
  \node[circle,fill=black,inner sep=1.5pt] (B2) at (2+4-2,3) {};
  \node[circle,fill=white,draw=black,text=black,inner sep=1pt] (B1p) at (3+4-1-1,2) {$B_1$};
  \node[circle,fill=black,inner sep=1.5pt] (B1) at (3+4-1,2) {};

  \begin{scope}[thick,>=stealth,UQpurple]
      % $p_4$:
      \draw (A4) edge[->] (0,2);
      \draw (0,1.6) node[anchor=east] {$1$};
      \draw (0,2) edge[->] (0,3);
      \draw (0,2.5) node[anchor=east] {$1$};
      \draw (0,3) edge[->] (0,4);
      \draw (0,3.5) node[anchor=east] {$1$};
      \draw (0,4) edge[->] (0,5);
      \draw (0,4.5) node[anchor=east] {$1$};
      \draw (0,5) edge[->] (B4p);
      \draw (0,5.4) node[anchor=east] {$1$};
      \draw (B4p) edge[->,color=black] (B4);
      \draw[color=black] (0.6,6) node[anchor=south] {$x_6$};
  \end{scope}
  \begin{scope}[thick,>=stealth,darkred]
      % $p_3$:
      \draw (A3) edge[->] (1,2);
      \draw (1,1.6) node[anchor=east] {$1$};
      \draw (1,2) edge[->] (1,3);
      \draw (1,2.5) node[anchor=east] {$1$};
      \draw (1,3) edge[->] (1,4);
      \draw (1,3.5) node[anchor=east] {$1$};
      \draw (1,4) edge[->] (2,4);
      \draw (1.5,4) node[anchor=north] {$x_4$};
      \draw (2,4) edge[->] (2,5);
      \draw (2,4.5) node[anchor=east] {$1$};
      \draw (2,5) edge[->] (B3p);
      \draw (2.4,5) node[anchor=north] {$x_5$};
      \draw (B3p) edge[->,color=black] (B3);
      \draw[color=black] (3.6,5) node[anchor=south] {$x_5$};
  \end{scope}
  \begin{scope}[thick,>=stealth,dbluecolor]
      % $p_2$:
      \draw (A2) edge[->] (2,2);
      \draw (2,1.6) node[anchor=east] {$1$};
      \draw (2,2) edge[->] (3,2);
      \draw (2.5,2) node[anchor=north] {$x_2$};
      \draw (3,2) edge[->] (B2p);
      \draw (3,2.4) node[anchor=east] {$1$};
      \draw (B2p) edge[->,color=black] (B2);
      \draw[color=black] (3.6,3) node[anchor=south] {$x_3$};
  \end{scope}
  \begin{scope}[thick,>=stealth,dgreencolor]
      % $p_1$:
      \draw (A1) edge[->] (4,1);
      \draw (3.6,1) node[anchor=north] {$x_1$};
      \draw (4,1) edge[->] (5,1);
      \draw (4.6,1) node[anchor=north] {$x_1$};
      \draw (5,1) edge[->] (B1p);
      \draw (5,1.4) node[anchor=east] {$1$};
      \draw (B1p) edge[->,color=black] (B1);
      \draw[color=black] (5.6,2) node[anchor=south] {$x_2$};
  \end{scope}
  
\end{tikzpicture}
\]
\begin{vershort}
\caption{A semistandard tableau of surface $(2,3,5,6)$ (left) and the corresponding NILP (right).
  Note that the final horizontal steps are fixed and correspond to the surface.}
\end{vershort}
\begin{verlong}
\caption{A semistandard tableau of surface $(2,3,5,6)$ (left) and the corresponding NILP (right).
  The black horizontal steps don't belong to the paths of the NILP; they stand for the last entries of the rows of the tableau.}
\end{verlong}
\label{fig:tableau_to_NILP}
\end{figure}

\begin{vershort}
\begin{proof}[Proof of Theorem~\ref{thm.tableau.jt}.]
The proof is similar to the usual bijection relating semistandard tableaux and NILPs that is used in proving the Jacobi--Trudi identity.
For example, see~\cite[First proof of Thm. 7.16.1]{Stanley-EC2}.\footnote{Note that in this proof, those paths go west instead of east, but this is just a reflection across the $y$-axis.}
Thus, we define two $k$-vertices $\uu = (A_1, \dotsc, A_k)$ and $\vv = (B_1, \dotsc, B_k)$ with
\[
A_i = (k-i,1)  \qquad \text{and} \qquad B_i = (\lambda_i+k-i-1, s_i).
\]
We can then define a bijection $\Phi \colon N(\uu, \vv) \to \SST(\lambda, s)$ by requiring that if $\pp = \tup{p_1, p_2, \dotsc, p_k} \in N(\uu, \vv)$ is a NILP, then the first $\lambda_i - 1$ entries of the $i$-th row of the tableau $\Phi(\pp)$ will be the $y$-coordinates of the east-steps of the path $p_i$ (while the last entry of the $i$-th row must of course be $s_i$).
Then, Proposition~\ref{prop.LGV.concrete} and~\eqref{eq.LGV.single-paths} yield Theorem~\ref{thm.tableau.jt}.
We comment on the differences to~\cite[First proof of Thm. 7.16.1]{Stanley-EC2}:
\begin{itemize}
\item The points $B_1, B_2, \ldots, B_k$ no longer lie on a horizontal line; but the conditions of Proposition~\ref{prop.LGV.concrete} are still satisfied since $s$ is strictly increasing and $\lambda$ is a pseudo-partition.
\item In the tableau $\Phi(\pp)$, only the first $\lambda_i - 1$ entries of the $i$-th row come from the path $p_i$; the last entry is $s_i$. Thus the $\prod_{i=1}^k x_{s_i}$ factor in $\sum_{T \in \SST(\lambda, s)} \wt(T)$.
\item Since $\lambda$ is only a pseudo-partition, a column of $\diag{\lambda}$ can have gaps between cells. Thus, we need to argue that if a tableau $T$ of shape $\lambda$ with surface $s$ has weakly increasing rows and its columns are strictly increasing between consecutive rows (\textit{i.e.}, we have $T(i, j) < T(i+1, j)$ whenever both entries exist), then the columns of $T$ are also strictly increasing across gaps (\textit{i.e.}, we have $T(i_1, j) < T(i_2, j)$ for all $i_1 < i_2$). This is an easy exercise using the assumption that $s$ is strictly increasing.
\end{itemize}
\end{proof}
\end{vershort}

\begin{verlong}
\begin{proof}[Proof of Theorem~\ref{thm.tableau.jt}.]
Let us define two $k$-vertices $\uu = \tup{A_1, A_2, \dotsc, A_k}$ and $\vv = \tup{B_1, B_2, \dotsc, B_k}$ by
\[
A_i = (k-i,1)  \qquad \text{and} \qquad B_i = (\lambda_i+k-i-1, s_i).
\]
The conditions of Proposition~\ref{prop.LGV.concrete} are satisfied (since $\lambda$ is a pseudo-partition and $s$ is weakly increasing).
Thus, Proposition~\ref{prop.LGV.concrete} yields
\begin{align}
\sum_{\pp \in N(\uu,\vv)} \wt(\pp) & = \det\left(  \sum_{p \in  N(A_i,B_j)}\wt(p) \right)_{i, j \in \ive{k}} \nonumber \\
 & = \det\left( h_{\lambda_j-j+i-1}(x_1, x_2, \dotsc, x_{s_j}) \right)_{i, j \in \ive{k}},
\label{pf.thm.tableau.jt.1}
\end{align}
where the second equality follows from the fact each $i,j \in \ive{k}$ satisfy
\[
\sum_{p \in N(A_i,B_j)} \wt(p) = h_{\lambda_j-j+i-1}(x_1, x_2, \dotsc, x_{s_j})
\]
(by~\eqref{eq.LGV.single-paths}, applied to $A = A_i$ and $B = B_j$).

Next, let us define a bijection
\[
\Phi \colon N(\uu, \vv) \to \SST(\lambda, s).
\]

Indeed, let $\tup{p_1, p_2, \dotsc, p_k} \in N(\uu,\vv)$ be arbitrary.
Thus, each $p_i$ is a path from $A_i$ to $B_i$, and no two of the paths $p_1, p_2, \dotsc, p_k$ have a vertex in common.

For each $i \in \ive{k}$, the path $p_i$ is a path from $A_i = (k-i, 1)$ to $B_i = (\lambda_i+k-i-1,s_i)$; thus, it must contain exactly $(\lambda_i+k-i-1) - (k-i) = \lambda_i - 1$ east-steps.

Let $p_{i,1},p_{i,2}, \dotsc, p_{i,\lambda_i-1}$ be the $y$-coordinates of these east-steps (from first to last).
Also, set $p_{i,\lambda_i} = s_i$.
Thus,
\begin{equation}
\label{pf.thm.tableau.jt.row-weak}
1 \leq p_{i,1} \leq p_{i,2} \leq \cdots \leq p_{i,\lambda_i-1} \leq p_{i,\lambda_i} = s_i
\end{equation}
for each $i \in \ive{k}$.
Note that for each $i \in \ive{k}$,
\begin{subequations}
\begin{gather}
\label{pf.thm.tableau.jt.pi-vert-1}
\text{the path }p_i\text{ contains the vertices } (k-i+j-1, p_{i,j})  \text{ for all } j \in \ive{\lambda_i}
\\
\label{pf.thm.tableau.jt.pi-vert-2}
\text{and the vertices } (k-i+j, p_{i,j}) \text{ for all } j \in \ive{\lambda_i-1}.
\end{gather}
\end{subequations}

Now, let $T$ be the tableau of shape $\lambda$ that sends each $(i,j) \in \diag{\lambda}$  to $p_{i,j}$.
Then, the entries of $T$ are weakly increasing along each row by~\eqref{pf.thm.tableau.jt.row-weak}.
We shall next show that the entries of $T$ are strictly increasing down each column.

First, we claim that if $i \in \ive{k-1}$, then, for all $j \in \ive{\lambda_i} \cap \ive{\lambda_{i+1}}$, we have
\begin{equation}
\label{pf.thm.tableau.jt.col-strict-no-gap}
p_{i,j} < p_{i+1,j}.
\end{equation}

\begin{subproof}
We will show~\eqref{pf.thm.tableau.jt.col-strict-no-gap} holds.
Let $i \in \ive{k-1}$.
We must prove~\eqref{pf.thm.tableau.jt.col-strict-no-gap} for each $j \in \ive{\lambda_i}  \cap \ive{\lambda_{i+1}}  $.
To do so, we assume the contrary, and pick the \emph{smallest} $j$ for which~\eqref{pf.thm.tableau.jt.col-strict-no-gap} fails.
Thus, $p_{i,j} \geq p_{i+1,j}$.
Assume that $j > 1$ (the argument for the case $j=1$ is similar).
Thus, the minimality of $j$ forces $p_{i,j-1} < p_{i+1,j-1}$.
Finally,~\eqref{pf.thm.tableau.jt.row-weak} yields $p_{i+1,j-1} \leq p_{i+1,j}$.
Hence, $p_{i,j-1} < p_{i+1,j-1} \leq p_{i+1,j} \leq p_{i,j}$.
No two of the paths $p_1, p_2, \dotsc, p_k$ have a vertex in common.
Hence, $p_i$ and $p_{i+1}$ have no vertex in common.
Note that~\eqref{pf.thm.tableau.jt.pi-vert-1} yields that $p_i$ contains the vertex $\tup{k-i+j-1, p_{i,j}}$.
Also,~\eqref{pf.thm.tableau.jt.pi-vert-2} (applied to $j-1$ instead of $j$) yields that $p_i$ contains the vertex $\tup{k-i+j-1, p_{i,j-1}}$.
However, the vertex $\tup{k-i+j-1, p_{i+1,j}}$ lies on the vertical line segment connecting the two vertices $\tup{k-i+j-1, p_{i,j}}$ and $\tup{k-i+j-1,p_{i,j-1}}$ (since $p_{i,j-1} \leq p_{i+1,j} \leq p_{i,j}$).
Hence, $p_i$ must contain the former vertex (since $p_i$ contains the latter two vertices).

From~\eqref{pf.thm.tableau.jt.row-weak}, we obtain $p_{i,j} \leq s_i < s_{i+1}$ (since $s$ is strictly increasing).
If we had $j = \lambda_{i+1}$, then we would have $p_{i+1,j} = p_{i+1,\lambda_{i+1}} = s_{i+1}$ (by the definition of
$p_{i+1,\lambda_{i+1}}$), which would contradict $p_{i+1,j} \leq p_{i,j} < s_{i+1}$.
Hence, $j \neq \lambda_{i+1}$. Thus, $j \in \ive{\lambda_{i+1}-1}$ (since $j \in \ive{\lambda_{i+1}}$).
Therefore,~\eqref{pf.thm.tableau.jt.pi-vert-2} (applied to $i+1$ instead of $i$) shows that $p_{i+1}$ contains the vertex $\tup{k-(i+1)+j, p_{i+1,j}} = \tup{k-i+j-1, p_{i+1,j}}$.
So we have shown that both $p_i$ and $p_{i+1}$ contain the vertex $\tup{k-i+j-1, p_{i+1,j}}$.
This contradicts the fact that $p_i$ and $p_{i+1}$ have no vertex in common.

We have assumed that $j > 1$; but the same argument works for $j = 1$ with minor 
changes
(we now need to use $\tup{k-i+j-1, 1}$ instead of $\tup{k-i+j-1, p_{i,j-1}}$).
Thus, we always obtain a contradiction.
This contradiction shows that our assumption was false; hence,~\eqref{pf.thm.tableau.jt.col-strict-no-gap} is proven.
\end{subproof}

Next, we claim that the entries of $T$ are strictly increasing down each column.

\begin{subproof}
Let $i_1,i_2\in\ive{k}$ be such that $i_1 < i_2$.
Let $j \in \ive{\lambda_{i_1}} \cap \ive{\lambda_{i_2}}$.
We must show that $T( i_1,j) < T(i_2, j)$.
In other words, we must prove that $p_{i_1,j} < p_{i_2,j}$ (since $T(i,j) = p_{i,j}$ for all $i$).

If the column $j$ of $\diag{\lambda}$ has no gaps between row $i_1$ and row $i_2$ (that is, if we have $(i,j) \in \diag{\lambda}$ for each $i \in \set{i_1, i_1+1, \dotsc, i_2}$), then this follows from~\eqref{pf.thm.tableau.jt.col-strict-no-gap}.
Hence, without loss of generality, we assume that column $j$ of $\diag{\lambda}$ has at least one gap between row $i_1$ and row $i_2$.
In other words, there exists some $q \in \set{i_1, i_1+1, \dotsc, i_2}$ such that $(q,j) \notin \diag{\lambda}$.
Consider the \emph{largest} such $q$.
Clearly, $i_1 < q < i_2$.
Also, $\lambda_{q} < j$ (since $(q,j) \notin \diag{\lambda}$).

The maximality of $q$ shows that column $j$ of $\diag{\lambda}$ has no gaps between row $q$ and row $i_2$.
Hence,~\eqref{pf.thm.tableau.jt.col-strict-no-gap} shows that $p_{q,j}<p_{i_2,j}$ (since $q<i_2$).

From~\eqref{pf.thm.tableau.jt.row-weak}, we obtain $p_{i_1,j} \leq s_{i_1} \leq s_{q}$ (since $i_1 < q$, but the sequence $s$ is weakly increasing).
Also, $p_{q,\lambda_{q}} = s_{q}$ (by the definition of $p_{q,\lambda_{q}}$), whence $s_{q} = p_{q,\lambda_{q}} \leq p_{q,j}$ (by~\eqref{pf.thm.tableau.jt.row-weak}, since $\lambda_{q} < j$).
Hence, $p_{i_1,j} \leq s_{q} \leq p_{q,j} < p_{i_2,j}$. This proves our claim.
\end{subproof}

Altogether, we thus know that $T$ is a semistandard tableau of shape $\lambda$ and surface $s$ (since $T(i, \lambda_i) = p_{i,\lambda_i} = s_i$ for all $i \in \ive{k}$).
That is, $T \in \SST(\lambda, s)$.
So let us set
\[
\Phi(p_1, p_2, \dotsc, p_k) = T.
\]
Thus, the map $\Phi$ is defined.

It is easy to see that this map $\Phi$ is injective (indeed, $\tup{p_1,p_2,\dotsc,p_k}$ can be reconstructed from $T$, since the $i$-th row of $T$ encodes the $y$-coordinates of all the east-steps of $p_i$) and surjective (indeed, the reverse of the above construction turns every $T \in \SST(\lambda, s)$ into a $k$-tuple $\tup{p_1, p_2, \dotsc, p_k}$ of paths $p_i$ from $A_i$ to $B_i$; furthermore, the strict increase of the entries of $T$ down columns forces these paths $p_1,p_2,\ldots,p_k$ to have no vertices in common).
Thus, $\Phi$ is a bijection.
Moreover, it is easy to see that
\[
\wt\bigl(  \Phi(\pp) \bigr) = \left( \prod_{i=1}^k x_{s_i} \right) \wt(\pp)
\]
for each $\pp \in N(\uu,\vv)$.
Hence,
\begin{align*}
\sum_{T \in \SST(\lambda, s)} \wt(T)  &  = \sum_{\pp\in N(\uu,\vv)  }\left( \prod_{i=1}^k x_{s_i} \right)  \wt(\pp)
= \left( \prod_{i=1}^k x_{s_i} \right) \sum_{\pp\in N(\uu,\vv)  } \wt(\pp) \\
&  \overset{\eqref{pf.thm.tableau.jt.1}}{=} \left(  \prod_{i=1}^k x_{s_i} \right)  \det\left( h_{\lambda_j-j+i-1}(x_1, x_2, \dotsc, x_{s_j}) \right)_{i, j \in \ive{k}}.
\end{align*}
\end{proof}

We note that during the construction of the bijection $\Phi$, the fixed surface $s$ condition on the semistandard tableaux translates into requiring each path $p_i$ to having a final arc $B_i \to (\lambda_i + k - i, s_i)$.
Otherwise the rest of the proof is similar to the usual Jacobi--Trudi bijection relating semistandard tableaux and NILPs.
\end{verlong}

See Figure~\ref{fig:tableau_to_NILP} for an example of the bijection $\Phi$ used in the proof of Theorem~\ref{thm.tableau.jt}.

%%%%%%%%%%
\subsection{Interlacing MLQs}

Let us introduce some further notations now.

First, we define two binary relations $\succeq$ and $\gg$ on the powerset of $\ive{n}$:

\begin{itemize}
\item Given two subsets $A = \set{a_1 > \cdots > a_{\alpha}}$ and $B = \set{b_1 > \cdots > b_{\beta}}$ of $\ive{n}$, we say that \defn{$A\succeq B$} if and only if $\alpha =\beta$ and every $1 \leq k \leq \alpha$ satisfies $a_k \geq b_k$.

\item Given two subsets $A$ and $B$ of $\ive{n}$, we say that \defn{$A \gg B$} if and only if every $a \in A$ and $b \in B$ satisfy $a > b$. (For nonempty $A$ and $B$, this is equivalent to $\min A > \max B$).
\end{itemize}

For example, $\set{10,8} \gg \set{4,6,7} \succeq \set{4,5,7} \succeq \set{2,5,6}  \gg \set{1}$.
Note that $A \gg \emptyset$ and $\emptyset \gg A$ for any subset $A$ of $\ive{n}$ (for vacuous reasons), so that $\gg$ is not a partial order (but it becomes a partial order if we forbid $\emptyset$).

The following criterion (proof left to the reader) will be useful:

\begin{lemma}
\label{lem:determinant_form.gale1}
Let $A, B \subseteq \ive{n}$.
We have $A \succeq B$ if and only if there exists a bijection $\phi \colon B \to A$ satisfying $\phi(b) \geq b$ for each $b \in B$.
\end{lemma}

Fix some $\ell \in \NN$ and a type $\mm = \tup{m_1,m_2, \dotsc, m_{\ell}, 0, 0, \ldots}$ with $\leq \ell$ classes.
Assume that $m_i > 0$ for all $i \in \ive{\ell-1}$ (but $m_{\ell}$ may be $0$).

Our next few definitions concern MLQs.
Consider an (ordinary) MLQ $\qq = \tup{q_1,\ldots,q_{\ell-1}}$ of type $\mm$.
Thus, for each $i \in \ive{\ell-1}$, we have $\abs{q_i} = p_i(\mm) = m_1 + m_2 + \cdots + m_i$.
Hence, for each $i \in \ive{\ell-1}$, we can subdivide the set $q_i$ into $i$ blocks: the block containing the largest $m_1$ elements; the block containing the next-largest $m_2$ elements; and so on, until the block containing the smallest $m_i$ elements.
Denote these $i$ blocks by $q_i^{(1)}, q_i^{(2)}, \dotsc, q_i^{(i)}$, respectively.
Pictorially, we can thus write $q_i$ as
\[
\bigl\{ \underbrace{a_1 > \cdots > a_{p_1(\mm)}}_{=q_i^{(1)}} > \underbrace{a_{p_1(\mm)+1} > \cdots > a_{p_2(\mm)}}_{=q_i^{(2)}} > \cdots > \underbrace{a_{p_{i-1}(\mm)+1} > \cdots > a_{p_i(\mm)}}_{=q_i^{(i)}} \bigr\}.
\]
Thus, $q_i^{(1)}, q_i^{(2)}, \dotsc, q_i^{(i)}$ are pairwise disjoint nonempty subsets of $\ive{n}$ satisfying
\begin{subequations}
\label{eq.determinant_form.qij}
\begin{align}
\label{eq.determinant_form.qij.1}
  q_i  & = q_i^{(1)} \cup q_i^{(2)} \cup \cdots \cup q_i^{(i)}
\qquad \text{and}
\\ \label{eq.determinant_form.qij.2}
  q_i^{(1)} & \gg q_i^{(2)} \gg \cdots \gg q_i^{(i)}
\qquad \text{and}
\\ \label{eq.determinant_form.qij.3}
  \abs{q_i^{(j)}} & = m_j \text{ for all } j \in \ive{i}.
\end{align}
\end{subequations}
Thus, we have defined nonempty queues $q_i^{(j)} \subseteq \ive{n}$ for all $i \in \ive{\ell-1}$ and $j \in \ive{i}$ whenever $\qq = \tup{q_1, \dotsc, q_{\ell-1}}$ is an MLQ of type $\mm$.

\begin{example}
\label{ex:qij_generic}
Consider $n = 15$.
Let $\mm = \tup{4,2,2}$ and $\qq = \tup{q_1, q_2, q_3}$, where
\[
q_1 = \set{2,4,9,12},
\qquad
q_2 = \set{1,5,6,8,12,15},
\qquad
q_3 = \set{1,2,4,5,8,9,13,14}.
\]
Therefore, we have
\begin{gather*}
q_1^{(1)} = \set{2,4,9,12},
\\
q_2^{(2)} = \set{1,5},
\hspace{50pt}
q_2^{(1)} = \set{6,8,12,15},
\\
q_3^{(3)} = \set{1,2},
\hspace{50pt}
q_3^{(2)} = \set{4,5},
\hspace{50pt}
q_3^{(1)} = \set{8,9,13,14}.
\end{gather*}
Similar to the graveyard diagram~\eqref{eq:boxes-and-balls-1}, we define the graveyard diagram of an MLQ $\qq = \tup{q_1, q_2, \ldots, q_{\ell-1}}$ as a matrix with $\ell-1$ rows, whose entries are circles and squares; its row $i$ has each element $p \in q_i$ represented as a circle at position $p$ labeled by the letter $\bigl( q_i( \cdots q_1(1^n) \cdots ) \bigr)_p$ and each element of $\ive{n} \setminus q_i$ represented as an unlabeled square (\textit{i.e.}, we suppress filling the squares with an $i+1$).
The $\qq$ in this example is represented by the following graveyard diagram:
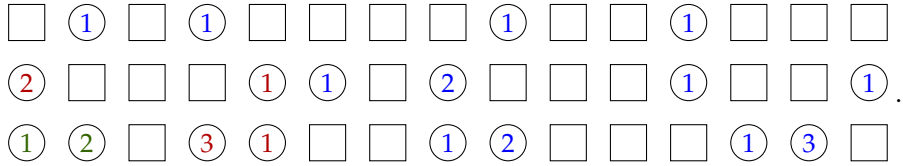
\[
\begin{tikzpicture}[scale=0.8,baseline=-30]
% q_1
\foreach \x in {2,4,9,12} {
    \node[blue] at (\x,0) {$1$};
    \draw (\x,0) circle (0.3);
}
\foreach \x in {1,3,5,6,7,8,10,11,13,14,15}
    \draw (\x-0.3, -0.3) rectangle (\x+0.3, +0.3);
% q_2
\foreach \x in {6,12,15} {
    \node[blue] at (\x,-1) {$1$};
    \draw (\x,-1) circle (0.3);
}
  \node[darkred] at (5,-1) {$1$};
  \draw (5,-1) circle (0.3);
  \node[darkred] at (1,-1) {$2$};
  \draw (1,-1) circle (0.3);
  \node[blue] at (8,-1) {$2$};
  \draw (8,-1) circle (0.3);
\foreach \x in {2,3,4,7,9,10,11,13,14}
    \draw (\x-0.3, -1-0.3) rectangle (\x+0.3, -1+0.3);
% q_3
  \node[dgreencolor] at (1,-2) {$1$};
  \draw (1,-2) circle (0.3);
  \node[dgreencolor] at (2,-2) {$2$};
  \draw (2,-2) circle (0.3);
  \node[darkred] at (4,-2) {$3$};
  \draw (4,-2) circle (0.3);
  \node[darkred] at (5,-2) {$1$};
  \draw (5,-2) circle (0.3);
\foreach \x in {8,13} {
    \node[blue] at (\x,-2) {$1$};
    \draw (\x,-2) circle (0.3);
}
  \node[blue] at (9,-2) {$2$};
  \draw (9,-2) circle (0.3);
  \node[blue] at (14,-2) {$3$};
  \draw (14,-2) circle (0.3);
\foreach \x in {3,6,7,10,11,12,15}
    \draw (\x-0.3, -2-0.3) rectangle (\x+0.3, -2+0.3);
\end{tikzpicture}\ .
\]
Note that the labeling is determined solely by the diagram and carries no additional data.
\end{example}

\begin{dfn}
We say that the MLQ $\qq = \tup{q_1, \dotsc, q_{\ell-1}}$ is \defn{interlacing} if each $i \in \set{2,3,\dotsc,\ell-1}$ satisfies
\begin{equation}
\label{eq.determinant_form.interlacing.def}
q_i^{(1)} \succeq q_{i-1}^{(1)} \gg
q_i^{(2)} \succeq q_{i-1}^{(2)} \gg
q_i^{(3)} \succeq q_{i-1}^{(3)} \gg \cdots \gg
q_i^{(i)}
\end{equation}
\end{dfn}

Note that we can rewrite~\eqref{eq.determinant_form.interlacing.def} as $q_i^{(j)} \succeq q_{i-1}^{(j)} \gg q_i^{(j+1)}$ for all $j \in \ive{i-1}$.

\begin{example}
Consider $n = 15$.
Let $\mm = \tup{3,2,4,0,\ldots}$ and $\qq = \tup{q_1, q_2, q_3}$, where
\[
q_1 = \set{9,12,13},
\qquad
q_2 = \set{7,8,11,12,14},
\qquad
q_3 = \set{1,3,5,6,8,10,11,14,15}.
\]
Therefore, we have
\begin{gather*}
q_1^{(1)} = \set{9,12,13},
\\
q_2^{(2)} = \set{7,8},
\hspace{50pt}
q_2^{(1)} = \set{11,12,14},
\\
q_3^{(3)} = \set{1,3,5,6},
\hspace{50pt}
q_3^{(2)} = \set{8,10},
\hspace{50pt}
q_3^{(1)} = \set{11,14,15}.
\end{gather*}
Thus $\qq$ is interlacing.
In this case, the elements of $q_i^{(j)}$ will be labeled by $j$ (as we shall see more generally in the proof of Lemma~\ref{lem:determinant_form.interl-act} below), and so the graveyard diagram of $\qq$ is
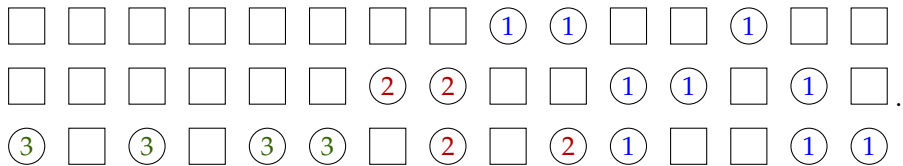
\[
\begin{tikzpicture}[scale=0.8,baseline=-30]
% q_1
\foreach \x in {9,10,13} {
    \node[blue] at (\x,0) {$1$};
    \draw (\x,0) circle (0.3);
}
\foreach \x in {1,...,8,11,12,14,15}
    \draw (\x-0.3, -0.3) rectangle (\x+0.3, +0.3);
% q_2
\foreach \x in {7,8} {
    \node[darkred] at (\x,-1) {$2$};
    \draw (\x,-1) circle (0.3);
}
\foreach \x in {11,12,14} {
    \node[blue] at (\x,-1) {$1$};
    \draw (\x,-1) circle (0.3);
}
\foreach \x in {1,2,3,4,5,6,9,10,13,15}
    \draw (\x-0.3, -1-0.3) rectangle (\x+0.3, -1+0.3);
% q_3
\foreach \x in {1,3,5,6} {
    \node[dgreencolor] at (\x,-2) {$3$};
    \draw (\x,-2) circle (0.3);
}
\foreach \x in {8,10} {
    \node[darkred] at (\x,-2) {$2$};
    \draw (\x,-2) circle (0.3);
}
\foreach \x in {11,14,15} {
    \node[blue] at (\x,-2) {$1$};
    \draw (\x,-2) circle (0.3);
}
\foreach \x in {2,4,7,9,12,13}
    \draw (\x-0.3, -2-0.3) rectangle (\x+0.3, -2+0.3);
\end{tikzpicture}\ .
\]
\end{example}

%Let $k = p_{\ell-1}(\mm) = m_1 + m_2 + \cdots + m_{\ell-1} = n - m_{\ell} \in \NN$.
Define a pseudo-partition
\begin{equation}
\label{eq.determinant_form.interlacing.lam}
\lambda^{\mm} := (
  \underbrace{1,1,\ldots,1}_{m_{\ell-1}\text{ times}},
  \underbrace{2,2,\ldots,2}_{m_{\ell-2}\text{ times}},
  \ldots,
  \underbrace{\ell-1,\ell-1,\ldots,\ell-1}_{m_1\text{ times}}
).
\end{equation}
Thus, the columns of the diagram $\diag{\lambda^{\mm}}$ are aligned to the bottom and have lengths $p_{\ell-1}(\mm), p_{\ell-2}(\mm), \dotsc, p_1(\mm)$ (from left to right).

For the following lemma, note that $\mm$ is fixed but $\qq$ is not.

\begin{lemma}
\label{lem:determinant_form.bij1}
Let $\ell$, $\mm$, and $m_i$ be as above.
Then, there is an injection
\[
P \colon \set{ \text{MLQs of type } \mm } \to \set{ \text{tableaux of shape } \lambda^{\mm} }
\]
with the following properties:
\begin{enumerate}
\item[(a)] We have $\wt\bigl(  P(\qq) \bigr) = \wt(\qq)$ for each MLQ $\qq$ of type $\mm$.

\item[(b)] If $\qq = \tup{q_1, q_2, \dotsc, q_{\ell-1}}$ is an MLQ of type $\mm$, then the surface of $P(\qq)$ is the list of all elements of $q_{\ell-1}$ in increasing order.

\item[(c)] The map $P$ restricts to a bijection
\[
\overline{P} \colon \set{\text{interlacing MLQs of type } \mm}  \to \set{  \text{semistandard tableaux of shape } \lambda^{\mm}}.
\]
\end{enumerate}
(In the border case $\ell = 1$, we interpret $q_0$ as the empty set whenever $\qq = \tup{q_1, q_2, \dotsc, q_{\ell-1}}$ is an MLQ.)
\end{lemma}

\begin{proof}
We define a map $P$ as follows.
Let $\qq = \tup{q_1, q_2, \dotsc, q_{\ell-1}}$ be any MLQ of type $\mm$.
Let $\widetilde{q}_i^{(j)}$ denote the entries of $q_i^{(j)}$ written vertically in a column, strictly increasing from top to bottom.
For example, if $q_i^{(j)} = \set{2,5,6}$, then
\[
\widetilde{q}_i^{(j)} = \raisebox{10pt}{$\begin{ytableau} 2 \\ 5 \\ 6 \end{ytableau}$}\ .
\]
We construct $P(\qq)$ as the following tableau of shape $\lambda^{\mm}$:
\begin{equation}
\label{pf.lem:determinant_form.bij1.visual}
\begin{array}[c]{r|c|ccc@{\hspace{30pt}}c@{\hspace{35pt}}c@{\hspace{35pt}}}
\cline{2-2}%
m_{\ell-1} \text{ rows } \left\{ \vphantom{\dfrac{2^2}{2^2}}\right.
    & \widetilde{q}_{\ell-1}^{(\ell-1)} &  &  &  &  & \\\cline{2-3}
m_{\ell-2} \text{ rows } \left\{ \vphantom{\dfrac{2^2}{2^2}}\right.
    & \widetilde{q}_{\ell-2}^{(\ell-2)} & \widetilde{q}_{\ell-1}^{(\ell-2)} & \multicolumn{1}{|c}{} &  &  & \\\cline{2-4}
m_{\ell-3} \text{ rows } \left\{  \vphantom{\dfrac{2^2}{2^2}}\right.
    & \widetilde{q}_{\ell-3}^{(\ell-3)} & \widetilde{q}_{\ell-2}^{(\ell-3)} & \multicolumn{1}{|c}{\widetilde{q}_{\ell-1}^{(  \ell-3)}} & \multicolumn{1}{|c}{} &  & \\\cline{2-5}%
\vdots & \mathbf{\vdots} & \mathbf{\vdots} & \multicolumn{1}{|c}{\vdots} & \multicolumn{1}{|c|}{\ddots} &  & \\\cline{2-6}
m_2 \text{ rows } \left\{  \vphantom{\dfrac{2^2}{2^2}}\right.
    & \widetilde{q}_2^{(2)} & \widetilde{q}_{3}^{(2)} & \multicolumn{1}{|c}{\widetilde{q}_{4}^{(2)}} & \multicolumn{1}{|c}{\cdots} & \multicolumn{1}{|c}{\widetilde{q}_{\ell-1}^{(2)}} & \multicolumn{1}{|c}{}\\\cline{2-7}
m_1 \text{ rows } \left\{ \vphantom{\dfrac{2^2}{2^2}}\right.
    & \widetilde{q}_1^{(1)} & \widetilde{q}_2^{(1)} & \multicolumn{1}{|c}{\widetilde{q}_{3}^{(1)}} & \multicolumn{1}{|c}{\cdots} & \multicolumn{1}{|c}{\widetilde{q}_{\ell-2}^{(1)}} & \multicolumn{1}{|c|}{\widetilde{q}_{\ell-1}^{(1)}}\\\cline{2-7}
\end{array}
\end{equation}
Formally speaking, this is the tableau of shape $\lambda^{\mm}$ whose $\ell-1$ columns are given as follows:
For each $j \in \ive{\ell-1}$, the $j$-th column consists of the columns $\widetilde{q}_{\ell-1}^{(\ell-j)}, \widetilde{q}_{\ell-2}^{(\ell-j-1)}, \dotsc, \widetilde{q}_j^{(1)}$ stacked atop each other (with $\widetilde{q}_{\ell-1}^{(\ell-j)}$ at the very top, $\widetilde{q}_{\ell-2}^{(\ell-j-1)}$ coming next under it, and so on).
Note that for each $i \in \ive{\ell-1}$, the parts $\widetilde{q}_i^{(i)}, \widetilde{q}_{i+1}^{(i)}, \dotsc, \widetilde{q}_{\ell-1}^{(i)}$ of columns $1,2,\ldots,\ell-i$ align with each other horizontally (and together form the topmost $m_i$ among the bottommost $m_1+m_2+\cdots+m_i$ rows of $P(\qq)$).

Thus, the map $P \colon \set{ \text{MLQs of type } \mm } \to \set{\text{tableaux of shape } \lambda^{\mm}}$ is defined.
This map $P$ is an injection, because the MLQ $\qq$ can be recovered
from the tableau $P(\qq)$ (indeed, each of the elements
of each of the queues of $\qq$ lands in some predictable cell of
$P(\qq)$).

We shall now prove the three properties we claimed about this injection $P$ to complete the proof.

Property~(a) is clear, since the entries of $q_1, q_2, \dotsc, q_{\ell-1}$ are in 1-to-1 correspondence with the entries of the tableau $P(\qq)$.

Furthermore, the surface of $P(\qq)$ is simply
\[
q_{\ell-1}^{(\ell-1)} \cup q_{\ell-1}^{(\ell-2)} \cup \cdots \cup q_{\ell-1}^{(2)} \cup q_{\ell-1}^{(1)} = q_{\ell-1}
\]
written in increasing order from top to bottom because of~\eqref{eq.determinant_form.qij.1} and~\eqref{eq.determinant_form.qij.2}.
In other words, the surface of $P(\qq)$ is the list of all elements of $q_{\ell-1}$ in increasing order.
This proves~(b).

Let $\qq$ be an MLQ of type $\mm$.
Recall that the columns $\widetilde{q}_i^{(j)}$ are strictly increasing.
Hence, from~\eqref{pf.lem:determinant_form.bij1.visual}, we see the following:
\begin{itemize}
\item The entries of the tableau $P(\qq)$ are weakly increasing along each row if and only if all $i \in \ive{\ell-1}$ and $j \in \ive{i-1}$ satisfy $q_i^{(j)} \succeq q_{i-1}^{(j)}$.
\item The entries of the tableau $P(\qq)$ are strictly increasing down each column if and only if all $i \in \ive{\ell-1}$ and $j \in \ive{i-1}$ satisfy $q_{i-1}^{(j)} \gg q_i^{(j+1)}$.
  (Here, we are also tacitly using the fact that the sets $q_i^{(j)}$ are nonempty.
  This ensures that the southern neighbor of a cell in $\widetilde{q}_i^{(j)}$ belongs either to $\widetilde{q}_i^{(j)}$ again or to $\widetilde{q}_{i-1}^{(j-1)}$, rather than (say) to $\widetilde{q}_{i-2}^{(j-2)}$.)
\end{itemize}
Combining these two observations, we conclude that the tableau $P(\qq)$ is semistandard if and only if the sets $q_i^{(j)}$ satisfy~\eqref{eq.determinant_form.interlacing.def}.
In other words, the tableau $P(\qq)$ is semistandard if and only if $\qq$ is interlacing.
Moreover, it is clear that given any semistandard tableau $T$ of shape $\lambda^{\mm}$, we can construct an interlacing MLQ $\qq$ of type $\mm$ satisfying $P(\qq) = T$.
(Namely, we can construct this $\qq$ by recovering the sets $q_i^{(j)}$ from the appropriate cells of $T$ in~\eqref{pf.lem:determinant_form.bij1.visual} and combining them to obtain queues $q_i$ and an MLQ $\qq$.)
Hence, the map $P$ restricted to interlacing MLQs is surjective onto the set of semistandard tableaux of shape $\lambda^{\mm}$, and since $P$ is injective, we have a bijection.
This proves~(c).
\end{proof}

\begin{cor}
\label{cor:determinant_form.bij1c}
Let $\ell$, $\mm$, and $m_i$ be as above.
Let $k \in \NN$ and $\lambda_1, \lambda_2, \ldots, \lambda_k$ be such that $\lambda^{\mm} = \tup{\lambda_1, \lambda_2, \dotsc, \lambda_k}$.
Let $S = \set{s_1 < s_2 < \cdots < s_k}$ be a $k$-queue.
Then,
\begin{align*}
&  \sum_{\substack{\qq=\tup{q_1,q_2,\ldots,q_{\ell-1}} \text{ is an interlacing}\\\text{MLQ of type $\mm$ with } q_{\ell-1}=S}} \wt(\qq) \\
&  = \left(  \prod_{i=1}^{k}x_{s_i}\right)
    \det\left(h_{\lambda_j-j+i-1}(x_1, x_2, \dotsc, x_{s_j}) \right)_{i, j \in \ive{k}} .
\end{align*}
\end{cor}

\begin{proof}
Consider the bijection $\overline{P}$ from Lemma~\ref{lem:determinant_form.bij1}(c), and recall the properties~(a) and~(b).
Hence, we can substitute $\overline{P}(\qq) = P(\qq)$ for $T$ in the sum on the left hand side of Theorem~\ref{thm.tableau.jt} (applied to $\lambda = \lambda^{\mm}$).
\begin{vershort}
We thus obtain
\begin{align*}
\sum_{T \in \SST(\lambda^{\mm}, s)} \wt(T)
  & = \sum_{\substack{\qq=\tup{q_1, q_2, \dotsc, q_{\ell-1}} \text{ is an interlacing}\\\text{MLQ of type } \mm \text{ such that} \\\text{the surface of }P(\qq) \text{ is } s}} \wt\bigl(  P(\qq) \bigr) \\
&  = \sum_{\substack{\qq = \tup{q_1,q_2,\dotsc,q_{\ell-1}} \text{ is an interlacing}\\\text{MLQ of type } \mm \text{ with } q_{\ell-1} = S}} \wt(\qq)
\end{align*}
where the second equality follows from Property~(a) and Property~(b) and that $s = S$ (where we consider $s$ as a set).
\end{vershort}
\begin{verlong}
We thus obtain
\begin{align*}
\sum_{T \in \SST(\lambda^{\mm}, s)} \wt(T)
  & = \sum_{\substack{\qq=\tup{q_1, q_2, \dotsc, q_{\ell-1}} \text{ is an interlacing}\\\text{MLQ of type } \mm \text{ such that} \\\text{the surface of } \overline{P}(\qq) \text{ is } s}} \wt\bigl(  P(\qq) \bigr) \\
&  = \sum_{\substack{\qq= \tup{q_1, q_2, \dotsc,q_{\ell-1}} \text{ is an interlacing}\\\text{MLQ of type } \mm \text{ such that}\\\text{the list of all elements of } q_{\ell-1} \text{ in} \\\text{increasing order is } s}} \wt(\qq) \\
&  = \sum_{\substack{\qq = \tup{q_1,q_2,\dotsc,q_{\ell-1}} \text{ is an interlacing}\\\text{MLQ of type } \mm \text{ with } q_{\ell-1} = S}} \wt(\qq)
\end{align*}
where the second equality follows from Property~(a) and Property~(b) and the last equality is since the list of all elements of $q_{\ell-1}$ in increasing order is $s$ if and only if $q_{\ell-1}=S$.
\end{verlong}
Therefore, the claim of Corollary~\ref{cor:determinant_form.bij1c} follows from Theorem~\ref{thm.tableau.jt}.
\end{proof}

Next, we shall connect the interlacingness of an MLQ with its action on the word $1^n$.
First, we need another notion:
If $u \in \mcW_n$ and $t \in \NN$, then we say that $u$ is \defn{weakly decreasing up to level $t$} if and only if every two sites $i < j$ in $\ive{n}$ satisfying $u_j \leq t$ must satisfy $u_i \geq u_j$.
Equivalently, $u$ is weakly decreasing up to level $t$ if and only if $u$ becomes weakly decreasing when all letters larger than $t$ are removed.
For example, the word ${\color{gray}5}4{\color{gray}55}4332{\color{gray}5}221{\color{gray}5}$ is weakly decreasing up to level $4$ (and up to any level $\leq 4$).
We also need the following auxiliary lemma about what queues can yield weakly decreasing words.

\begin{lemma}
\label{lem:determinant_form.interl-A}
Let $u \in \mcW_n$ and $t > 0$.
Let $q$ be a queue.
Assume that the word $q(u)$ has at least one letter equal to $t$, is weakly decreasing up to level $t$, and has exactly $\abs{q}$ letters that are at most $t$.
Then:

\begin{enumerate}
\item[(a)] The word $u$ is weakly decreasing up to level $t-1$.

\item[(b)] For each $h \in \ive{t-1}$, we have
\[
\set{ p \in \ive{n} \mid \bigl( q(u) \bigr)_p = h }
\succeq
\set{ p \in \ive{n} \mid u_p = h }
\gg
\set{ p \in \ive{n} \mid \bigl( q(u) \bigr)_p = h+1 } .
\]
\end{enumerate}
\end{lemma}

\begin{example}
\mbox{}
\begin{itemize}
\item Let $n = 9$, $u = {\color{gray}3}221{\color{gray}3}11{\color{gray}33}$, $t = 3$ and $q = \set{1, 2, 3, 5, 6, 7}$.
Then, $q(u) = 322{\color{gray}4}111{\color{gray}44}$ satisfies all assumptions of
Lemma~\ref{lem:determinant_form.interl-A}.
Thus, part (a) of the lemma says that $u$ is weakly decreasing up to level $2$ (which is evident).
% Part (b) of the lemma, applied to $h = 2$, says that $\set{ 2, 3 } \succeq \set{ 2, 3 } \gg \set{ 1 }$.
Part (b) of the lemma, applied to $h = 1$, says that
\[
\set{ 5, 6, 7 } \succeq \set{ 4, 6, 7 } \gg \set{ 2, 3 }.
\]

\item The assumption that $q(u)$ has at least one letter equal to $t$ cannot be removed from Lemma~\ref{lem:determinant_form.interl-A}.
Indeed, if $n = 4$, $u = {\color{gray}33}12$, $t = 3$ and $q = \set{1, 3}$, then $q(u) = 2{\color{gray}4}1{\color{gray}4}$ satisfies all assumptions except for this one, but the claim of Lemma~\ref{lem:determinant_form.interl-A}(a) does not hold.

\item The assumption that $q(u)$ is weakly decreasing up to level $t$ cannot be removed from Lemma~\ref{lem:determinant_form.interl-A}.
Indeed, if $n = 4$, $u = 12{\color{gray}33}$, $t = 3$ and $q = \set{1, 2, 3}$, then $q(u) = 123{\color{gray}4}$ satisfies all assumptions except for this one, but the claim of Lemma~\ref{lem:determinant_form.interl-A}(a) does not hold.

\item The assumption that $q(u)$ has exactly $\abs{q}$ letters that are at most $t$ cannot be removed from Lemma~\ref{lem:determinant_form.interl-A}.
Indeed, if $n = 5$, $u = {\color{gray}3}2112$, $t = 3$ and $q = \set{2, 3, 5}$, then $q(u) = 321{\color{gray}4}1$ satisfies all assumptions except for this one, but the claim of Lemma~\ref{lem:determinant_form.interl-A}(a) does not hold.
% We might be able to loosen this assumption to ``$q(u)$ has \emph{at most} $\abs{q}$ letters that are at most $t$'', at least as far as part (a) of the lemma is concerned. But not for part (b); the counterexample for that is $n = 4$, $u = 3312$, $t = 2$ and $q = \set{1, 2, 3}$.
\end{itemize}
\end{example}

\begin{proof}[Proof of Lemma~\ref{lem:determinant_form.interl-A}.]
Fix a permutation $\tup{i_1, i_2, \dotsc, i_n}$ of $\tup{1, 2, \dotsc, n}$ such that $u_{i_1}\leq u_{i_2}\leq \cdots \leq u_{i_n}$.
From the definition of $q(u)$ and since $q(u)$ has precisely $\abs{q}$ letters at most $t$,
we see that
all letters set during Phase~II of the algorithm computing $q(u)$ are at most $t$,
while all letters set during Phase~I are strictly greater than $t$
(see also Remark~\ref{rmk:t-splitting}).
Therefore, it is sufficient to consider the letters set during Phase~II.
We first show that no ``wrapping around the cycle'' can occur during Phase~II.
Let $j_{\kappa}$ denote the site set in the $\kappa$-th step of Phase~II
(that is, the site $j$ found at the time when $i = i_\kappa$).
Hence, $\set{j_1, j_2, \ldots, j_{\abs{q}}} = q$.

\begin{claim}
\label{claim:no_wrapping}
We have $j_{\kappa} \geq i_{\kappa}$ for all $\kappa \leq \abs{q}$ satisfying $u_{i_{\kappa}} \leq t-1$.
\end{claim}

\begin{subproof}
Assume the contrary.
Thus, there exists a $\kappa \leq \abs{q}$ satisfying $u_{i_{\kappa}} \leq t-1$ such that $j_{\kappa} < i_{\kappa}$ (\textit{i.e.}, the path of $u_{i_{\kappa}}$ ``wraps around the cycle'').
Without loss of generality, assume $\kappa$ is minimal.
Then either $j_{\kappa} = \min q$ or $\min q$ has already been set on a previous step of Phase~II, and in either case, we have $\bigl( q(u) \bigr)_{\min q} < t$.
Since $q(u)$ contains at least one $t$ (which is set during Phase~II), there exists an $h > \kappa$ such that $\bigl( q(u) \bigr)_{j_h} = t$ with $j_h > \min q$.
However, this contradicts that $q(u)$ is weakly decreasing up to level $t$.
\end{subproof}

Next, suppose Lemma~\ref{lem:determinant_form.interl-A}(a) does not hold.
Thus, there exists a $k$ such that $i_{k'} < i_k$ and $u_{i_{k'}} < u_{i_k} \leq t-1$ for some $k' < k$.
Without loss of generality, assume $k$ is minimal with this property.
Since $q(u)_{j_{\kappa}} = u_{i_{\kappa}}$ for all $\kappa \leq \abs{q}$,
we thus have $q(u)_{j_{k'}} < q(u)_{j_k} \leq t-1$.
From our assumption that $q(u)$ is weakly decreasing up to level $t$, we have $j_{k'} > j_k$.
Claim~\ref{claim:no_wrapping} shows that $j_k \geq i_k$.
Hence, $i_{k'} < i_k \leq j_k \leq j_{k'}$.
Thus, during Phase~II, at the time when $i = i_{k'}$, the first site $j$ weakly to the right of $i$ such that $j \in q$ and $(q(u))_j$ is not set cannot be $j_{k'}$ (because $j_k$ comes before it, and $(q(u))_{i_k}$ is still not set).
This contradicts the definition of $j_{k'}$.
This proves Lemma~\ref{lem:determinant_form.interl-A}(a).

Now we prove Lemma~\ref{lem:determinant_form.interl-A}(b).
Let $h \in \ive{t-1}$, and set
\begin{align*}
A & = \set{ p \in \ive{n} \mid \bigl( q(u) \bigr)_p = h },
\\
B & = \set{ p \in \ive{n} \mid u_p = h },
\\
C & = \set{ p \in \ive{n} \mid \bigl( q(u) \bigr)_p = h+1 }.
\end{align*}

We first show $A \succeq B$.
Define a map $\phi \colon B \to A$ by $\phi(i_k) = j_k$, where $i_k \in B$.
The map $\phi$ is well-defined (\textit{i.e.}, we have $j_k \in A$ for all applicable $k$) and bijective from the definition of Phase~II (see also Remark~\ref{rmk:t-splitting} for why $\abs{A} = \abs{B}$).
We have $\phi(i_k) = j_k \geq i_k$ by Claim~\ref{claim:no_wrapping}.
Hence $A \succeq B$ by Lemma~\ref{lem:determinant_form.gale1}.

To prove $B \gg C$, assume the contrary.
Thus, there exist $b \in B$ and $c \in C$ such that $b \leq c$.
Let $k$ be minimal such that $i_k \in B$ and $i_k \leq c$.
Since $q(u)$ is weakly decreasing up to level $t$ and $h < t$, we must have $A \gg C$.
Hence, we have $i_k \leq c < j_k$ since $j_k \in A$.
\begin{vershort}
However, this is a contradiction since $q(u)_c$ would instead have been set during the $k$-th step of Phase~II.
\end{vershort}
\begin{verlong}
But during Phase~II, at the time when $i = i_k$,
the letter $(q(u))_c$ is not set yet (since $(q(u))_c = h+1$ is
larger than $u_{i_k} = h$).
Thus, at this time, the first site $j$ weakly to the right of $i$ such that $j \in q$ and $(q(u))_j$ is not set cannot be $j_k$ (because $c$ comes before it, and $(q(u))_c$ is still not set).
This contradicts the definition of $j_k$.
\end{verlong}
This shows $B \gg C$.
\end{proof}

Now we prove our main technical lemma:

\begin{lemma}
\label{lem:determinant_form.interl-act}
Let $\ell$ be a positive integer.
Let $\mm = \tup{m_1, m_2, \dotsc, m_{\ell}, 0, 0, \ldots}$ be such that $m_i > 0$ for all $1 \leq i < \ell$.
Let $\qq = \tup{q_1, q_2, \dotsc, q_{\ell-1}}$ be an ordinary MLQ with $\ell-1$ queues.
Then, the following are equivalent:
\begin{enumerate}
\item[($\alpha$)] The word $\qq(1^n)$ has type $\mm$ and is weakly decreasing up to level $\ell-1$.
\item[($\beta$)] The MLQ $\qq$ has type $\mm$ and is interlacing.
\end{enumerate}
\end{lemma}

\begin{example}
\mbox{}
\begin{itemize}
\item We require that $m_i > 0$ for all $1 \leq i \leq \ell - 1$.
  To see why, consider $u = 2{\color{gray}4}1{\color{gray}4}$, which is of type $\mm = \tup{1,1,0,2,0,\ldots}$ and is weakly decreasing up to level $3$.
  However, the MLQ $\qq = \tup{\set{2}, \set{3,4}, \set{1,3}}$ is not interlacing but $\qq(1111) = u$:
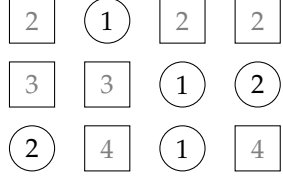
  \[
  \begin{tikzpicture}
  \def\ystep{-0.85};
  % q_1
  \node at (2,0) {$1$};
  \draw (2,0) circle (0.3);
  \foreach \x in {1,3,4} {
      \draw (\x-0.3, -0.3) rectangle (\x+0.3, +0.3);
      \node[color=gray] at (\x, 0) {$2$};
  }
  % q_2
  \node at (3,\ystep) {$1$};
  \draw (3,\ystep) circle (0.3);
  \node at (4,\ystep) {$2$};
  \draw (4,\ystep) circle (0.3);
  \foreach \x in {1,2} {
      \draw (\x-0.3, \ystep-0.3) rectangle (\x+0.3, \ystep+0.3);
      \node[color=gray] at (\x, \ystep) {$3$};
  }
  % q_3
  \node at (1,2*\ystep) {$2$};
  \draw (1,2*\ystep) circle (0.3);
  \node at (3,2*\ystep) {$1$};
  \draw (3,2*\ystep) circle (0.3);
  \foreach \x in {2,4} {
      \draw (\x-0.3, 2*\ystep-0.3) rectangle (\x+0.3, 2*\ystep+0.3);
      \node[color=gray] at (\x, 2*\ystep) {$4$};
  }
  \end{tikzpicture}
  \]
\item The interlacing of the MLQ (equivalently the weakly decreasing up to level $\ell - 1$) is necessary as seen by Example~\ref{ex:qij_generic}.
\end{itemize}
\end{example}

\begin{proof}[Proof of Lemma~\ref{lem:determinant_form.interl-act}.]
We prove the claim by induction on $\ell$.
The base case of $\ell = 1$ is trivial.
Thus, we assume the claim holds for $\ell$.
Let $\mm = \tup{m_1, m_2, \ldots, m_{\ell+1}, 0, \ldots}$ with $m_i > 0$ for all $1 \leq i < \ell + 1$.
Let $\qq = \tup{q_1, q_2, \dotsc, q_{\ell}}$ be an ordinary MLQ.
The word $u := \qq(1^n)$ has type $\mm$ if and only if $\qq$ has type $\mm$ by Lemma~\ref{lemma:mlq-type}.
Thus we can assume that both of these statements hold,
and it is remains to show that $u = \qq(1^n)$ is weakly decreasing up to level $\ell$ if and only if $\qq$ is interlacing.
We shall prove the $\Longrightarrow$ and $\Longleftarrow$ directions separately.

First, let us prepare.
Let $\qq' := \tup{q_1, \dotsc, q_{\ell-1}}$ and $u' := \qq'(1^n)$.
Thus, $u = \qq(1^n) = q_{\ell} (u')$.
Both the word $u'$ and the ordinary MLQ $\qq'$ have type
\[
\mm' := (m_1, m_2, \dotsc, m_{\ell-1}, n - p_{\ell-1}(\mm), 0, \ldots)
\]
since $\abs{q_{\ell}} = p_{\ell}(\mm)$.

Recall the definition of $q_i^{(h)}$ given by~\eqref{eq.determinant_form.qij}.

$\Longrightarrow:$
Suppose $u = \qq(1^n)$ is weakly decreasing up level $\ell$.
Thus, Lemma~\ref{lem:determinant_form.interl-A}(a) (applied to $u'$, $q_\ell$ and $\ell$ instead of $u$, $q$ and $t$) shows that $u'$ is weakly decreasing up to level $\ell-1$
(since $u = q_{\ell}(u')$ has type $\mm$ with $m_{\ell} > 0$
and $p_{\ell}(\mm) = \abs{q_{\ell}}$).
Hence, by our induction assumption, $\qq'$ is interlacing.
Since $u'$ has type $\mm'$ and is weakly decreasing up to level $\ell-1$, we see that $q_{\ell-1}^{(h)} = \{p \mid u'_p = h\}$ for all $h \in \ive{\ell-1}$.
Similarly, we have $q_\ell^{(h)} = \{p \mid u_p = h\}$ for all $h \in \ive{\ell}$.
Therefore, by Lemma~\ref{lem:determinant_form.interl-A}(b), we obtain $q_\ell^{(h)} \succeq q_{\ell-1}^{(h)} \gg q_\ell^{(h+1)}$ for each $h \in \ive{\ell-1}$.
Hence, $\qq$ is interlacing.

$\Longleftarrow:$
Suppose the MLQ $\qq$ is interlacing.
Thus, $\qq'$ is also interlacing.
Hence, by our induction assumption, the word $u' = \qq'(1^n)$ is weakly decreasing up to level $\ell - 1$;
recall that this word has type $\mm'$.
Hence, writing $u'$ as $u'_1 u'_2 \cdots u'_n$, we have
$q_{\ell-1}^{(h)} = \set{p \mid u'_p = h}$ for all $h \in \ive{\ell-1}$.
\begin{vershort}
We compute $u = q_{\ell}(u')$ by choosing the permutation such that for each $h$, we pick all letters $h$ in $u'$ from left to right.
Therefore, the interlacingness of $\qq$ shows that $q_{\ell}^{(h)} = \set{p \mid u_p = h}$ for all $h \in \ive{\ell}$.
Hence, the word $u$ is weakly decreasing up to level $\ell$.
\end{vershort}
\begin{verlong}

Now, we compute $u = q_{\ell} (u')$ using the algorithm constructing $q_{\ell} (u')$,
choosing the permutation
$\tup{i_1, i_2, \ldots, i_n}$ of $\tup{1, 2, \ldots, n}$
in such a way that $(u_{i_1}, 1) \leq (u_{i_2}, 2) \leq \cdots \leq (u_{i_n}, n)$
in lexicographic order
(thus, in Phase~II, we pick our sites $i$ in order of increasing letters in $u$,
while breaking ties by moving left to right).
For each $k \in \ive{n}$, we let $j_k$ be the site $j$ found at the step at which $i = i_k$.

The interlacingness relation $q_{\ell}^{(h)} \succeq q_{\ell-1}^{(h)} \gg q_{\ell}^{(h+1)}$,
combined with the fact that $q_{\ell-1}^{(h)} = \{p \mid u'_p = h\}$ for all $h \in \ive{\ell-1}$,
reveals that throughout Phase~II of our algorithm,
whenever the site $i_k$ belongs to $q_{\ell-1}^{(h)}$,
the corresponding site $j_k$ will belong to $q_{\ell}^{(h)}$
(and, more precisely, as the $i$ in our algorithm runs through the
sites in $q_{\ell-1}^{(h)}$ from left to right,
the $j$ will run through the sites in $q_{\ell}^{(h)}$ from left to right);
indeed, the $q_{\ell}^{(h)} \succeq q_{\ell-1}^{(h)}$ part guarantees that
each site in $q_{\ell-1}^{(h)}$ is reachable from the corresponding
site in $q_{\ell}^{(h)}$ by moving right, whereas the
$q_{\ell-1}^{(h)} \gg q_{\ell}^{(h+1)}$ part ensures that the search for $j$
will actually reach this site
(rather than already stopping at some other site in $q_{\ell}$ further left of it).
(Strictly speaking, this argument is a strong induction on $k$, as we are
using that all the previous sites have been correctly dispatched.)
Hence, the converse also holds:
If $j_k \in q_{\ell}^{(h)}$, then $i_k \in q_{\ell-1}^{(h)}$
(since $q_{\ell}^{(h)} \succeq q_{\ell-1}^{(h)}$ yields
$\abs{q_{\ell}^{(h)}} = \abs{q_{\ell-1}^{(h)}}$).

As a consequence, we have $u_p = h$ for each $p \in q_{\ell}^{(h)}$ for each $h \in \ive{\ell-1}$
(since $p = j_k$ for some $k \in \ive{n}$ found during Phase~II of our
algorithm, and thus
$u_p = u_{j_k} = ( q_{\ell}(u') )_{j_k} = u'_{i_k} = h$ since the corresponding
$i_k$ lies in $q_{\ell-1}^{(h)} = \{p \mid u'_p = h\}$).
Hence, $q_{\ell}^{(h)} = \{p \mid u_p = h\}$ for each $h \in \ive{\ell-1}$, and thus also for $h = \ell$
(by the exclusion principle: the letters of $u$ that are $\leq \ell$ appear in the positions $p \in q_{\ell}$).
Therefore, $u$ is weakly decreasing up to level $\ell$.
\end{verlong}
\end{proof}

%%%%%%%%%%
\subsection{Proof of Theorem~\ref{thm:determinant_form}}

We are now ready to prove Theorem~\ref{thm:determinant_form}.

\begin{proof}[Proof of Theorem~\ref{thm:determinant_form}.]
Write the $r$-tuple $\tup{v_1, v_2, \dotsc, v_r}$ in the form
\begin{equation}
\label{pf.thm:determinant_form.v=}
\tup{v_1, v_2, \dotsc, v_r} = (
  \underbrace{\ell-1,\ldots,\ell-1}_{m_{\ell-1}\text{ times}},
  \underbrace{\ell-2,\ldots,\ell-2}_{m_{\ell-2}\text{ times}},
  \dotsc,
  \underbrace{1,\ldots,1}_{m_1\text{ times}}
)
\end{equation}
for some $m_1, m_2, \ldots, m_{\ell-1} > 0$
(we can do this, since $\set{v_1, v_2, \dotsc, v_r} = \ive{\ell-1}$ and $v_1 \geq v_2 \geq \cdots \geq v_r$).
\begin{verlong}
Thus, the first $m_{\ell-1}$ of the entries of the $r$-tuple $\tup{v_1 \geq v_2 \geq \cdots \geq v_r}$ equal $\ell-1$; the next $m_{\ell-2}$ of its entries equal $\ell-2$; the next $m_{\ell-3}$ of its entries equal $\ell-3$; and so on.
\end{verlong}
Also set $m_{\ell} = n - r$, and define a type $\mm = \tup{m_1, m_2, \dotsc, m_{\ell}, 0, 0, \ldots}$.
Then, the word $u(v)$ has type $\mm$
\begin{vershort}
by the construction of $u(v)$.
\end{vershort}
\begin{verlong}
(since the entries of $u(v)$ are precisely $v_1, v_2, \dotsc, v_r$ and also $n-r$ entries equal to $\ell$).
\end{verlong}
Additionally, the word $u(v)$ is packed (since $m_1, m_2, \ldots, m_{\ell-1} > 0$).
Furthermore, the definition of $\mm$ shows that $r = m_1 + m_2 + \cdots + m_{\ell-1} = p_{\ell-1}(\mm)$.
\begin{verlong}
Each $i \in \ive{\ell}$ satisfies
\[
m_i =
  \begin{cases}
  \abs{ \set{ j \in \ive{r} \mid v_j = i}} & \text{if } i < \ell,\\
  n - r & \text{if } i = \ell
  \end{cases}
\]
(by~\eqref{pf.thm:determinant_form.v=} and the definition of $m_{\ell}$).
\end{verlong}

%First note that for any set $B = \set{ b_1 < b_2 < \cdots < b_r }  \subseteq \ive{n}$, weakly decreasing $r$-tuple $\tup{v_1 \geq v_2 \geq \dotsm \geq v_r}$ such that $\set{ v_1, v_2, \dotsc, v_r } = \ive{\ell-1}$, the word $u \in \mcW_n$ defined by $u_i = v_j$ if $i = b_j$ for some $j$, otherwise $u_i = \ell$ is weakly decreasing up to level $\ell-1$.
%Conversely, for any word $u = u_1 \dotsm u_n \in \mcW_n$ that is weakly decreasing up to level $\ell-1$ such that $\ive{\ell-1} \subseteq \set{u_1, u_2, \dotsc, u_n}$, we can construct the set $B = \set{ b_1 < b_2 < \cdots < b_r \mid u_{b_k} < \ell} \subseteq \ive{n}$ and the $r$-tuple $\tup{u_{b_1} \geq u_{b_2} \geq \cdots \geq u_{b_r}}$.
\begin{vershort}
The word $u(v)$ is weakly decreasing up to level $\ell-1$.
\end{vershort}
\begin{verlong}
The word $u(v)$ is weakly decreasing up to level $\ell-1$ since $v_1 \geq v_2 \geq \cdots \geq v_r$ and $\set{v_1, v_2, \ldots, v_r} = \ive{\ell-1}$.
\end{verlong}
Moreover, the letters $\leq \ell-1$ of this word lie in the positions $j \in B$.
Therefore, for an MLQ $\qq = \tup{q_1, \ldots, q_{\ell-1}}$ of type $\mm$, the following are equivalent:
\begin{enumerate}
\item We have $u(v) = \qq(1^n)$.
\item The word $\qq(1^n)$ has type $\mm$ and is weakly decreasing up to level $\ell-1$ and its letters $\leq \ell-1$ lie in the positions $j \in B$.
\item The word $\qq(1^n)$ has type $\mm$ and is weakly decreasing up to level $\ell-1$ and we have $q_{\ell-1} = B$.
\item The MLQ $\qq$ has type $\mm$ and is interlacing, and we have $q_{\ell-1} = B$.
\end{enumerate}
It is straightforward to see $(1) \Longleftrightarrow (2) \Longleftrightarrow (3)$, and we have $(3) \Longleftrightarrow (4)$ by Lemma~\ref{lem:determinant_form.interl-act}.

Define the pseudo-partition $\lambda^{\mm}$ by~\eqref{eq.determinant_form.interlacing.lam}.
% This pseudo-partition $\lambda^{\mm}$ has $p_{\ell-1}(\mm) = r$ entries.
Comparing~\eqref{eq.determinant_form.interlacing.lam} with~\eqref{pf.thm:determinant_form.v=}, we obtain
\[
\lambda^{\mm} = \tup{\ell-v_1, \ell-v_2, \dotsc, \ell-v_r}
= \tup{\gamma_1, \gamma_2, \ldots, \gamma_r}
\]
(since $\ell - v_j = \gamma_j$ for all $j \in \ive{r}$).

The definition of $\swt{u(v)}$ yields
\begin{align*}
\swt{u(v)} & = \sum_{\substack{\qq \text{ is an MLQ of type } \mm; \\ u(v)=\qq(1^n)}}\wt(\qq) \\
%&  = \sum_{\substack{\qq = \tup{q_1, q_2, \dotsc, q_{\ell-1}} \text{ is an MLQ of type } \mm; \\ \qq \text{ is interlacing and satisfies } q_{\ell-1} = B}} \wt(\qq) \qquad \text{(by Lemma~\ref{lem:determinant_form.interl-act})} \\
&  = \sum_{\substack{\qq = \tup{q_1, q_2, \dotsc, q_{\ell-1}} \text{ is an interlacing} \\\text{MLQ of type } \mm \text{ with } q_{\ell-1}=B}} \wt(\qq) \qquad \text{(by $(1) \Longleftrightarrow (4)$)} \\
&  = \left(  \prod_{i=1}^r x_{b_i} \right) \det\biggl( h_{\gamma_j-j+i-1}\left(x_1,x_2,\ldots,x_{b_j}\right)  \biggr)_{i, j \in \ive{r}} ,
\end{align*}
where the last equality is from Corollary~\ref{cor:determinant_form.bij1c} (applied to $r$, $\gamma_i$, $B$ and $b_i$ instead of $k$, $\lambda_i$, $S$ and $s_i$).
\begin{verlong}
Hence the claim follows since $\prod_{i=1}^r x_{b_i} = \prod_{b \in B} x_b$.
\end{verlong}
\end{proof}

%=====================================================================
\section{Proof of Theorem~\ref{thm:permutation}}
\label{sec:thm_proof}

In this section, we shall prove Theorem~\ref{thm:permutation} by relying on the
braid relation for dual configurations. We need some definitions first.

An \defn{$(r_1,r_2)$-configuration} shall mean a pair $C = (q_1, q_2)$, where $q_1$ is an $r_1$-queue and $q_2$ is an $r_2$-queue.
As usual, we consider $C$ as a function on words by $C(u) := q_2\bigr(q_1(u)\bigr)$, and we define the weight of $C$ by $\wt(C) := \wt(q_1) \wt(q_2)$.
We construct the \defn{dual}\footnote{This is a different duality than the contragredient duality of Lemma~\ref{le:dual}.} $(r_2,r_1)$-configuration to $C$, which we denote by $C'$, as follows.

We begin by constructing a sequence of parentheses as follows: For each $j = 1, 2, \dotsc, n$ (in that order), we write
\begin{itemize}
\item an opening parenthesis `$($' if $j \in q_1$ and $j \notin q_2$;
\item a closing parenthesis `$)$' if $j \notin q_1$ and $j \in q_2$;
\item a matched pair of parentheses `$()$' if $j \in q_1$ and $j \in q_2$;
\item nothing otherwise.
\end{itemize}
We say that these parentheses are \defn{contributed} by the site $j$.
Next, we match as many parentheses as possible following the \defn{standard parenthesis-matching algorithm}:
every time you find an opening parenthesis to the left of a closing one, with only frozen parentheses between them, you match these two parentheses and declare them frozen.
Here, we understand our sequence of parentheses as being written on a circle -- so the last opening parenthesis can be matched with the first closing parenthesis if all parentheses ``between them'' (\textit{i.e.}, to the right of the former and to the left of the latter) are frozen.

At the end of this algorithm, there will be exactly $\abs{r_1 - r_2}$ parentheses left unmatched; these unmatched parentheses will be opening if $r_1 > r_2$ and closing if $r_1 < r_2$. The unmatched parentheses will not depend on the order in which we perform the matching.
This can be easily seen by considering the parentheses as an infinite shifted-periodic Motzkin path (see Figure~\ref{fig:balanced_motzkin}), where a ``nothing'' corresponds to a flat step.
If $r_1 > r_2$, then the unmatched positions correspond to down steps in the Motzkin path such that the path never returns to the same height.
When $r_1 < r_2$, the unmatched positions correspond to the up steps where the Motzkin path first obtains a certain height.

We define \defn{$\SP(q_1, q_2)$} to be the following data:
\begin{itemize}
\item the sequence of parentheses obtained from $q_1$ and $q_2$;
\item the matching obtained by the algorithm; and
\item the correspondence between sites and \emph{closing} parentheses (\textit{i.e.}, which site contributes which closing parenthesis).
\end{itemize}

Note that each $j \in \ZZ/n\ZZ$ that belongs to both $q_1$ and $q_2$ contributes a pair of parentheses `$()$`, which can be immediately matched (to each other) at the beginning of the algorithm.
Thus, if we disregard these pairs of parentheses, the outcome of the algorithm does not change (apart from these pairs disappearing).
Hence, if we skip the sites $j \in q_1 \cap q_2$ in the definition of our sequence of parentheses, then the outcome of the algorithm (specifically, the set of unmatched parentheses) will be the same.
The sequence of parentheses obtained as above, but skipping these sites $j$, will be called the \defn{reduced parenthesis sequence}.

%If a site $j \in \ZZ / n \ZZ$ contributes a matched parenthesis or two matched parentheses or no parenthesis at all, we say $j$ is \defn{balanced}; otherwise we say $j$ is \defn{unbalanced}.
A site $j \in \ZZ/n\ZZ$ is \defn{unbalanced} if it contributes an unmatched parenthesis; otherwise we say $j$ is \defn{balanced}.
There are exactly $\abs{r_1 - r_2}$ unbalanced sites.

We construct $C' = (q'_1, q'_2)$ as follows.
For balanced $j$, we have $j \in q'_i$ if and only if $j \in q_i$ for $i=1,2$.
For unbalanced $j$, we have $j \in q'_i$ if and only if $j \in q_{3-i}$ for $i = 1,2$.
Note that $C$ and $C'$ have the same balanced sites, since the parentheses that were matched for $C$ remain matched in $C'$ and the remaining parentheses are all of one kind.
Hence, we have $C'' = C$. Also, we have $\wt(C) = \wt(C')$.

\begin{example}
\label{ex:parentheses_form}
Consider the configuration
\[
C = (q_1, q_2) = (\{1,2,5,6,8,11,13,14,17,18,19\}, \{2,12,15,16,18,19,20\})
\]
which is given pictorially in Figure~\ref{fig:balanced}.
Converting $C$ to a parenthesis sequence, we obtain
\[
( \quad () \quad \cdot \quad \cdot \quad ( \quad ( \quad \cdot \quad ( \quad \cdot \quad \cdot \quad ( \quad ) \quad ( \quad ( \quad ) \quad ) \quad ( \quad () \quad () \quad ),
\]
which results in 4 unpaired $($'s from positions $\{1,5,6,8\}$.
Therefore, we have $q_1' = q_1 \setminus \{1,5,6,8\}$ and $q_2' = q_2 \cup \{1,5,6,8\}$.
Alternatively, the dual configuration $C'$ is given by sliding all of the circles not in a shaded box from the upper level to the lower level.
\end{example}

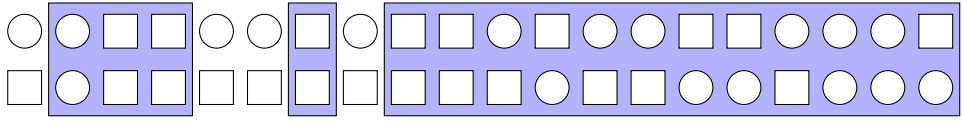
\begin{figure}[t]
\[
\begin{tikzpicture}[scale=0.75]
  \def\sc{0.85}   % Change this to adjust the x-scaling
  \def\ll{2}   % level 2
  \def\l{1}   % level 1
  \draw[fill=blue!30] (1.5*\sc,\l-.5) rectangle(4.5*\sc,\ll+.5);
  \draw[fill=blue!30] (6.5*\sc,\l-.5) rectangle(7.5*\sc,\ll+.5);
  \draw[fill=blue!30] (8.5*\sc,\l-.5) rectangle(20.5*\sc,\ll+.5);
  \foreach \i in {1,2,5,6,8,11,13,14,17,18,19} { \draw[fill=white] (\i*\sc,\ll) circle (0.3); }
  \foreach \i in {3,4,7,9,10,12,15,16,20} { \draw[fill=white] (\i*\sc-.3,\ll-.3) rectangle +(0.6,+0.6); }
  \foreach \i in {2,12,15,16,18,19,20} { \draw[fill=white] (\i*\sc,\l) circle (0.3); }
  \foreach \i in {1,3,4,5,6,7,8,9,10,11,13,14,17} { \draw[fill=white] (\i*\sc-.3,\l-.3) rectangle +(0.6,+0.6); }
\end{tikzpicture}
\]
\caption{We draw a $\bigcirc$ in position $i$ in row $j$ corresponding to $i \in q_j$ and a $\square$ if $i \notin q_j$.
The shaded boxes mark the balanced sites.}
\label{fig:balanced}
\end{figure}

\begin{figure}[t]
\[
\begin{tikzpicture}[scale=0.45]
\draw[-,thin,red!50] (0,0) -- (24,0);
\draw[-,thin,red!50] (1,1) -- (24,1);
\draw[-,thin,red!50] (8,2) -- (24,2);
\draw[-,thin,red!50] (9,3) -- (24,3);
\draw[-,thin,red!50] (11,4) -- (24,4);
\draw[-,thin,red!50] (14,5) -- (16,5);
\draw[-,thin,red!50] (18,5) -- (22,5);
\draw[-,thick,dashed] (-1,1) -- (0,0);
\draw[-,dotted] (1,0) -- (1,1);
\draw[-,dotted] (3,0) -- (3,1);
\draw[-,dotted] (4,0) -- (4,1);
\draw[-,dotted] (5,0) -- (5,1);
\draw[-,dotted] (6,0) -- (6,2);
\draw[-,dotted] (7,0) -- (7,2);
\draw[-,dotted] (8,0) -- (8,2);
\draw[-,dotted] (9,0) -- (9,3);
\draw[-,dotted] (10,0) -- (10,4);
\draw[-,dotted] (11,0) -- (11,4);
\draw[-,dotted] (12,0) -- (12,5);
\draw[-,dotted] (13,0) -- (13,4);
\draw[-,dotted] (14,0) -- (14,5);
\draw[-,dotted] (15,0) -- (15,6);
\draw[-,dotted] (16,0) -- (16,5);
\draw[-,dotted] (17,0) -- (17,4);
\draw[-,dotted] (18,0) -- (18,5);
\draw[-,dotted] (20,0) -- (20,5);
\draw[-,dotted] (22,0) -- (22,5);
\draw[-,dotted] (23,0) -- (23,4);
\draw[-,thick] (0,0) -- (1,1) -- (2,2) -- (3,1) -- (4,1) -- (5,1) -- (6,2) -- (7,2) -- (8,2) -- (9,3) -- (10,4) -- (11,4) -- (12,5) -- (13,4) -- (14,5) -- (15,6) -- (16,5) -- (17,4) -- (18,5) -- (19,6) -- (20,5) -- (21,6) -- (22,5) -- (23,4);
\draw[-,thick,dashed] (23,4) -- (24,5);
\fill[blue] (0,0) circle (4pt);
\fill[blue] (5,1) circle (4pt);
\fill[blue] (8,2) circle (4pt);
\fill[blue] (9,3) circle (4pt);
\fill[blue] (23,4) circle (4pt);
\end{tikzpicture}
\]
\caption{The Motzkin path corresponding to the configuration in Figure~\ref{fig:balanced}.
The corresponding unbalanced up-steps are marked by dots.}
\label{fig:balanced_motzkin}
\end{figure}
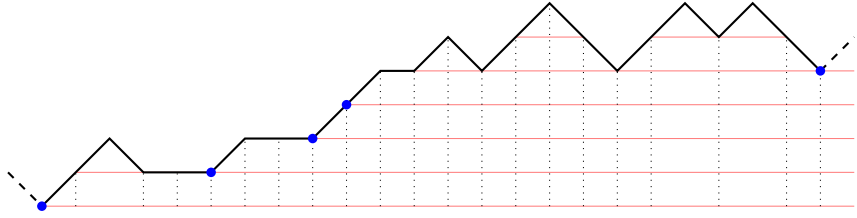

Recall the notations introduced just before Lemma~\ref{le:dual}.

\begin{remark}
\label{rmk:balanced-dual-let}
Let $C = (q_1, q_2)$ be an $(r_1, r_2)$-configuration.
Thus, the dual configuration of the $(n-r_1, n-r_2)$-configuration $C^* = \tup{q_1^*, q_2^*}$ is obtained from the dual configuration $\tup{q_1', q_2'}$ of $C$ by
\begin{equation}
 (C^*)' = \bigl( (q_1')^*, (q_2')^* \bigr) .
 \label{eq.rmk:balanced-dual-let.dual}
\end{equation}
To see this, we compute the dual configurations of both $C$ and $C^*$ using the reduced parenthesis sequences. The sequence for $C^*$ is obtained from that for $C$ by reflecting both the parentheses (\textit{i.e.}, opening become closing and vice versa) and the sequence itself. Thus, the parenthesis matching algorithm proceeds in the same way.

In addition, applying Lemma~\ref{le:dual} twice, we obtain
\[
C(u) = q_2\bigl( q_1(u) \bigr) = q_2\bigl( (q^*_1(u^*))^* \bigr) = \bigl( q_2^*\bigl( q_1^*(u^*) \bigr) \bigr)^* = \bigl( C^*(u^*) \bigr)^* .
\]
%In particular, for $u \in \set{1,2}^n$, we have $C(u)_i = 5 - C^*(u^*)_{n+1-i}$, where we treat $u$ and $C^*(u^*)$ as a word with $2$ and $4$ classes, respectively.
\end{remark}

Fix $k \geq 1$.
In the following, we simplify our terminology and say that an MLQ
is a $k$-tuple of queues (without any restriction on their sizes).
We want to define an action of $\SymGp{k}$ on MLQs.
%by letting each simple transposition $s_i$ act as the map
For each $i \in \ive{k-1}$, we define a map $\fraks_i \colon \set{\text{MLQs}} \to \set{\text{MLQs}}$ by
\[
\fraks_i(q_1, q_2, \dotsc, q_k) = (q_1, \dotsc, q_{i-1}, q'_i, q'_{i+1}, q_{i+2}, \dotsc, q_k),
\]
where $\tup{q'_i, q'_{i+1}}$ is the dual configuration of $\tup{q_i, q_{i+1}}$.
From the definition of a dual configuration, it is clear that $\fraks_i \fraks_i \qq = \qq$.
It is also clear from the definition that $\fraks_i \fraks_j \qq = \fraks_j \fraks_i \qq$ if $\abs{i - j} > 1$.
Thus, the following proposition shows that $\fraks_i$ defines an action of $\SymGp{k}$ on the set of all MLQs.

\begin{prop} \label{prop:braid}
We have
\[
\fraks_i \fraks_{i+1} \fraks_i \qq
	   = \fraks_{i+1} \fraks_i \fraks_{i+1} \qq
\]
for any MLQ $\qq = \tup{q_1, \dotsc, q_k}$ and any $i \in \set{1, 2, \ldots, k-2}$.
\end{prop}

\begin{proof}
We shall deduce the claim from~\cite[Ch.~5, (5.6.3)]{Loth}.

Let $A$ be the $\tup{k+1}$-element set $\set{\circ, 1, 2, \ldots, k}$.
Let $A^*$ denote the set of all words on the alphabet $A$
(of any finite length).

We construct a $k \times n$-matrix $M_{\qq} \in A^{k \times n}$ from $\qq$ by setting the $\tup{i, j}$-th
entry to $i$ if $j \in q_i$ and $\circ$ otherwise.
We then construct a word $\word(\qq) \in A^*$ by reading $M_{\qq}$ from top-to-bottom, left-to-right (\textit{i.e.}, column by column).
For example, if $n = 5$, $k = 3$, then
\begin{align*}
\qq = \tup{\set{1, 3}, \set{2},
\set{2, 5}}
& \longleftrightarrow
 M_{\qq}
 =
 \begin{array}{ccccc}
  1 & \circ & 1 & \circ & \circ \\
  \circ & 2 & \circ  & \circ & \circ \\
  \circ & 3  & \circ & \circ & 3
 \end{array}
 \\ & \longrightarrow
 \word(\qq) = 1 \circ \circ \circ 2 3 1 \circ \circ \circ \circ \circ \circ \circ 3 .
\end{align*}
Clearly, an MLQ $\qq$ is uniquely determined by $\word(\qq)$ since $n$ is fixed.
In other words, the map $\word \colon \set{\text{MLQs}} \to A^*$
is injective.

Following~\cite[\S5.5]{Loth}, for each $i \in \set{1, 2, \ldots, k-1}$, we give an operator
$\sigma_i \colon A^* \to A^*$.
This operator $\sigma_i$ acts on a word $p \in A^*$ % = \tup{p_1, p_2, \ldots, p_\ell}$
as follows:
\begin{enumerate}
 \item Treat all letters $i$ in $p$ as opening parentheses `$($',
       all letters $i+1$ as closing parentheses `$)$',
       and consider all other letters to be frozen.
       Now, match as many parentheses as possible
       according to the standard parenthesis-matching algorithm,
       but treating the word as a word in the usual sense
       (\textit{i.e.}, not written on a circle, but having a beginning
       and an end).
       The result is independent of the choices in the algorithm,
       and is always a word whose non-frozen
       part (\textit{i.e.}, the word obtained by removing
       all frozen letters) is
       \begin{equation}
       \label{eq:reduced_parens}
       \underbrace{))\cdots)}_{a\text{ parentheses}}
       \underbrace{((\cdots(}_{b\text{ parentheses}}
       \end{equation}
       for some integers $a, b \geq 0$.
 \item Now, replace the non-frozen part~\eqref{eq:reduced_parens} by
       \[
       \underbrace{))\cdots)}_{b\text{ parentheses}}
       \underbrace{((\cdots(}_{a\text{ parentheses}}
       \]
       while keeping all frozen letters in their places.
       The resulting word is $\sigma_i p$.
\end{enumerate}
From~\cite[Eq.~(5.6.3)]{Loth}, these operators
$\sigma_i$ satisfy
\begin{equation}
 \sigma_i \sigma_{i+1} \sigma_i
 = \sigma_{i+1} \sigma_i \sigma_{i+1}
 \label{pf.prop:braid.loth-eq}
\end{equation}
for all $i \in \set{1, 2, \ldots, k-2}$.
(See Remark~\ref{rmk:braid.leeuwen} for another reference for this identity.)

Let $\zeta \colon \mcW_n \to \mcW_n$ be the cyclic shift map that sends each word $w_1 w_2 \cdots w_n$ to $w_2 w_3 \cdots w_n w_1$.
We also abuse the notation $\zeta$ for the map that sends each queue $q$ to the queue $\zeta q = \set{ i - 1 \mid i \in q }$ (recall that $0 = n$ as sites).
This map $\zeta$ shall act on MLQs entrywise (since an MLQ is a tuple of queues).
Clearly,
\begin{equation}
 \word(\zeta \qq) = \zeta^k \word(\qq)
 \label{pf.prop:braid.word-zeta}
\end{equation}
for any MLQ $\qq = (q_1, \ldots, q_k)$.

Now, we claim that
\begin{equation}
 \word(\fraks_i \qq) = \sigma_i\bigl( \word(\qq) \bigr)
 \qquad \text{ for each MLQ } \qq \text{ and each } i .
 \label{pf.prop:braid.inter}
\end{equation}
Note that it is sufficient to show that $\word(\fraks_1 \qq) = \sigma_1\bigl( \word(\qq) \bigr)$ for $\qq = (q_1, q_2)$
(because the definition of $\sigma_i$ only relies on the letters $i$ and $i+1$, while all other letters stay in their places and have no effect).

Thus, let $\qq = (q_1, q_2)$.
We want to show $\word(\fraks_1 \qq) = \sigma_1\bigl( \word(\qq) \bigr)$.
If $\abs{q_1} = \abs{q_2}$, then $\fraks_1 \qq = \qq$ by the definition of $\fraks_1$.
Moreover, we have $\sigma_1\bigl( \word(\qq) \bigr) = \word(\qq)$ in this case, since the word $\word(\qq)$ has as many letters $1$ as it has letters $2$, but the map $\sigma_1$ leaves such words unchanged.
Hence, the claim holds when $\abs{q_1} = \abs{q_2}$.
Thus, we assume that $\abs{q_1} \neq \abs{q_2}$.
Therefore, there exists at least one unbalanced site for the configuration $\qq = (q_1, q_2)$.

The operator $\fraks_1$ commutes with the cyclic shift map $\zeta$ on MLQs because the standard parenthesis-matching algorithm used in the definition of dual configurations is clearly invariant under cyclic shift.
The operator $\sigma_1$ commutes with the cyclic shift map $\zeta$ on words in $A^*$ by~\cite[Prop.~5.6.1]{Loth}.
Hence, and because of~\eqref{pf.prop:braid.word-zeta}, we can apply $\zeta$ to $\qq$ any number of times without loss of generality.
By our assumption that at least one unbalanced site $j$ exists, we can thus assume without loss of generality that the site $1$ is unbalanced for the configuration $\qq = (q_1, q_2)$, since otherwise we can cyclically shift site $j$ until it is $1$.
Thus, in the standard parenthesis-matching algorithm used in the construction of the dual configuration $\fraks_1 \qq$, the site $1$ induces a parenthesis that stays unmatched throughout the algorithm.
Hence, no two parentheses that get matched to each other during this algorithm have the parenthesis from site $1$ lying between them.
Therefore, it does not matter whether we regard the sequence of parentheses as written on a cycle or on a straight line (as the wrapping-around is not used in the matching process).
Hence, the standard parenthesis-matching algorithm used in the construction of the dual configuration $\fraks_1 \qq$ proceeds exactly the same way as the standard parenthesis-matching algorithm used when applying $\sigma_1$ to $\word(\qq)$ (ignoring the letters that are neither $i$ nor $i+1$). That is, the parentheses that get matched in the former are the same as those that get matched in the latter.
% Note that the construction of balanced sites is exactly the sequence of parentheses constructed when applying $\sigma_1$ to $\word(\qq)$ (removing all $\circ$ letters).
Now, recall that $\fraks_1$ merely toggles the unbalanced sites between $q_1$ and $q_2$, whereas $\sigma_1$ switches the number of unmatched `$)$'s with the number of unmatched `$($'s (which, in the case of $\word(\qq)$, boils down to just turning each unmatched `$)$' into a `$($' or vice versa, because one of these numbers is $0$).
Since the sites of the unmatched parentheses are precisely the unbalanced sites, this shows that the two maps agree -- that is, we have $\word(\fraks_1 \qq) = \sigma_1\bigl( \word(\qq) \bigr)$.
This proves~\eqref{pf.prop:braid.inter}.

The equality~\eqref{pf.prop:braid.inter} can be rewritten as the
commutative diagram
\[
\xymatrix{
 \set{\text{MLQs}} \ar[r]^{\fraks_i} \ar[d]_{\word} & \set{\text{MLQs}} \ar[d]^{\word} \\
 A^* \ar[r]_{\sigma_i} & A^*
}
\]
for all $i \in \set{1, 2, \dotsc, k-1}$.
In view of the injectivity of the map $\word \colon \set{\text{MLQs}} \to A^*$,
this diagram allows us to translate~\eqref{pf.prop:braid.loth-eq} into
$\fraks_i \fraks_{i+1} \fraks_i = \fraks_{i+1} \fraks_i \fraks_{i+1}$.
\end{proof}

\begin{example}
We perform the parenthesis-matching algorithm to the parenthesis sequence from Example~\ref{ex:parentheses_form}:
\[
( \quad \underbracket{( \quad )} \quad ( \quad ( \quad ( \quad \underbracket{( \quad )} \quad \underbracket{( \quad \underbracket{( \quad )} \quad )} \quad \underbracket{( \quad \underbracket{( \quad )} \quad \underbracket{( \quad )} \quad )}.
\]
As noted in the proof of Proposition~\ref{prop:braid}, none of the unpaired parentheses are nested inside of a matched pair of parentheses.
Hence, when we consider the corresponding dual configuration, we obtain the parenthesis sequence and matching
\[
) \quad \underbracket{( \quad )} \quad ) \quad ) \quad ) \quad \underbracket{( \quad )} \quad \underbracket{( \quad \underbracket{( \quad )} \quad )} \quad \underbracket{( \quad \underbracket{( \quad )} \quad \underbracket{( \quad )} \quad )}.
\]
\end{example}

\begin{remark}
Our letters $1, \ldots, k$ correspond to the letters $a_k, \ldots, a_1$ in~\cite{Loth},
since the definition of $\sigma_i$ in~\cite{Loth} involves $a_i$ rather
than $i+1$ as closing parenthesis and $a_{i+1}$ rather than $i$ as opening one.
Also,~\cite{Loth} does not include the letter $\circ$ in the alphabet,
but this makes no difference to the proof, since all letters $\circ$ are always frozen.
% Note that we have arbitrarily chosen to break our cycle at $n$,
% however by~\cite[Prop.~5.6.1]{Loth}, the result does not depend on this choice.
\end{remark}

\begin{remark}
The operator $\sigma_i$ is essentially a combination of co-plactic operators.
Moreover, it corresponds to the Weyl group action on a tensor product of crystals~\cite{BS17}.
Note that the bracketing rule given above is precisely the usual signature rule (see, \textit{e.g.},~\cite[Sec.~2.4]{BS17} for a description) for computing tensor products.
This arises from considering the MLQ as a binary $m \times n$ matrix and the natural $(\mathfrak{sl}_m \oplus \mathfrak{sl}_n)$-action.
\end{remark}

\begin{remark}
\label{rmk:braid.leeuwen}
The identity~\eqref{pf.prop:braid.loth-eq} can also be derived from \cite[Lemma 2.3]{vanLeeuwen-dc}. Indeed, to each word $p = p_1 p_2 \cdots p_d \in A^*$, we assign a $\tup{k+1} \times d$-matrix $B_p$ (a ``binary matrix'' in the parlance of \cite{vanLeeuwen-dc}), whose $j$-th column has an entry $1$ in its $k+1-p_j$-th position (we are treating $\circ$ as $0$ here) and entries $0$ everywhere else. Clearly, $p$ is uniquely determined by $B_p$. Now, van Leeuwen notices in \cite[before Proposition 1.3.5]{vanLeeuwen-dc} that the ``upward'' and ``downward'' moves between two consecutive rows of a binary matrix correspond to changing unmatched parentheses in a certain parenthesis sequence. When the binary matrix is $B_p$ and the two consecutive rows are the rows $k-i$ and $k-i+1$, this latter sequence is exactly the sequence of parentheses constructed in our definition of $\sigma_i$. Thus, applying the operator $\sigma_i$ to a word $p \in A^*$ is tantamount to applying van Leeuwen's operator $\sigma_{k-i}^{\updownarrow}$ to the binary matrix $B_p$. Hence,~\eqref{pf.prop:braid.loth-eq} follows from the relation $\sigma_{k-i}^{\updownarrow} \sigma_{k-(i+1)}^{\updownarrow} \sigma_{k-i}^{\updownarrow} = \sigma_{k-(i+1)}^{\updownarrow} \sigma_{k-i}^{\updownarrow} \sigma_{k-(i+1)}^{\updownarrow}$ between the latter operators on binary matrices, but the latter relation is part of \cite[Lemma 2.3]{vanLeeuwen-dc}.
\end{remark}

Next, we define two queues corresponding to a word $w \in \mcW_n$ of type $\mm$.
Namely, for $k \in \set{p_i(\mm) \mid i \geq 0}$, let $[w]_k$ denote the set of the indices $i \in \ive{n}$
corresponding to the $k$ smallest letters $w_i$ of $w$.

Our proof of Theorem~\ref{thm:permutation} will rely on a connection between dual configurations and the action of queues on words.
We begin with a string of lemmas.

\begin{lemma} \label{lem:SL.reconstruct}
Let $w, w' \in \mcW_n$ be two words of the same type $\mm$.
Assume that $[w]_k = [w']_k$ for each $k \in \set{p_i(\mm) \mid i \geq 0}$.
Then, $w = w'$.
\end{lemma}

\begin{proof}
Fix some $i \geq 1$.
The sites containing the letter $i$ in $w$ are the elements of $[w]_{p_i(\mm)} \setminus [w]_{p_{i-1}(\mm)}$
(since $w$ has type $\mm$).
Likewise,
the sites containing the letter $i$ in $w'$ are the elements of $[w']_{p_i(\mm)} \setminus [w']_{p_{i-1}(\mm)}$.
Since our assumption ($[w]_k = [w']_k$) yields $[w]_{p_i(\mm)} \setminus [w]_{p_{i-1}(\mm)}
= [w']_{p_i(\mm)} \setminus [w']_{p_{i-1}(\mm)}$,
we conclude that these are the same sites.
Since this holds for all letters $i$, we have $w = w'$.
\end{proof}

\begin{lemma} \label{lem:SL.dual.1}
Let $\tup{q_1, q_2}$ be a configuration with $\abs{q_1} \leq \abs{q_2}$.
Let $i \in q_1$.
Let $j$ be the first site weakly to the right of $i$ that belongs to $q_2$.
Then, $\SP(q_1 \setminus \set{i}, q_2 \setminus \set{j})$ is obtained from $\SP(q_1, q_2)$ by removing a pair of matched parentheses.\footnote{The matching of all other parentheses remains the same, and so does the correspondence between sites and closing parentheses (except for the one we removed).}
Moreover, the closing parenthesis of that pair is contributed by $j \in q_2$.
\end{lemma}

\begin{proof}
The assumption $\abs{q_1} \leq \abs{q_2}$ implies that all closing parentheses in $\SP(q_1, q_2)$ are matched.
In particular, the closing parenthesis $\gamma$ contributed by $j \in q_2$ is matched.
Let $\alpha$ be the opening parenthesis contributed by $i \in q_1$.
All parentheses between $\alpha$ and $\gamma$ are opening (by the minimality of $j$).
We know that $\gamma$ is matched to an opening parenthesis $\beta$.
Hence, the parenthesis sequence $\SP(q_1, q_2)$ between $\alpha$ and $\gamma$ (inclusive) is
\[
\begin{array}{cccccc}
( & ( & ( & \cdots & ( & ) \\
\alpha & & & & \beta & \gamma,
\end{array}
\]
where $\alpha$ and $\beta$ may be the same parenthesis.
Therefore $\SP(q_1 \setminus \set{i}, q_2 \setminus \set{j})$ is obtained from $\SP(q_1, q_2)$ by removing $\alpha$ and $\gamma$.
Yet this is tantamount to removing $\beta$ and $\gamma$ from $\SP(q_1, q_2)$, since both times we are simply shortening the string of opening parentheses before $\gamma$ by $1$ (and since we do not store the positions of the opening parentheses).
\end{proof}

\begin{lemma} \label{lem:SL.dual.3}
Let $u \in \mcW_n$ be a word of type $\mm$.
Let $k = p_{\alpha}(\mm)$ for some $\alpha$.
Let $q$ be a queue such that $k \leq \abs{q}$.
Each site $j \in \ive{q(u)}_k$ contributes a matched closing parenthesis to $\SP(\ive{u}_k, q)$.
\end{lemma}

\begin{proof}
Choose a permutation $(i_1, i_2, \dotsc, i_n)$ of $(1, 2, \dotsc, n)$ such that $u_{i_1} \leq u_{i_2} \leq \cdots \leq u_{i_n}$.
Use this permutation to construct $q(u)$ as per the definition.
For each $p \in \ive{n}$, let $j_p$ denote the site that is set while processing $i = i_p$ when constructing $q(u)$.
Thus, $(q(u))_{j_p} = u_{i_p}$ whenever $p \leq \abs{q}$ (since these entries of $q(u)$ are set in Phase~II),
and $(q(u))_{j_p} = u_{i_p} + 1$ otherwise (Phase~I).
Therefore, $\ive{q(u)}_k = \set{j_1, j_2, \ldots, j_k}$ (since $k \leq \abs{q}$).
Thus, we only need to prove that each of the sites $j_1, j_2, \ldots, j_k$ contributes a matched closing parenthesis to $\SP(\ive{u}_k, q)$.
Denote $q_1^{(p)} := q_1 \setminus \{i_1, i_2, \dotsc, i_p\}$ and $q_2^{(p)} := q_2 \setminus \{j_1, j_2, \dotsc, j_p\}$.

\begin{claim}
\label{claim:matching_SP}
Let $p \in \set{0, 1, \ldots, k}$.
Then, $\SP(q_1^{(p)}, q_2^{(p)})$ is obtained from $\SP(q_1 , q_2)$ by removing $p$ pairs of matched parentheses.
The closing parentheses of these $p$ pairs are contributed by $j_1, j_2, \ldots, j_p$.
\end{claim}

\begin{subproof}
We shall prove Claim~\ref{claim:matching_SP} by induction on $p$.
The base case ($p = 0$) is obvious.
For the induction step, we assume that Claim~\ref{claim:matching_SP} holds for $p-1$.
The site $j_p$ is chosen in Phase~II (since $p \leq k \leq \abs{q}$), and therefore $j_p$ is the first site weakly to the right of $i_p$ that belongs to $q_2^{(p-1)}$.
Hence, Lemma~\ref{lem:SL.dual.1}
(applied to $q_1^{(p-1)}$, $q_2^{(p-1)}$, $i_p$ and $j_p$ instead of $q_1$, $q_2$, $i$ and $j$)
implies that $\SP(q_1^{(p)}, q_2^{(p)})$ is obtained from $\SP(q_1^{(p-1)}, q_2^{(p-1)})$ by removing a pair of matched parentheses.
Moreover, the closing parenthesis of that pair is contributed by $j_p \in q_2^{(p-1)}$.
Thus, Claim~\ref{claim:matching_SP} follows from the induction hypothesis.
This completes the induction step.
\end{subproof}

The above claim applied to $p = k$ shows that $\SP(q_1^{(k)}, q_2^{(k)})$ is obtained from $\SP(q_1 , q_2)$ by removing $k$ pairs of matched parentheses, and that the closing parentheses of these $k$ pairs are contributed by $j_1, j_2, \ldots, j_k$.
Hence, each of the sites $j_1, j_2, \ldots, j_k$ contributes a matched closing parenthesis to $\SP(\ive{u}_k, q)$.
\end{proof}

\begin{prop} \label{prop:SL.dual}
Let $u \in \mcW_n$ be a word of type $\mm$.
Let $k = p_{\alpha}(\mm)$ for some $\alpha$.
Let $q$ be a queue.
The dual configuration of $\tup{ [u]_k , q }$ has the form $\tup{q^{\dagger} , [q(u)]_k }$, where $q^{\dagger}$ is some queue.
\end{prop}

\begin{proof}
Let $\tup{q'_1, q'_2}$ be the dual configuration of $\tup{ [u]_k , q }$.
We must then prove that $q'_2 = \ive{q(u)}_k$.
Note that $\abs{q'_2} = \abs{\ive{u}_k} = k = \abs{\ive{q(u)}_k}$.

Suppose $k \leq \abs{q}$.
Then, Lemma~\ref{lem:SL.dual.3} shows that each site $j \in \ive{q(u)}_k$ contributes a matched closing parenthesis to $\SP(\ive{u}_k, q)$.
Therefore, all of these sites are balanced, and hence belong to $q'_2$.
Thus, $q'_2 \supseteq \ive{q(u)}_k$, whence $q'_2 = \ive{q(u)}_k$
(since $\abs{q'_2} = \abs{\ive{q(u)}_k}$).

Now suppose $k > \abs{q}$.
Let $\ell$ be the number of classes in $u$, and consider the contragredient duals $q^*$ and $u^*$ as in Lemma~\ref{le:dual}.
Thus, $n-k < n - \abs{q} = \abs{q^*}$.
Hence, applying the $k \leq \abs{q}$ case (proven above)
to $n-k$, $u^*$ and $q^*$ instead of $k$, $u$ and $q$,
we see that
$\tup{ [u^*]_{n-k} , q^* }' = \tup{ q^\dagger , [q^*(u^*)]_{n-k} }$
for some queue $q^\dagger$.
Now,
\begin{align*}
 \bigl( (q'_1)^*, (q'_2)^* \bigr)
 &= \bigl( ([u]_k)^* , q^* \bigr)'
 = \tup{ [u^*]_{n-k} , q^* }' \\
 & = \tup{ q^{\dagger} , [q^*(u^*)]_{n-k} }
 = \tup{ q^{\dagger} , \bigl[\bigl(q(u) \bigr)^* \bigr]_{n-k} } ,
\end{align*}
where
\begin{itemize}
 \item the first equality follows from~\eqref{eq.rmk:balanced-dual-let.dual};
 \item the second equality is because $([u]_k)^* = [u^*]_{n-k}$; and
 \item the fourth equality follows from Lemma~\ref{le:dual}.
\end{itemize}
Therefore, $(q'_2)^* = \bigl[\bigl( q(u) \bigr)^* \bigr]_{n-k} = ([q(u)]_k)^*$, so that
$q'_2 = [q(u)]_k$.
Hence, Proposition~\ref{prop:SL.dual} is proven in the case when $k > \abs{q}$.
\end{proof}

\begin{proof}[Proof of Theorem~\ref{thm:permutation}.]
Recall that any permutation in $\SymGp{\ell-1}$ is a product of simple transpositions $s_1, s_2, \ldots, s_{\ell-2}$.
Hence, in order to prove Theorem~\ref{thm:permutation}, it suffices to show that $\swt{u}_{\sigma} = \swt{u}_{\sigma s_i}$ for each $\sigma \in \SymGp{\ell-1}$ and $i \in \ive{\ell-2}$.
Then, Theorem~\ref{thm:permutation} follows by induction on length, \textit{i.e.} the minimal number of simple transpositions needed to write $\sigma$.

In order to prove $\swt{u}_{\sigma} = \swt{u}_{\sigma s_i}$, we need to show, for a $\sigma$-twisted MLQ $\qq$ of type $\mm$ satisfying $u = \qq (1^n)$, that $\fraks_i \qq$ is a $\sigma s_i$-twisted MLQ of type $\mm$ satisfying $u = (\fraks_i \qq) (1^n)$ (since this will show that $\fraks_i$ bijects the former MLQs to the latter).
The only nontrivial part is showing $u = (\fraks_i\qq) (1^n)$.
More generally, we will show that $(\fraks_i\qq)(w) = \qq(w)$ for any word $w \in \mcW_n$.
The proof of this claim reduces to showing that for any configuration $C = \tup{q_1, q_2}$ and any word $w \in \mcW_n$ the dual configuration $\fraks_1 C = C' = (q'_1, q'_2)$ of $C$ satisfies $C'(w) = C(w)$.

Each word $w$ can be obtained from a standard word by a sequence of
merges (each of which sends a word $u$ to $\vee^{(k)} u$ for some
$k \in \set{ p_j(\mm) \mid j \geq 1 }$, where $\mm$ is the type of
$u$). Lemma~\ref{lemma:queue_merge_commute} shows that these merges
commute with the action of a queue (and thus of an MLQ).
Hence, it is sufficient to consider standard words $w$.
Thus, assume that $w$ is standard of type $\mm$.
It is straightforward to see
(using Equation~\eqref{eq:queue_type_change} and $\abs{q'_2} = \abs{q_1}$ and $\abs{q'_1} = \abs{q_2}$)
that the words
$C(w) = q_2\bigl( q_1(w) \bigr)$
and
$C'(w) = q'_2\bigl( q_1'(w) \bigr)$
have the same type.
%indeed, the action of a queue $q$ on a word $w$ modifies its
%type in a predictable way (namely, the number $\abs{q}$ is
%inserted into the type of $w$ at the place where it would
%keep the type weakly increasing).
Let $\nn$ be this type.
We shall now show that
$[C'(w)]_k = [C(w)]_k$
for all $k \in \set{ p_i(\nn) \mid i \geq 0}$.
According to Lemma~\ref{lem:SL.reconstruct}, this will yield $C'(w) = C(w)$, and thus our proof will be complete.

Let $k \in \set{ p_i(\nn) \mid i \geq 0}$.
Thus, $k \in \set{0, 1, \ldots, n} = \set{ p_i(\mm) \mid i \geq 0 }$ (since $w$ is standard).
Hence, $[w]_k$ is well-defined.
Note that $q_i(w)$ and $q'_i(w)$ are also standard words.
Using Proposition~\ref{prop:SL.dual} to compute dual
configurations, we can see how the MLQ
$\qq = \tup{[w]_k, q_1, q_2}$ transforms under the action of
$\fraks_1 \fraks_2 \fraks_1$: Namely, we have
\begin{align*}
\tup{[w]_k, q_1, q_2} & \overset{\fraks_1}{\longmapsto} \tup{\ast, [q_1(w)]_k, q_2}
\\ & \overset{\fraks_2}{\longmapsto} \tup{\ast, \ast, \bigl[ q_2\bigl( q_1(w) \bigr) \bigr]_k}
\\ & \overset{\fraks_1}{\longmapsto} \tup{\ast, \ast, \bigl[ q_2\bigl( q_1(w) \bigr) \bigr]_k},
\end{align*}
where $\ast$ denotes some queue.
Likewise, the action of $\fraks_2 \fraks_1 \fraks_2$ is given by
\begin{align*}
\tup{[w]_k, q_1, q_2} & \overset{\fraks_2}{\longmapsto} \tup{[w]_k, q'_1, q'_2}
\\ & \overset{\fraks_1}{\longmapsto} \tup{\ast, [q'_1(w)]_k , q'_2}
\\ & \overset{\fraks_2}{\longmapsto} \tup{\ast, \ast, \bigl[ q'_2\bigl( q'_1(w) \bigr) \bigr]_k}.
\end{align*}
(See Figure~\ref{fig:crossing_diagrams} for the actions depicted using crossing diagrams.)
Yet, the two maps are equal by Proposition~\ref{prop:braid}.
Thus, the resulting MLQs must be identical:
\[
\tup{\ast, \ast, \bigl[ q'_2\bigl(q'_1(w) \bigr) \bigr]_k}
=
\tup{\ast, \ast, \bigl[ q_2\bigl(q_1(w) \bigr) \bigr]_k}.
\]
Hence, we have
\[
\bigl[ q'_2\bigl( q'_1(w) \bigr) \bigr]_k = \bigl[ q_2\bigl( q_1(w) \bigr) \bigr]_k.
\]
In other words, $[C'(w)]_k = [C(w)]_k$.
\end{proof}

\begin{figure}
\begin{gather*}
\begin{tikzpicture}[xscale=3.5]
\node (i1) at (0,0) {$[w]_k$};
\node (i2) at (0,-1) {$q_1$};
\node (i3) at (0,-2) {$q_2$};
\node (s11) at (1,0) {$\ast$};
\node (s12) at (1,-1) {$[q_1(w)]_k$};
\node (s13) at (1,-2) {$q_2$};
\node (s21) at (2,0) {$\ast$};
\node (s22) at (2,-1) {$\ast$};
\node (s23) at (2,-2) {$\bigl[ q_2\bigl( q_1(w) \bigr) \bigr]_k$};
\node (t1) at (3,0) {$\ast$};
\node (t2) at (3,-1) {$\ast$};
\node (t3) at (3,-2) {$\bigl[ q_2\bigl( q_1(w) \bigr) \bigr]_k$};
\path[-] (i2) edge (s11) (s11) edge (s21) (s21) edge (t2);
\path[-] (i3) edge (s13) (s13) edge (s22) (s22) edge (t1);
\path[-,red,thick] (i1) edge (s12) (s12) edge (s23) (s23) edge (t3);
\end{tikzpicture}
\\
\begin{tikzpicture}[xscale=3.5]
\node (i1) at (0,0) {$[w]_k$};
\node (i2) at (0,-1) {$q_1$};
\node (i3) at (0,-2) {$q_2$};
\node (s11) at (1,0) {$[w]_k$};
\node (s12) at (1,-1) {$q'_1$};
\node (s13) at (1,-2) {$q'_2$};
\node (s21) at (2,0) {$\ast$};
\node (s22) at (2,-1) {$[q'_1(w)]_k$};
\node (s23) at (2,-2) {$q'_2$};
\node (t1) at (3,0) {$\ast$};
\node (t2) at (3,-1) {$\ast$};
\node (t3) at (3,-2) {$\bigl[ q'_2\bigl( q'_1(w) \bigr) \bigr]_k$};
\path[-] (i2) edge (s13) (s13) edge (s23) (s23) edge (t2);
\path[-] (i3) edge (s12) (s12) edge (s21) (s21) edge (t1);
\path[-,red,thick] (i1) edge (s11) (s11) edge (s22) (s22) edge (t3);
\end{tikzpicture}
\end{gather*}
\caption{Crossing diagrams representing the action of $\fraks_1 \fraks_2 \fraks_1$ (top) and $\fraks_2 \fraks_1 \fraks_2$ (bottom).}
\label{fig:crossing_diagrams}
\end{figure}

\begin{remark}
Theorem~\ref{thm:permutation} for the special case of $x_1 = \cdots = x_n = 1$ was proven in~\cite{AAMP} using different techniques.
We also sketch an alternative direct approach in the FPSAC extended abstract version of this work~\cite{AGS18}.
\end{remark}

%=====================================================================
\section{Final remarks}
\label{sec:remarks}

We conclude by giving some additional examples, remarks, and comments about our results.
We begin with an example to illustrate the proof of Theorem~\ref{thm:merge} in more detail.

\begin{example}
\label{ex:merge_theorem}
In order to compute $\swt{135452}$, we need to examine MLQs of type $(1,1,1,1, 2, 0, 0, \ldots)$.
We take a particular MLQ $\qq$ and add the $5$-queue $\set{1,2,3,5,6}$ as follows:
\[
\qq = \;
\begin{tikzpicture}[baseline=66,scale=0.75,every node/.style={inner sep=2pt}]
\node at (0, 4){2};\node at (1, 4){2};\node at (2, 4){2};\node[circle, draw=black] at (3, 4){1};\node at (4, 4){2};\node at (5, 4){2};\node at (0, 3){3};\node at (1, 3){3};\node[circle, draw=black] at (2, 3){2};\node at (3, 3){3};\node at (4, 3){3};\node[circle, draw=black] at (5, 3){1};\node[circle, draw=black] at (0, 2){1};\node[circle, draw=black] at (1, 2){3};\node at (2, 2){4};\node at (3, 2){4};\node[circle, draw=black] at (4, 2){2};\node at (5, 2){4};\node[circle, draw=black] at (0, 1){1};\node[circle, draw=black] at (1, 1){3};\node at (2, 1){5};\node[circle, draw=black] at (3, 1){4};\node at (4, 1){5};\node[circle, draw=black] at (5, 1){2};
\end{tikzpicture}
\; \xrightarrow{\hspace{30pt}} \;
\begin{tikzpicture}[baseline=66,scale=0.75,every node/.style={inner sep=2pt}]
\node[circle, draw=black] at (0, 5){1};\node[circle, draw=black] at (1, 5){1};\node[circle, draw=black] at (2, 5){1};\node at (3, 5){2};\node[circle, draw=black] at (4, 5){1};\node[circle, draw=black] at (5, 5){1};\node at (0, 4){2};\node at (1, 4){2};\node at (2, 4){3};\node[circle, draw=black] at (3, 4){1};\node at (4, 4){2};\node at (5, 4){2};\node at (0, 3){3};\node at (1, 3){4};\node[circle, draw=black] at (2, 3){2};\node at (3, 3){3};\node at (4, 3){3};\node[circle, draw=black] at (5, 3){1};\node[circle, draw=black] at (0, 2){1};\node[circle, draw=black] at (1, 2){3};\node at (2, 2){4};\node at (3, 2){4};\node[circle, draw=black] at (4, 2){2};\node at (5, 2){5};\node[circle, draw=black] at (0, 1){1};\node[circle, draw=black] at (1, 1){3};\node at (2, 1){5};\node[circle, draw=black] at (3, 1){4};\node at (4, 1){6};\node[circle, draw=black] at (5, 1){2};
\end{tikzpicture}
\; = \widetilde{\qq}.
\]
Thus, we obtain a $(s_4 s_3 s_2 s_1)$-twisted MLQ $\widetilde{\qq}$ of type $(1, 1, 1, 1, 1, 1, 0, \ldots)$.
Furthermore, note that $135452 = \merge{5} 135462$.
Now, by Theorem~\ref{thm:permutation}, such MLQs $\widetilde{\qq}$ are in bijection with ordinary MLQs contributing to, in this case, $\swt{135462}$.
In more detail, let $R_i(\widetilde{\qq})$ be the MLQ formed by taking the configuration $C = (\widetilde{q}_i, \widetilde{q}_{i+1})$ and replacing it with the dual configuration.
By taking $R_4 R_3 R_2 R_1(\widetilde{\qq})$ to bring the top row to the bottom, we obtain the ordinary MLQ as follows:
\begin{align*}
\; \widetilde{\qq} & \xrightarrow[\hspace{15pt}]{R_1} \;
\begin{tikzpicture}[xscale=0.7,yscale=0.7,every node/.style={inner sep=0.8pt},baseline=55]
\node at (0, 5){2};\node at (1, 5){2};\node[circle, draw=black] at (2, 5){1};\node at (3, 5){2};\node at (4, 5){2};\node at (5, 5){2};\node[circle, draw=black] at (0, 4){2};\node[circle, draw=black] at (1, 4){2};\node at (2, 4){3};\node[circle, draw=black] at (3, 4){1};\node[circle, draw=black] at (4, 4){2};\node[circle, draw=black] at (5, 4){2};\node at (0, 3){3};\node at (1, 3){4};\node[circle, draw=black] at (2, 3){2};\node at (3, 3){3};\node at (4, 3){3};\node[circle, draw=black] at (5, 3){1};\node[circle, draw=black] at (0, 2){1};\node[circle, draw=black] at (1, 2){3};\node at (2, 2){4};\node at (3, 2){4};\node[circle, draw=black] at (4, 2){2};\node at (5, 2){5};\node[circle, draw=black] at (0, 1){1};\node[circle, draw=black] at (1, 1){3};\node at (2, 1){5};\node[circle, draw=black] at (3, 1){4};\node at (4, 1){6};\node[circle, draw=black] at (5, 1){2};
\end{tikzpicture}
\; \xrightarrow[\hspace{15pt}]{R_2} \;
\begin{tikzpicture}[xscale=0.7,yscale=0.7,every node/.style={inner sep=0.8pt},baseline=55]
\node at (0, 5){2};\node at (1, 5){2};\node[circle, draw=black] at (2, 5){1};\node at (3, 5){2};\node at (4, 5){2};\node at (5, 5){2};\node at (0, 4){3};\node[circle, draw=black] at (1, 4){2};\node at (2, 4){3};\node at (3, 4){3};\node at (4, 4){3};\node[circle, draw=black] at (5, 4){1};\node[circle, draw=black] at (0, 3){3};\node at (1, 3){4};\node[circle, draw=black] at (2, 3){2};\node[circle, draw=black] at (3, 3){3};\node[circle, draw=black] at (4, 3){3};\node[circle, draw=black] at (5, 3){1};\node[circle, draw=black] at (0, 2){1};\node[circle, draw=black] at (1, 2){3};\node at (2, 2){4};\node at (3, 2){4};\node[circle, draw=black] at (4, 2){2};\node at (5, 2){5};\node[circle, draw=black] at (0, 1){1};\node[circle, draw=black] at (1, 1){3};\node at (2, 1){5};\node[circle, draw=black] at (3, 1){4};\node at (4, 1){6};\node[circle, draw=black] at (5, 1){2};
\end{tikzpicture}
\allowdisplaybreaks
\\[1.3em] &
\; \xrightarrow[\hspace{15pt}]{R_3} \;
\begin{tikzpicture}[xscale=0.7,yscale=0.7,every node/.style={inner sep=0.8pt},baseline=55]
\node at (0, 5){2};\node at (1, 5){2};\node[circle, draw=black] at (2, 5){1};\node at (3, 5){2};\node at (4, 5){2};\node at (5, 5){2};\node at (0, 4){3};\node[circle, draw=black] at (1, 4){2};\node at (2, 4){3};\node at (3, 4){3};\node at (4, 4){3};\node[circle, draw=black] at (5, 4){1};\node[circle, draw=black] at (0, 3){3};\node at (1, 3){4};\node at (2, 3){4};\node at (3, 3){4};\node[circle, draw=black] at (4, 3){2};\node[circle, draw=black] at (5, 3){1};\node[circle, draw=black] at (0, 2){1};\node[circle, draw=black] at (1, 2){3};\node[circle, draw=black] at (2, 2){4};\node[circle, draw=black] at (3, 2){4};\node[circle, draw=black] at (4, 2){2};\node at (5, 2){5};\node[circle, draw=black] at (0, 1){1};\node[circle, draw=black] at (1, 1){3};\node at (2, 1){5};\node[circle, draw=black] at (3, 1){4};\node at (4, 1){6};\node[circle, draw=black] at (5, 1){2};\end{tikzpicture}
\; \xrightarrow[\hspace{15pt}]{R_4} \;
\begin{tikzpicture}[xscale=0.7,yscale=0.7,every node/.style={inner sep=0.8pt},baseline=55]
\node at (0, 5){2};\node at (1, 5){2};\node[circle, draw=black] at (2, 5){1};\node at (3, 5){2};\node at (4, 5){2};\node at (5, 5){2};\node at (0, 4){3};\node[circle, draw=black] at (1, 4){2};\node at (2, 4){3};\node at (3, 4){3};\node at (4, 4){3};\node[circle, draw=black] at (5, 4){1};\node[circle, draw=black] at (0, 3){3};\node at (1, 3){4};\node at (2, 3){4};\node at (3, 3){4};\node[circle, draw=black] at (4, 3){2};\node[circle, draw=black] at (5, 3){1};\node[circle, draw=black] at (0, 2){1};\node[circle, draw=black] at (1, 2){3};\node at (2, 2){5};\node[circle, draw=black] at (3, 2){4};\node[circle, draw=black] at (4, 2){2};\node at (5, 2){5};\node[circle, draw=black] at (0, 1){1};\node[circle, draw=black] at (1, 1){3};\node[circle, draw=black] at (2, 1){5};\node[circle, draw=black] at (3, 1){4};\node at (4, 1){6};\node[circle, draw=black] at (5, 1){2};
\end{tikzpicture}
\; = \widetilde{\qq}',
\end{align*}
which contributes to $\swt{135462}$.
\end{example}

There is the natural cyclic symmetry on our spectral weights.

\begin{prop}
  Let $C_n \subseteq \SymGp{n}$ denote the cyclic group of order $n$ generated by the $n$-cycle $(1 \; 2 \; \dotsm \; n)$ and $u \in \mcW_n$.
  We have $\swt{u}_{\tau}\sigma = \swt{u \sigma}_{\tau}$ for any $\sigma \in C_n$ and $\tau \in \SymGp{\ell-1}$, where $\SymGp{n}$ acts on $\mcW_n$ from the right by $(u_1 \dotsm u_n) \sigma = u_{\sigma(1)} \dotsm u_{\sigma(n)}$, that is to say permutations act on positions, and similarly on monomials in $\xx$.
\end{prop}

\begin{proof}
  For a queue $q$, define $\sigma q := \{ \sigma(i) \mid i \in q\}$.
  It is clear from the definition of a queue that $q(u \sigma) = \bigl( (\sigma q)(u) \bigr) \sigma$.
  Hence, for any $\tau$-twisted MLQ $\qq$ of type $\mm$, we have
  $\qq (u \sigma) = \bigl( (\sigma \qq)(u) \bigr) \sigma$,
  where the action of $C_n$ on MLQs is obtained by acting on each queue separately.
  Thus, in particular, for any $\tau$-twisted MLQ $\qq$ of type $\mm$, we have
  $\qq (1^n) = \bigl( (\sigma \qq)(1^n) \bigr) \sigma$.
  From here, the claim follows by an obvious bijection
  (given by the action of $\sigma$) between the sums defining
  $\swt{u}_{\tau} \sigma$ and $\swt{u \sigma}_{\tau}$.
\end{proof}

Let $B^{r,s}$, where $r \in [n-1]$ and $s \in \ZZ_{>0}$, be a \defn{Kirillov--Reshetikhin (KR) crystal} in type $A_{n-1}^{(1)}$~\cite{KKMMNN92}.
Recall from~\cite{NY97,Shimozono02} that the combinatorial $R$-matrix is the unique crystal isomorphism
\[
R \colon B^{r_1,s_1} \otimes B^{r_2,s_2} \to B^{r_2,s_2} \otimes B^{r_1,s_1}.
\]
There is a well-known (in slightly different terminology) bijection $\Xi_r$ relating $r$-queues with an element of the KR crystal $B^{r,1}$ in type $A_{n-1}^{(1)}$ by considering $q = \{b_1 < \cdots < b_r\}$  as a single-column Young tableau of height $r$.
In~\cite{KMO15}, this was extended to a bijection $\Xi$ between multiline queues of type $\mm$ with $\ell+1$ classes and $B^{p_1,1} \otimes B^{p_2,1} \otimes \dotsm \otimes B^{p_{\ell},1}$ by
\[
\Xi(\qq) = \Xi_{p_1}(q_1) \otimes \Xi_{p_2}(q_2) \otimes \dotsm \otimes \Xi_{p_{\ell}}(q_{\ell}).
\]
Thus, by comparing the description of the combinatorial $R$-matrix for $B^{r_1,1}$ and $B^{r_2,1}$ from~\cite{NY97}, we obtain that taking the dual configuration is equivalent to applying the combinatorial $R$-matrix under $\Xi$.

\begin{prop}
\label{prop:dual_is_R}
Let $C$ be an $(r_1, r_2)$-configuration. Then
\[
\Xi(C') = R\bigl( \Xi(C) \bigr),
\]
where $C'$ is the dual configuration of $C$.
\end{prop}

Note that Proposition~\ref{prop:dual_is_R} gives another proof of Proposition~\ref{prop:braid} since the combinatorial $R$-matrix is well-known to satisfy the Yang-Baxter equation.

\begin{example}
Suppose $n = 9$.
Consider the $(4,6)$-configuration and dual $(6,4)$-configuration
\[
C = (\set{1,4,5,6}, \set{2,3,4,6,7,8}),
\qquad\quad
C' = (\set{1,3,4,5,6,8}, \set{2,4,6,7}).
\]
We have
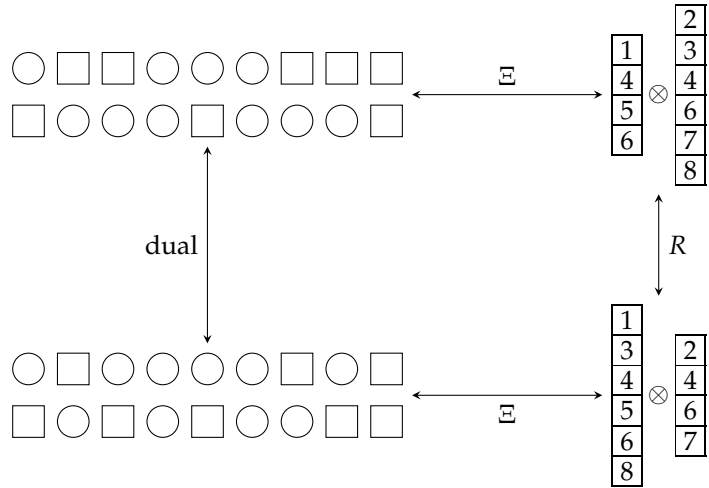
\[
\begin{tikzpicture}[xscale=6,yscale=4,>=stealth]
\node (C) at (0,0) { %{$(\set{1,4,5,6}, \set{2,3,4,6,7,8})$};
\begin{tikzpicture}[scale=0.7]
  \def\sc{0.85}   % Change this to adjust the x-scaling
  \def\ll{2}   % level 2
  \def\l{1}   % level 1
  \foreach \i in {1,4,5,6} { \draw[fill=white] (\i*\sc,\ll) circle (0.3); }
  \foreach \i in {2,3,7,8,9} { \draw[fill=white] (\i*\sc-.3,\ll-.3) rectangle +(0.6,+0.6); }
  \foreach \i in {2,3,4,6,7,8} { \draw[fill=white] (\i*\sc,\l) circle (0.3); }
  \foreach \i in {1,5,9} { \draw[fill=white] (\i*\sc-.3,\l-.3) rectangle +(0.6,+0.6); }
\end{tikzpicture}
};
\node (Cp) at (0,-1) { %{$(\set{1,3,4,5,6,8}, \set{2,4,6,7})$};
\begin{tikzpicture}[scale=0.7]
  \def\sc{0.85}   % Change this to adjust the x-scaling
  \def\ll{2}   % level 2
  \def\l{1}   % level 1
  \foreach \i in {1,3,4,5,6,8} { \draw[fill=white] (\i*\sc,\ll) circle (0.3); }
  \foreach \i in {2,7,9} { \draw[fill=white] (\i*\sc-.3,\ll-.3) rectangle +(0.6,+0.6); }
  \foreach \i in {2,4,6,7} { \draw[fill=white] (\i*\sc,\l) circle (0.3); }
  \foreach \i in {1,3,5,8,9} { \draw[fill=white] (\i*\sc-.3,\l-.3) rectangle +(0.6,+0.6); }
\end{tikzpicture}
};
\node (KR) at (1,0) {$\young(1,4,5,6) \otimes \young(2,3,4,6,7,8)$};
\node (KRp) at (1,-1) {$\young(1,3,4,5,6,8) \otimes \young(2,4,6,7)$};
\draw[<->] (C) -- (Cp) node[midway,anchor=east] {dual};
\draw[<->] (C) -- (KR) node[midway,anchor=south] {$\Xi$};
\draw[<->] (KR) -- (KRp) node[midway,anchor=west] {$R$};
\draw[<->] (Cp) -- (KRp) node[midway,anchor=north] {$\Xi$};
\end{tikzpicture}
\]
\end{example}

Since the combinatorial $R$-matrix satisfies the Yang--Baxter equation, we have an action of the symmetric group $\SymGp{\ell}$ acting on MLQs with $\ell + 1$ queues.
Given the description of the combinatorial $R$-matrix using the Robinson--Schensted--Knuth (RSK) bijection~\cite{Shimozono02}, this $\SymGp{\ell}$-action can be considered as corresponding to the one given by van Leeuwen~\cite[Lemma~2.3]{vanLeeuwen-dc}.\footnote{This could also be considered as an interpretation of the Littlewood--Richardson rule, along with the fact that $s_{\lambda} s_{\mu}$, where $s_{\lambda}$ and $s_{\mu}$ are Schur functions corresponding to rectangles, is multiplicity free~\cite{Stembridge01}, and using Howe duality~\cite[Ch.~9,App.~B]{BS17}.}
Furthermore, the $\SymGp{\ell}$-action on MLQs has been considered by Danilov and Koshevoy~\cite{DanilovKoshevoy} in a different context (see also~\cite[Ch.~4]{Gorodentsev2}).
Unlike the combinatorial $R$-matrix perspective, they do not have the natural interpretation of the weight since $\wt(\qq) = \wt\bigl( \Xi(\qq) \bigr)$, where the crystal weight is the usual tableaux weight.

Next, we describe how to interpret our action from looking at the corner transfer matrix described in~\cite{KMO15}, which can be given diagrammatically by
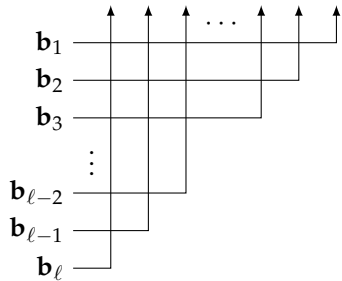
\[
\begin{tikzpicture}[>=latex,scale=0.5]
\draw[->] (0,1) node[anchor=east] {$\mathbf{b}_{\ell}$} -- (1,1) -- (1,8);
\foreach \x in {1,2} {
  \draw[->] (0,\x+1) node[anchor=east] {$\mathbf{b}_{\ell-\x}$} -- (\x+1,\x+1) -- (\x+1,8);
}
\foreach \x in {1,2,3} {
  \draw[->] (0,8-\x) node[anchor=east] {$\mathbf{b}_{\x}$} -- (8-\x,8-\x) -- (8-\x,8);
}
\node at (4,7.5) {$\cdots$};
\node at (0.5,4) {$\vdots$};
\end{tikzpicture}
\]
where every crossing is a combinatorial $R$-matrix and $\mathbf{b}_i \in B^{p_i,1}$.
Our action effectively does the following on the corner transfer matrix:
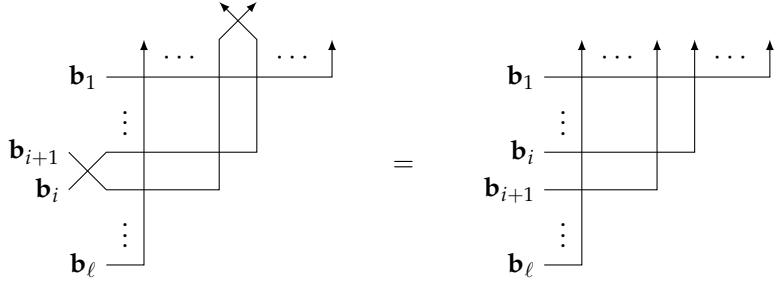
\[
\begin{tikzpicture}[>=latex,scale=0.5,baseline=50]
\draw[->] (0,1) node[anchor=east] {$\mathbf{b}_{\ell}$} -- (1,1) -- (1,7);
\draw[->] (0,6) node[anchor=east] {$\mathbf{b}_1$} -- (6,6) -- (6,7);
\draw[->] (-1,3) node[anchor=east] {$\mathbf{b}_i$} -- (0,4) -- (4,4) -- (4,7) -- (3,8);
\draw[->] (-1,4) node[anchor=east] {$\mathbf{b}_{i+1}$} -- (0,3) -- (3,3) -- (3,7) -- (4,8);
\node at (2,6.5) {$\cdots$};
\node at (0.5,2) {$\vdots$};
\node at (5,6.5) {$\cdots$};
\node at (0.5,5) {$\vdots$};
\end{tikzpicture}
\qquad = \qquad
\begin{tikzpicture}[>=latex,scale=0.5,baseline=50]
\draw[->] (0,1) node[anchor=east] {$\mathbf{b}_{\ell}$} -- (1,1) -- (1,7);
\draw[->] (0,6) node[anchor=east] {$\mathbf{b}_1$} -- (6,6) -- (6,7);
\draw[->] (0,4) node[anchor=east] {$\mathbf{b}_i$} -- (4,4) -- (4,7);
\draw[->] (0,3) node[anchor=east] {$\mathbf{b}_{i+1}$} -- (3,3) -- (3,7);
\node at (2,6.5) {$\cdots$};
\node at (0.5,2) {$\vdots$};
\node at (5,6.5) {$\cdots$};
\node at (0.5,5) {$\vdots$};
\end{tikzpicture}
\]
where the equality comes from applying the Yang--Baxter equation and $R^2 = \id$.

It could be interesting to see if our proof has any implications for the work of Danilov and Koshevoy~\cite{DanilovKoshevoy} or van Leeuwen~\cite{vanLeeuwen-dc}.

\appendix

%=====================================================================
\section{Connections with other constructions}
\label{app:queue-relations}

We describe how our definition of $q(u)$ (when $q$ is a queue and $u$ a word) is connected to the construction of~\cite{FM07}, which was coined bully paths in~\cite{AasLin17}, and the tableaux combinatorics of~\cite{AssSea18}.

To clarify the connection, we recall our visualization of a queue $q$ in~\eqref{eq:boxes-and-balls-1} by a row of circles and squares.
We modify this visualization slightly:
Namely, we still represent each $i \in q$ by a circle as in Example~\ref{ex:first_queue}, but we no longer represent the $i \notin q$ by squares.
We refer to the circles as ``boxes''.

\subsection{Connection with Ferrari--Martin and Aas--Linusson MLQs}

We first connect our MLQs with those from~\cite{FM07}.
We will use the language of~\cite{AasLin17}, where we are only considering ``discrete MLQs'' as we consider our ring to have a finite number of sites.

Consider a MLQ $\qq = (q_1, q_2, \dotsc, q_{\ell})$.
A \defn{labeling} of $\qq$ is a sequence of maps $\ff = (f_1, \dotsc, f_{\ell})$, where $f_i \colon q_i \to \ive{i}$.
We represent this by placing an $f_i(j)$, which we call the \defn{label}, inside of the circle corresponding to $j \in q_i$.
The \defn{canonical labeling} $\ff_{\qq}$ of $\qq$ is the labeling $(f_1, \dotsc, f_{\ell})$ defined by
\begin{equation}
f_k(j) = \bigl( q_k( \cdots q_1(1^n) \cdots ) \bigr)_j.
\label{eq.AasLin-lab.formula}
\end{equation}
For example, the labeling of each MLQ in Example~\ref{ex:merge_theorem} is precisely the circled values.
When $\qq$ is an ordinary MLQ, we can also construct this canonical labeling recursively as follows:
\begin{enumerate}
\item Set $q_0 = \emptyset$, and let $f_0 : q_0 \to \ive{0}$ be the trivial map.
\item For each $k = 0, 1, \ldots, \ell-1$:
\begin{itemize}
\item Suppose $f_k \colon q_k \to \ive{k}$ is already defined. Let $\tup{j_1, j_2, \ldots, j_r}$ be a list of all elements of $q_k$ in the order of increasing label in $f_k$; that is, $f_k(j_1) \leq f_k(j_2) \leq \cdots \leq f_k(j_r)$.
      (The relative order between elements with equal label is immaterial.)
\item For $i = i_1, i_2, \ldots, i_r$, do the following:
      Find the first site $j$ weakly to the right (cyclically) of $i$ such that $j \in q_{k+1}$ and $f_{k+1}(j)$ is not set; then set $f_{k+1}(j) = f_k(i)$.
\item For all sites $j \in q_{k+1}$ for which $f_{k+1}(j)$ is not set yet, set $f_{k+1}(j) = k+1$.
\end{itemize}
This defines $f_{k+1} \colon q_{k+1} \to \ive{k+1}$.
\end{enumerate}
\begin{verlong}
In order to see that this recursive construction yields the same labeling $\ff_\qq$ as the equality~\eqref{eq.AasLin-lab.formula}, we can proceed as follows:

If $r$ is a queue and $I$ is a set, then a \defn{labeling} of $r$ by $I$ means a map $f \colon r \to I$.
Of course, we refer to a value $f \tup{i}$ of such a labeling as the \defn{label} of $i$ in $f$.
To represent such a labeling $f$, we label each of the boxes of $r$ with its corresponding value (i.e., we label the box in position $i$ by $f\tup{i}$).
Thus, a labeling of an MLQ $\qq = (q_1, q_2, \dotsc, q_{\ell})$ is a sequence of labelings $f_i \colon q_i \to \ive{i}$ of each of its queues $q_i$; and its visual representation is obtained by stacking the visual representations of the $f_i$ one atop another.

Now, fix a $k \in \NN$ and two queues $r$ and $q$ with $\abs{r} \leq \abs{q}$. Also fix a labeling $f \colon r \to \ive{k}$ of $r$ by $\ive{k}$.
Then, we can define a labeling \defn{$q(f)$} of $q$ by $\ive{k+1}$ as follows:
\begin{enumerate}
\item Let $\tup{i_1, i_2, \ldots, i_{\abs{r}}}$ be a list of all elements of $r$ in the order of increasing label in $f$ (that is, $f(i_1) \leq f(i_2) \leq \cdots \leq f(i_{\abs{r}})$.
 (The relative order between elements with equal label is immaterial.)

\item Initialize a labeling $g$ of $q$ by $\ive{k+1}$. For now, none of its values is set.

\item For $i = i_1, i_2, \ldots, i_{\abs{r}}$, do the following.
    Find the first site $j$ weakly to the right (cyclically) of $i$ such that $j \in q$ and $g\tup{j}$ is not set.
    Then set $g\tup{j} = f\tup{i}$.

\item For all sites $j \in q$ for which $g \tup{j}$ is not set yet, set $g \tup{j} = k+1$.

\item Define $q(f)$ to be the resulting labeling $g$.
\end{enumerate}

The connection between this labeling and the action of queues on words is simple.
Indeed, let $k \in \NN$, let $r$ be a queue, and let $f \colon r \to \ive{k}$ be a labeling.
Then we can can assign to $f$ a word $\omega_k f \in \mcW_n$ as follows:
For each $j \in r$, we set $\tup{\omega_k f}_j = f \tup{j}$.
All other letters of $\omega_k f$ shall be $k+1$.
Now, it is easy to see that if $q$ is a queue satisfying $\abs{r} \leq \abs{q}$, then
\begin{equation}
q \tup{ \omega_k f } = \omega_{k+1} \tup{ q(f) } .
\label{eq.AasLin-lab.two-actions}
\end{equation}
Thus, the action of a queue on a labeling can be computed in terms of its action on a word.
(The converse is not true, since the condition $\abs{r} \leq \abs{q}$ is restrictive.
Thus, our action of queues on words, with its two phases, is more general than their
action on labelings.)

Now, let $\qq = (q_1, q_2, \dotsc, q_{\ell})$ be an ordinary MLQ
(so that $\abs{q_1} \leq \abs{q_2} \leq \cdots \leq \abs{q_{\ell}}$),
and consider the labeling $\ff_\qq = (f_1, \dotsc, f_{\ell})$ constructed by the above recursive procedure
(not by~\eqref{eq.AasLin-lab.formula}).
Then, the recursive procedure clearly yields
\begin{align*}
f_i &= q_i \tup{ f_{i-1} } \qquad \text{for all } i \in \ive{\ell} ,
\end{align*}
where $f_0$ is understood to be the (trivial) labeling of the empty queue $\varnothing$.
In view of~\eqref{eq.AasLin-lab.two-actions}, this entails
\[
\omega_i \tup{ f_i } = q_i \tup{ \omega_{i-1} \tup{ f_{i-1} } } \qquad \text{for all } i \in \ive{\ell} .
\]
Thus, by induction,
\[
\omega_i \tup{ f_i } = q_i( \cdots q_1(1^n) \cdots )
 \qquad \text{for all } i \in \set{0, 1, \dotsc, \ell}
\]
(since $\omega_0 \tup{ f_0 } = 1^n$).
Hence,
\[
f_i(j) = \bigl( q_i( \cdots q_1(1^n) \cdots ) \bigr)_j
 \qquad \text{for all } i \in \set{0, 1, \dotsc, \ell}
        \text{ and } j \in q_i .
\]
But this is precisely the equality~\eqref{eq.AasLin-lab.formula}.
Hence, we have shown that the recursive construction yields the same labeling $\ff_\qq$ as the equality~\eqref{eq.AasLin-lab.formula}.

\end{verlong}
Note that the canonical labeling $\ff_{\qq}$ are the elements in the circles in the graveyard diagram of $\qq$ (as in Example~\ref{ex:qij_generic}).

Now consider the labeling procedure in~\cite[\S 2.2]{AasLin17} given by \defn{$k$-bully paths}.
Note that a $k$-bully path from one queue to the next precisely corresponds to the path of a letter $k$ under Phase~II of our definition of a queue as a function on words.
For example, the bully paths would correspond to the blue paths in Example~\ref{ex:first_queue}.
In addition, this is exactly the recursive labeling procedure given above.
Thus, the labeling $\ff_{\qq}$ is equivalently constructed following the labeling procedure of~\cite{AasLin17} using bully paths.

\subsection{Connection with Kohnert diagrams and Assaf--Searles theory}

Next we relate the action of queues on words with the \defn{Kohnert labelings} in~\cite[Def.~2.5]{AssSea18} and the \defn{thread decomposition} in~\cite[Def.~3.5]{AssSea18}.
We remark that the thread decomposition is the same as the Kohnert labeling when the shape is an antipartition (\textit{i.e.} a weakly increasing sequence of positive integers).
Roughly speaking, Kohnert diagrams are MLQs built of queues that live on a half-line (instead of a circle), and the construction of the Kohnert labeling (and the thread decomposition) is a standardization of the bully path construction.

In more detail (and using the notations of~\cite{AssSea18}):
If $\alpha$ is a weak composition, and if $D \in \KM(\alpha)$ is a Kohnert diagram, then we view the columns of $D$ as queues.
This time, our queues are subsets of $\NN$ (or $\ZZ$) instead of $\ZZ/n\ZZ$; thus, there is no ``wrapping around''.
We consider the reflection of $D$ across the line $x = y$ as an MLQ $\qq_D = (q_1, \dotsc, q_{\ell})$:
namely, a cell in row $i$ and column $j$ in the reflected diagram corresponds to a $j \in q_i$,
and $\ell$ is the number of columns in $D$.
We then apply the bully path construction to the boxes of this reflected Kohnert diagram.
To obtain the thread decomposition we need to distinguish paths with a fixed label such that these paths are also constructed by the bully path algorithm, where we consider the labels to be decreasing from left to right.
Hence, this can be considered as a standardization of our construction or, equivalently, as fixing specific permutations for how the queues act on words.

\begin{example}
\label{ex:thread_decomp}
Consider the thread decomposition of the Kohnert diagram
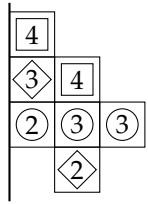
\[
%K =\ 
\begin{tikzpicture}[scale=0.6, baseline=-40]
\draw[thick] (0,0.2) -- (0,-4.2);
\foreach \x/\y in {0/0, 0/-1, 1/-1, 0/-2, 1/-2, 2/-2, 1/-3} {
    \draw (\x,\y) rectangle (\x+1,\y-1);
}
\foreach \x/\y in {0/0, 1/-1} {
    \node at (\x+0.5,\y-0.5) {$4$};
    \draw (\x+0.5-0.35,\y-0.5+0.35) rectangle (\x+0.5+0.35,\y-0.5-0.35);
}
\foreach \x/\y in {0/-1, 1/-2, 2/-2}
    \node at (\x+0.5,\y-0.5) {$3$};
\foreach \x/\y in {0/-2, 1/-3}
    \node at (\x+0.5,\y-0.5) {$2$};
\foreach \x in {0,1,2}
    \draw (\x+0.5,-2.5) circle (0.35);
\foreach \x/\y in {0/-1, 1/-3}
    \draw[shift={(\x+0.5,\y-0.5)},rotate=45] (-0.3,0.3) rectangle (0.3,-0.3);
\end{tikzpicture}
\]
given in~\cite[Fig.~11]{AssSea18}.
Thus, the corresponding MLQ is $\tup{\set{2}, \set{1,2,3}, \set{2,3,4}}$;
we can draw it using bully paths as follows:
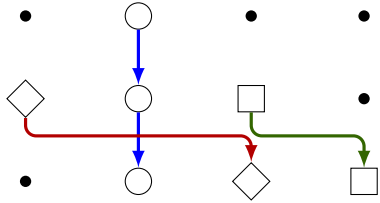
\[
\begin{tikzpicture}[>=latex]
\def\xs{1.5}
\def\ys{1.1}
\foreach \y in {0,1,2}
    \node[draw,circle,minimum size=0.35cm] (2\y) at (2*\xs,-\y*\ys) {};
\foreach \x/\y in {1/1, 3/2}
    \node[draw,rectangle,rotate=45,minimum size=0.35cm] (\x\y) at (\x*\xs,-\y*\ys) {};
\foreach \x/\y in {3/1, 4/2}
    \node[draw,rectangle,minimum size=0.35cm] (\x\y) at (\x*\xs,-\y*\ys) {};
\draw[very thick, blue, ->] (20) -- (21);
\draw[very thick, blue, ->] (21) -- (22);
\draw[very thick, darkred, ->, rounded corners] (11) -- (1*\xs,-1.45*\ys) -- (3*\xs,-1.45*\ys) -- (32);
\draw[very thick, dgreencolor, ->, rounded corners] (31) -- (3*\xs,-1.45*\ys) -- (4*\xs,-1.45*\ys) -- (42);
\foreach \pt in {(1*\xs,0), (3*\xs,0), (4*\xs,0), (4*\xs,-1*\ys), (1*\xs,-2*\ys)}
    \fill[black] \pt circle(0.075);
\end{tikzpicture}
\]
where each bully path matches with a thread in the decomposition.
Note that for each fixed $k$, the $k$-bully paths must be constructed from left to right.\footnote{In terms of our permutation $(i_1, i_2, \dotsc, i_n)$ for Phase~II, for $u_{i_j} = u_{i_{j+1}} = \cdots = u_{i_{j'}} = k$, we must have $i_j < i_{j+1} < \cdots < i_{j'}$.}
If not, the diamond and square in the bottom row would be interchanged.
\end{example}

Note that the distinction between $\NN$ and $\ZZ/n\ZZ$ never arises since, for $\qq_D$ and all $i$, we have
\[
\abs{\set{ j \in q_i \mid j \geq k }} \leq \abs{\set{  j \in q_{i+1} \mid j \geq k }},
\]
(by~\cite[Lemma~2.2]{AssSea18}), which means that there is no ``wrapping'' around the cylinder.

To obtain a Kohnert labeling from a Kohnert diagram of height $K$ (for this, we require $n \gg 1$), we can construct an MLQ
\[
\widetilde{\qq}_D = (\widetilde{q}_1, \widetilde{q}_2, \dotsc, \widetilde{q}_{\widetilde{\ell}}, q_1, q_2, \dotsc, q_{\ell})
\]
from the MLQ $\qq_D = (q_1, q_2, \dotsc, q_{\ell})$ to obtain the correct labelings, where $\ell + \ell'$ is the largest label appearing in the Kohnert labeling.
Indeed, a label added in column $i$ comes from a $k \in q$ with $k > K$ for sufficiently many queues $q$ before $q_i$.
In particular, when a smaller label appears, it must be in the bottom row of the Kohnert diagram, which would correspond to the bully path wrapping around the cylinder.
We then only consider the labels in the regime $k \in q_i$ for all $1 \leq k \leq K$ and all $1 \leq i \leq \ell$.
We leave the precise details for the interested reader.

\begin{example}
We consider the Kohnert labeling from Example~\ref{ex:thread_decomp}.
The following MLQ, given as a graveyard diagram, is an MLQ that gives the corresponding Kohnert labeling:
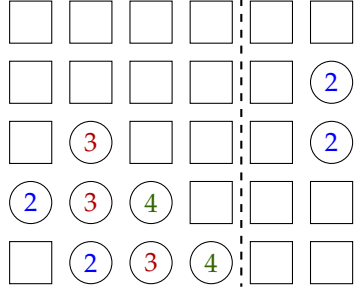
\[
\begin{tikzpicture}[scale=0.8]
\foreach \pt in {(6,-1), (2,-2),(6,-2), (1,-3),(2,-3),(3,-3), (2,-4),(3,-4),(4,-4)}
    \draw \pt circle (0.35);
\foreach \x/\y in {1/0,2/0,3/0,4/0,5/0,6/0, 1/-1,2/-1,3/-1,4/-1,5/-1, 1/-2,3/-2,4/-2,5/-2, 4/-3,5/-3,6/-3, 1/-4,5/-4,6/-4}
    \draw (\x-0.35,\y+0.35) rectangle (\x+0.35,\y-0.35);
\foreach \x/\y in {6/-1, 6/-2, 1/-3, 2/-4}
    \node[blue] at (\x,\y) {$2$};
\foreach \x/\y in {2/-2, 2/-3, 3/-4}
    \node[darkred] at (\x,\y) {$3$};
\foreach \x/\y in {3/-3, 4/-4}
    \node[dgreencolor] at (\x,\y) {$4$};
\draw[dashed,thick] (4.5,0.4) -- (4.5,-4.4);
\end{tikzpicture}
\]
where we have suppressed the $(i+1)$'s that appear in row $i$.
Note that all circles that appear to the left of the dashed line correspond to the Kohnert labeling and that the labels match.
\end{example}

\begin{verlong}
%=====================================================================
\section{Appendix: Old proof of Theorem~\ref{thm:permutation}}
\label{sec:thm_proof-old6}

This section contains a previous proof of Theorem~\ref{thm:permutation}.
This proof uses similar ideas to the proof given in Section~\ref{sec:thm_proof} above,
but differs in its details (in particular, it defines the dual configuration
in a different way, although the two definitions can be shown to be equivalent).

We need some definitions first.

An \defn{$(r_1,r_2)$-configuration} shall mean a pair $C = (q_1, q_2)$, where $q_1$ is an $r_1$-queue and $q_2$ is an $r_2$-queue.
As usual, we consider $C$ as a function on words by $C(u) := q_2\bigr(q_1(u)\bigr)$, and we define the weight of $C$ by $\wt(C) := \wt(q_1) \wt(q_2)$.
We construct the \defn{dual}\footnote{This is a different duality than the contragredient duality of Lemma~\ref{le:dual}.} $(r_2,r_1)$-configuration to $C$, which we denote by $C'$, as follows.

We first consider the case $r_1 = r_2$, in which case we define $C' = C$.
Thus, assume $r_1 \neq r_2$.

For any two sites $i$ and $j$, let $\inter[i,j]$ denote a closed (cyclic) interval from $i$ to $j$.
This is the set $\set{i, i+1, \ldots, j}$ when $i \leq j$ and the set $\set{i, i+1, \ldots, n, 1, 2, \ldots, j}$ when $i > j$.\footnote{We trivially consider an empty interval to be balanced, and note that if $j = i-1$, then $\inter[i,j] = \ive{n}$.}
Let $c^{\uparrow}[i,j]$ (resp.~$c^{\downarrow}[i,j]$) denote the number of $\ell \in \inter[i,j]$ such that $\ell \in q_1$ (resp.~$\ell \in q_2$).

We say that a closed cyclic interval $\inter[i,j]$ is \defn{balanced} if $c^{\uparrow}[i,j] = c^{\downarrow}[i,j]$ and for each $k \in \inter[i,j]$, we have $c^\uparrow[i,k] \geq c^\downarrow[i,k]$.
Equivalently, $\inter[i,j]$ is balanced if and only if $c^{\downarrow}[i,j] = c^{\uparrow}[i,j]$ and for each $k \in \inter[i,j]$, we have $c^\downarrow[k,j] \geq c^\uparrow[k,j]$.
A set $S$ of sites is \defn{balanced} if it is a disjoint union of balanced intervals.
For $i \in \ive{n}$, we say that $i$ is \defn{balanced} if $i$ belongs to some balanced interval, and \defn{unbalanced} otherwise.

The following facts are straightforward:
\begin{enumerate}
 \item For any balanced interval $\mcI$, we have $\lvert q_1 \cap \mcI \rvert = \lvert q_2 \cap \mcI \rvert$.
 \item The empty interval is balanced.
 \item The set $\ive{n}$ of all sites is not balanced, since $r_1 \neq r_2$.
 \item If $A$ and $B$ are two balanced sets, then the set $A \cap B$
       is balanced.\footnote{To prove this, first check it when $A$ and $B$
       are balanced intervals.}
 \item If $A$ and $B$ are balanced intervals, and if $A \cup B$ is an
       interval, then $A \cup B$ is balanced.\footnote{This is obvious
       when one of $A$ and $B$ contains the other. Otherwise, argue
       that $A \cap B$, $A \setminus B$ and $B \setminus A$ are balanced.}
 \item Thus, the union of all balanced intervals (\textit{i.e.}, the set of all balanced
       $i \in \ive{n}$) is also the disjoint union of all maximal balanced intervals
       (where ``maximal'' means ``maximal under inclusion'').
 \item For $r_1 < r_2$ and $j \in \ive{n}$ unbalanced, we have $j \notin q_1$ and $j \in q_2$.
 \item For $r_1 > r_2$ and $j \in \ive{n}$ unbalanced, we have $j \in q_1$ and $j \notin q_2$.
 \item There are exactly $\abs{r_1 - r_2}$ unbalanced sites $i$.
\end{enumerate}

We construct $C' = (q'_1, q'_2)$ by letting $q'_i \cap \mcI = q_i \cap \mcI$ for $i=1,2$ and each balanced interval $\mcI$ of $C$.
For unbalanced $j$, we have $j \in q'_i$ if and only if $j \in q_{3-i}$ for $i = 1,2$.
Note that $C$ and $C'$ have the same balanced intervals.
It is clear that $C'' = C$ and $\wt(C) = \wt(C')$.

(We could adapt this construction to the case $r_1 = r_2$, but we would need to be more precise about defining intervals, since the set of all sites can be written as $\inter[i, i-1]$ for any value of $i$. If done correctly, this results in every $i \in \ive{n}$ being balanced when $r_1 = r_2$, and thus the construction yields $C' = C$ as we defined.)

\begin{example}
Consider the configuration $C$ given in Figure~\ref{fig:balanced}.
The dual configuration $C'$ is given by sliding all of the circles not boxed from the upper level to the lower level.
In particular, we have $q_1' = q_1 \setminus \{1,5,6,8\}$ and $q_2' = q_2 \cup \{1,5,6,8\}$.
\end{example}

\begin{figure}[t]
\[
\begin{tikzpicture}[scale=0.75]
  \def\sc{0.85}   % Change this to adjust the x-scaling
  \def\ll{2}   % level 2
  \def\l{1}   % level 1
  \draw[fill=blue!30] (1.5*\sc,\l-.5) rectangle(4.5*\sc,\ll+.5);
  \draw[fill=blue!30] (6.5*\sc,\l-.5) rectangle(7.5*\sc,\ll+.5);
  \draw[fill=blue!30] (8.5*\sc,\l-.5) rectangle(20.5*\sc,\ll+.5);
  \foreach \i in {1,2,5,6,8,11,13,14,17,18,19} { \draw[fill=white] (\i*\sc,\ll) circle (0.3); }
  \foreach \i in {3,4,7,9,10,12,15,16,20} { \draw[fill=white] (\i*\sc-.3,\ll-.3) rectangle +(0.6,+0.6); }
  \foreach \i in {2,12,15,16,18,19,20} { \draw[fill=white] (\i*\sc,\l) circle (0.3); }
  \foreach \i in {1,3,4,5,6,7,8,9,10,11,13,14,17} { \draw[fill=white] (\i*\sc-.3,\l-.3) rectangle +(0.6,+0.6); }
\end{tikzpicture}
\]
\caption{We draw a $\bigcirc$ in position $i$ in row $j$ corresponding to $i \in q_j$ and a $\square$ if $i \notin q_j$.
The maximal balanced intervals are boxed.}
\label{fig:balanced-old6}
\end{figure}

Recall the notations introduced just before Lemma~\ref{le:dual}.

\begin{remark}
\label{rmk:balanced-dual-let-old6}
Let $C = (q_1, q_2)$ be an $(r_1, r_2)$-configuration.
Then, the interval $\mcI$ of $C$ is balanced if and only if the interval $\mcI^{\operatorname{refl}} = \set{ n + 1 - i \mid i \in \mcI }$ of the $(n-r_1, n-r_2)$-configuration $C^* = (q_1^*, q_2^*)$ is balanced.
In other words, the balanced intervals of $C$ are exactly the balanced intervals of $C^*$ but reflected through the middle of $\ive{n}$.
Thus, the dual configuration of $C^* = \tup{q_1^*, q_2^*}$ is obtained from the dual configuration $\tup{q_1', q_2'}$ of $C$ by
\begin{equation}
 (C^*)' = \bigl( (q_1')^*, (q_2')^* \bigr) .
 \label{eq.rmk:balanced-dual-let.dual-old6}
\end{equation}
In addition, applying Lemma~\ref{le:dual} twice, we obtain
\[
C(u) = q_2\bigl( q_1(u) \bigr) = q_2\bigl( (q^*_1(u^*))^* \bigr) = \bigl( q_2^*\bigl( q_1^*(u^*) \bigr) \bigr)^* = \bigl( C^*(u^*) \bigr)^* .
\]
In particular, for $u \in \set{1,2}^n$, we have $C(u)_i = 5 - C^*(u^*)_{n+1-i}$, where we treat $u$ and $C^*(u^*)$ as a word with $2$ and $4$ classes, respectively.
\end{remark}

Fix $k \geq 1$.
In the following, we simplify our terminology and say that an MLQ
is a $k$-tuple of queues (without any restriction on their sizes).
We want to define an action of $\SymGp{k}$ on MLQs.
%by letting each simple transposition $s_i$ act as the map
For each $i \in \ive{k-1}$, we define a map $\fraks_i \colon \set{\text{MLQs}} \to \set{\text{MLQs}}$ by
\[
\fraks_i(q_1, q_2, \dotsc, q_k) = (q_1, \dotsc, q_{i-1}, q'_i, q'_{i+1}, q_{i+2}, \dotsc, q_k),
\]
where $\tup{q'_i, q'_{i+1}}$ is the dual configuration of $\tup{q_i, q_{i+1}}$.
From the definition of a dual configuration, it is clear that $\fraks_i \fraks_i \qq = \qq$.
It is also clear from the definition that $\fraks_i \fraks_j \qq = \fraks_j \fraks_i \qq$ if $\abs{i - j} > 1$.
Thus, the following proposition shows that $\fraks_i$ defines an action of $\SymGp{k}$ on the set of all MLQs.

\begin{prop} \label{prop:braid-old6}
We have
\[
\fraks_i \fraks_{i+1} \fraks_i \qq
	   = \fraks_{i+1} \fraks_i \fraks_{i+1} \qq
\]
for any MLQ $\qq = \tup{q_1, \dotsc, q_k}$ and any $i \in \set{1, 2, \ldots, k-2}$.
\end{prop}

\begin{proof}
We shall deduce the claim from~\cite[Ch.~5, (5.6.3)]{Loth}.

Let $A$ be the $\tup{k+1}$-element set $\set{\circ, 1, 2, \ldots, k}$.
Let $A^*$ denote the set of all words on the alphabet $A$
(of any finite length).

We construct a $k \times n$-matrix $M_{\qq} \in A^{k \times n}$ from $\qq$ by setting the $\tup{i, j}$-th
entry to $i$ if $j \in q_i$ and $\circ$ otherwise.
We then construct a word $\word(\qq) \in A^*$ by reading $M_{\qq}$ from top-to-bottom, left-to-right (\textit{i.e.}, column by column).
For example, if $n = 5$, $k = 3$, then
\begin{align*}
\qq = \tup{\set{1, 3}, \set{2},
\set{2, 5}}
& \longleftrightarrow
 M_{\qq}
 =
 \begin{array}{ccccc}
  1 & \circ & 1 & \circ & \circ \\
  \circ & 2 & \circ  & \circ & \circ \\
  \circ & 3  & \circ & \circ & 3
 \end{array}
 \\ & \longrightarrow
 \word(\qq) = 1 \circ \circ \circ 2 3 1 \circ \circ \circ \circ \circ \circ \circ 3 .
\end{align*}
Clearly, an MLQ $\qq$ is uniquely determined by $\word(\qq)$ since $n$ is fixed.
In other words, the map $\word \colon \set{\text{MLQs}} \to A^*$
is injective.

Now, for each $i \in \set{1, 2, \ldots, k-1}$, we recall the operator
$\sigma_i \colon A^* \to A^*$ from~\cite[\S5.5]{Loth}.
This operator $\sigma_i$ acts on a word $p \in A^*$ % = \tup{p_1, p_2, \ldots, p_\ell}$
as follows:
\begin{enumerate}
 \item Treat all letters $i$ in $p$ as opening parentheses `$($',
       all letters $i+1$ as closing parentheses `$)$',
       and consider all other letters to be frozen.
       Now, match as many parentheses as possible
       according to the standard parenthesis-matching
       algorithm (\textit{i.e.}, every time you find an opening
       parenthesis to the left of a closing one, with only
       frozen letters between them, you match these two
       parentheses and declare them frozen).
       Notice that this algorithm is non-deterministic, but
       the outcome is independent of the steps chosen;
       the result is always a word whose non-frozen
       part (\textit{i.e.}, the word obtained by removing
       all frozen letters) is
       $\underbrace{))\cdots)}_{a\text{ parentheses}}
       \underbrace{((\cdots(}_{b\text{ parentheses}}$
       for some integers $a, b \geq 0$.
       We call this the \defn{reduced signature} of $p$.
 \item Now, replace this non-frozen part by
       $\underbrace{))\cdots)}_{b\text{ parentheses}}
       \underbrace{((\cdots(}_{a\text{ parentheses}}$
       while keeping all frozen letters in their places.
       The resulting word is $\sigma_i p$.
\end{enumerate}
From~\cite[Eq.~(5.6.3)]{Loth}, these operators
$\sigma_i$ satisfy
\begin{equation}
 \sigma_i \sigma_{i+1} \sigma_i
 = \sigma_{i+1} \sigma_i \sigma_{i+1}
 \label{pf.prop:braid.loth-eq-old6}
\end{equation}
for all $i \in \set{1, 2, \ldots, k-2}$.

Let $\zeta \colon \mcW_n \to \mcW_n$ be the cyclic shift map that sends each word $w_1 w_2 \cdots w_n$ to $w_2 w_3 \cdots w_n w_1$.
We also abuse the notation $\zeta$ for the map that sends each queue $q$ to the queue $\zeta q = \set{ i - 1 \mid i \in q }$ (recall that $0 = n$ as sites).
This map $\zeta$ shall act on MLQs entrywise (since an MLQ is a tuple of queues).
Clearly,
\begin{equation}
 \word(\zeta \qq) = \zeta^k \word(\qq)
 \label{pf.prop:braid.word-zeta-old6}
\end{equation}
for any MLQ $\qq = (q_1, \ldots, q_k)$.

Now, we claim that
\begin{equation}
 \word(\fraks_i \qq) = \sigma_i\bigl( \word(\qq) \bigr)
 \qquad \text{ for each MLQ } \qq \text{ and each } i .
 \label{pf.prop:braid.inter-old6}
\end{equation}
Note that it is sufficient to show that $\word(\fraks_1 \qq) = \sigma_1\bigl( \word(\qq) \bigr)$ for $\qq = (q_1, q_2)$
(because the definition of $\sigma_i$ only relies on the letters $i$ and $i+1$, while all other letters stay in their places and have no effect).

Thus, let $\qq = (q_1, q_2)$.
We want to show $\word(\fraks_1 \qq) = \sigma_1\bigl( \word(\qq) \bigr)$.
If $\abs{q_1} = \abs{q_2}$, then $\fraks_1 \qq = \qq$ by the definition of $\fraks_1$.
Moreover, we have $\sigma_1\bigl( \word(\qq) \bigr) = \word(\qq)$ in this case, since the word $\word(\qq)$ has as many letters $1$ as it has letters $2$, but the map $\sigma_1$ leaves such words unchanged.
Hence, the claim holds when $\abs{q_1} = \abs{q_2}$.
Thus, we assume that $\abs{q_1} \neq \abs{q_2}$.
Therefore, there exists at least one unbalanced site for the configuration $\qq = (q_1, q_2)$.

The operator $\fraks_1$ commutes with the cyclic shift map $\zeta$ on MLQs because $\zeta$ merely shifts the balanced intervals.
The operator $\sigma_1$ commutes with the cyclic shift map $\zeta$ on words in $A^*$ by~\cite[Prop.~5.6.1]{Loth}.
Hence, and because of~\eqref{pf.prop:braid.word-zeta-old6}, we can apply $\zeta$ to $\qq$ any number of times without loss of generality.
Thus, we assume that the site $1$ is unbalanced for the configuration $\qq = (q_1, q_2)$, since at least one unbalanced site $j$ exists and we can cyclically shift until it is $1$.
Therefore, no balanced interval has the form $\inter[i,j]$ with $i > j$.

We construct a sequence of parentheses as follows: For each $j = 1, 2, \dotsc, n$ (in that order), we write
\begin{itemize}
\item an opening parenthesis `$($' if $j \in q_1$ and $j \notin q_2$,
\item a closing parenthesis `$)$' if $j \notin q_1$ and $j \in q_2$,
\item a matched pair of parentheses `$()$' if $j \in q_1$ and $j \in q_2$,
\item nothing otherwise.
\end{itemize}
Note that this is exactly the sequence of parentheses constructed when applying $\sigma_1$ to $\word(\qq)$ (removing all $\circ$ letters).
Furthermore, every `$($' and `$)$' corresponds to a contribution to $c^{\uparrow}[1,n]$ and $c^{\downarrow}[1,n]$, respectively.
Additionally, every matched pair of parentheses from sites $j \leq j'$ under the standard matching algorithm corresponds to the endpoints of a balanced interval $\inter[j,j']$. (This is easily proven by induction on the time at which the parentheses got matched: At this time, all the parentheses inbetween have already been matched, thus forming balanced intervals, and the newly matched pair merely wraps them in a bigger balanced interval.)
Thus, if the algorithm would leave both a `$)$' and a `$($' unmatched, then the (cyclic) interval between the rightmost unmatched `$)$' and the leftmost unmatched `$($' would also be a balanced interval, which would contradict the fact that no balanced interval has the form $\inter[i,j]$ with $i > j$.
Consequently, the algorithm either leaves only `$)$' parentheses unmatched, or leaves only `$($' parentheses unmatched.
The precise outcome depends on which of $\abs{q_1}$ and $\abs{q_2}$ is larger.
Consequently, the sites of the unmatched parentheses are precisely the unbalanced sites.

Now, recall that $\fraks_1$ merely toggles the unbalanced sites between $q_1$ and $q_2$, whereas $\sigma_1$ switches the number of unmatched `$)$'s with the number of unmatched `$($'s (which, in the case of $\word(\qq)$, boils down to just turning each unmatched `$)$' into a `$($' or vice versa, because one of the numbers is $0$). Since the sites of the unmatched parentheses are precisely the unbalanced sites, this shows that the two maps agree -- that is, we have $\word(\fraks_1 \qq) = \sigma_1\bigl( \word(\qq) \bigr)$. This proves~\eqref{pf.prop:braid.inter-old6}.

The equality~\eqref{pf.prop:braid.inter-old6} can be rewritten as the
commutative diagram
\[
\xymatrix{
 \set{\text{MLQs}} \ar[r]^{\fraks_i} \ar[d]_{\word} & \set{\text{MLQs}} \ar[d]^{\word} \\
 A^* \ar[r]_{\sigma_i} & A^*
}
\]
for all $i \in \set{1, 2, \dotsc, k-1}$.
In view of the injectivity of the map $\word \colon \set{\text{MLQs}} \to A^*$,
this diagram allows us to translate~\eqref{pf.prop:braid.loth-eq-old6} into
$\fraks_i \fraks_{i+1} \fraks_i = \fraks_{i+1} \fraks_i \fraks_{i+1}$.
\end{proof}

\begin{remark}
Our letters $1, \ldots, k$ correspond to the letters
$a_k, \ldots, a_1$ in~\cite{Loth},
since the definition of $\sigma_i$ in~\cite{Loth} involves $a_i$ rather
than $i+1$ as closing parenthesis and $a_{i+1}$ rather than $i$ as opening one.
Also,~\cite{Loth} does not include the letter $\circ$ in the alphabet,
but this makes no difference to the proof, since all letters $\circ$ are always frozen.
% Note that we have arbitrarily chosen to break our cycle at $n$,
% however by~\cite[Prop.~5.6.1]{Loth}, the result does not depend on this choice.
\end{remark}

\begin{remark}
The operator $\sigma_i$ is essentially a combination of co-plactic operators.
Moreover, it corresponds to the Weyl group action on a tensor product of crystals~\cite{BS17}.
Note that the bracketing rule given above is precisely the usual signature rule (see, \textit{e.g.},~\cite[Sec.~2.4]{BS17} for a description) for computing tensor products.
This arises from considering the MLQ as a binary $m \times n$ matrix and the natural $(\mathfrak{sl}_m \oplus \mathfrak{sl}_n)$-action.
\end{remark}

Next, we define two queues corresponding to a word $w \in \mcW_n$ of type $\mm$.
Namely, for $k \in \set{p_i(\mm) \mid i \geq 0}$, let $[w]_k$ denote the set of the indices $i \in \ive{n}$
corresponding to the $k$ smallest letters $w_i$ of $w$.

The crucial tools in our proof of Theorem~\ref{thm:permutation} will be the following two facts.

\begin{lemma} \label{lem:SL.reconstruct-old6}
Let $w, w' \in \mcW_n$ be two words of the same type $\mm$.
Assume that $[w]_k = [w']_k$ for each $k \in \set{p_i(\mm) \mid i \geq 0}$.
Then, $w = w'$.
\end{lemma}

\begin{proof}
Fix some $i \geq 1$.
The sites containing the letter $i$ in $w$ are the elements of $[w]_{p_i(\mm)} \setminus [w]_{p_{i-1}(\mm)}$
(since $w$ has type $\mm$).
Likewise,
the sites containing the letter $i$ in $w'$ are the elements of $[w']_{p_i(\mm)} \setminus [w']_{p_{i-1}(\mm)}$.
Since our assumption ($[w]_k = [w']_k$) yields $[w]_{p_i(\mm)} \setminus [w]_{p_{i-1}(\mm)}
= [w']_{p_i(\mm)} \setminus [w']_{p_{i-1}(\mm)}$,
we conclude that these are the same sites.
Since this holds for all letters $i$, we have $w = w'$.
\end{proof}

\begin{prop} \label{prop:SL.dual-old6}
Let $u \in \mcW_n$ be a word of type $\mm$.
Let $k = p_{\alpha}(\mm)$ for some $\alpha$.
Let $q$ be a queue.
The dual configuration of $\tup{ [u]_k , q }$ has the form $\tup{q^{\dagger} , [q(u)]_k }$, where $q^{\dagger}$ is some queue.
\end{prop}

\begin{proof} %[Proof of Proposition~\ref{prop:SL.dual-old6}.]
Let $C' = (q'_1, q'_2)$ denote the dual configuration of $C = \tup{ [u]_k , q }$.
The notations $c^{\uparrow}$ and $c^{\downarrow}$ as well as the concept of balanced
intervals shall refer to $C$.

Choose a permutation $\tup{i_1, i_2, \ldots, i_n}$ of $\tup{1, 2, \ldots, n}$
such that $u_{i_1} \leq u_{i_2} \leq \cdots \leq u_{i_n}$.
Use this permutation to construct $q(u)$ (as in the definition of $q(u)$).
For each $p \in \ive{n}$, let $j_p$ be the site $j$ that is found in this construction when $i = i_p$.
Thus, $j_p \in q$ if $p \leq \abs{q}$, whereas $j_p \notin q$ if $p > \abs{q}$.
Also, $q(u)_{j_1} \leq q(u)_{j_2} \leq \cdots \leq q(u)_{j_n}$, so that
$[q(u)]_k = \set{j_1, j_2, \ldots, j_k}$ for each $k$ for which $[q(u)]_k$ is well-defined.

We want to prove that $(q'_1, q'_2)$ has the form $\tup{q^{\dagger} , [q(u)]_k }$.
In other words, we want to prove that $q'_2 = [q(u)]_k$.
If $k = \abs{q}$, then this is obvious (because in this case, the two queues in the configuration $C$ have the same size, so that its dual configuration $C'$ equals $C$, and thus $q'_2 = q$; but the assumption $k = \abs{q}$ also yields $[q(u)]_k = q$ because of the construction of $q(u)$, and therefore we obtain $q'_2 = q = [q(u)]_k$).

Suppose next that $k < \abs{q}$.
Thus, each site in $[u]_k$ is balanced.
But $[u]_k = \set{i_1, i_2, \ldots, i_k}$ by the definition of the permutation.

If $S$ is a set of sites, then the \defn{connected components} of $S$ are the maximal intervals contained in $S$.

We say that a closed cyclic interval $\inter[i,j]$ is \defn{top-heavy} if for each $\ell \in \inter[i,j]$, we have $c^\uparrow[i,\ell] \geq c^\downarrow[i,\ell]$.
Thus, a balanced cyclic interval is just a top-heavy cyclic interval $\inter[i,j]$ satisfying $c^{\uparrow}[i,j] = c^{\downarrow}[i,j]$.
% A set of sites will be called \defn{top-heavy} if each of its connected components is top-heavy.

For each $p \in \ive{k}$, we define an interval $I_p$ by $I_p = \inter[i_p, j_p]$.

We first observe that $I_p \cap q \subseteq \set{j_1, j_2, \ldots, j_p}$
for each $p \in \ive{k}$.
[\textit{Proof.} Let $p \in \ive{k}$.
Recall that the site $j_p$ is found
(in Phase~II of the algorithm computing $q(u)$) as the first site $j \in q$ weakly to the right of $i_p$ such that $q(u)_j$ is not set yet.
In other words, $j_p$ is the first site $j \in q$ weakly to the right of $i_p$ that is not one of $j_1, j_2, \ldots, j_{p-1}$
(since the sites $j \in q$ such that $q(u)_j$ has already been set are $j_1, j_2, \ldots, j_{p-1}$).
In other words, all sites in $I_p = \inter[i_p, j_p]$ that belong to $q$ must be among $j_1, j_2, \ldots, j_{p-1}, j_p$.
In other words, $I_p \cap q \subseteq \set{j_1, j_2, \ldots, j_p}$, qed.]

Let $U = I_1 \cup I_2 \cup \cdots \cup I_k$.
Then, $U \cap q \subseteq \set{j_1, j_2, \ldots, j_k}$ (since $I_p \cap q \subseteq \set{j_1, j_2, \ldots, j_p}$ for each $p \in \ive{k}$).
Combining this with $\set{j_1, j_2, \ldots, j_k} \subseteq U \cap q$ (since each $j_p$ satisfies $j_p \in I_p \subseteq U$ and $j_p \in \set{j_1, j_2, \ldots, j_{\abs{q}}} = q$),
we obtain $U \cap q = \set{j_1, j_2, \ldots, j_k}$.

If we had $U = \ive{n}$, then this would rewrite as $q = \set{j_1, j_2, \ldots, j_k}$, which would entail $\abs{q} = k$; this would contradict $k < \abs{q}$.
Hence, $U \neq \ive{n}$.

Each connected component of the set $U$ is top-heavy.
[\textit{Proof.} Let $\inter[a,b]$ be a connected component of $U$. Then, we must prove that $\inter[a,b]$ is top-heavy.
Indeed, since $U$ is the union of the connected intervals $I_1, I_2, \ldots, I_k$, its connected component $\inter[a,b]$ must have the form $\inter[a,b] = \bigcup_{p \in R} I_p$ for some nonempty subset $R$ of $\ive{k}$, and be disjoint from all the $I_p$ with $p \notin R$. Consider this $R$.
Thus, $\inter[a,b] \cap \set{i_1, i_2, \ldots, i_k} = \set{i_p \mid p \in R}$ (indeed, $\inter[a,b] = \bigcup_{p \in R} I_p$ shows that all of the $i_p$ with $p \in R$ must lie in $\inter[a,b] \cap \set{i_1, i_2, \ldots, i_k}$; on the other hand, none of the remaining elements of $\set{i_1, i_2, \ldots, i_k}$ can belong to $\inter[a,b]$, since $\inter[a,b]$ is disjoint from all the $I_p$ with $p \notin R$).
Likewise, $\inter[a,b] \cap \set{j_1, j_2, \ldots, j_k} = \set{j_p \mid p \in R}$.
From $[u]_k = \set{i_1, i_2, \ldots, i_k}$, we obtain
\[
\inter[a,b] \cap [u]_k = \inter[a,b] \cap \set{i_1, i_2, \ldots, i_k} = \set{i_p \mid p \in R} .
\]
From $\inter[a,b] \subseteq U$, we obtain
\[
\inter[a,b] \cap q = \inter[a,b] \cap \underbrace{U \cap q}_{= \set{j_1, j_2, \ldots, j_k}}
= \inter[a,b] \cap \set{j_1, j_2, \ldots, j_k} = \set{j_p \mid p \in R} .
\]
Thus, for each $\ell \in \inter[a,b]$, the number $c^\uparrow[a,\ell] = \abs{\inter[a,\ell] \cap [u]_k}$ counts all of the $i_p$ with $p \in R$ that fall into the interval $\inter[a,\ell]$, while the number $c^\downarrow[a,\ell] = \abs{\inter[a,\ell] \cap q}$ counts all of the $j_p$ with $p \in R$ that fall into this interval.
Hence, the former number is at least as large as the latter number (because if $p \in R$ is such that $j_p$ falls into $\inter[a,\ell]$, then $i_p$ must also fall into $\inter[a,\ell]$\footnote{In fact, let $p \in R$ be such that $j_p \in \inter[a,\ell]$. But $\inter[i_p,j_p] = I_p$ is a subinterval of $\inter[a,b]$, due to $\inter[a,b] = \bigcup_{p \in R} I_p$. The interval $\inter[a,\ell]$ is a ``prefix'' of $\inter[a,b]$; thus, every subinterval of $\inter[a,\ell]$ that ends inside $\inter[a,\ell]$ must also begin inside $\inter[a,\ell]$. Applying this to the subinterval $\inter[i_p, j_p]$, we conclude that $i_p \in \inter[a,\ell]$, as we wanted to show. Here, we have tacitly used the fact that $\inter[a,b] \neq \ive{n}$, which follows from $\inter[a,b] \subseteq U \neq \ive{n}$.}).
We have thus proven that for each $\ell \in \inter[i,j]$, we have $c^\uparrow[a,\ell] \geq c^\downarrow[a,\ell]$. In other words, the interval $\inter[a,b]$ is top-heavy, qed.]

Consider again the connected components of $U$. Each of them is top-heavy (as we have just shown), and thus contains at least as many elements of $[u]_k$ as it contains elements of $q$. Hence, the set $U$ altogether contains at least as many elements of $[u]_k$ as it contains elements of $q$, and this inequality becomes an equality only if each connected component of $U$ is balanced.
But this inequality does become an equality, because the set $U$ contains all $k$ elements of $[u]_k$ (since $[u]_k = \set{i_1, i_2, \ldots, i_k} \subseteq U$) and contains exactly $k$ elements of $q$ (since $U \cap q = \set{j_1, j_2, \ldots, j_k}$).
Thus, each connected component of $U$ is balanced. In other words, $U$ is balanced.
Hence, each element of $[q(u)]_k$ is balanced (since $[q(u)]_k = \set{j_1, j_2, \ldots, j_k} = U \cap q \subseteq U$).
Due to how the dual configuration $(q'_1, q'_2)$ was defined, we thus conclude that each element of $[q(u)]_k$ is in $q'_2$. In other words, $[q(u)]_k \subseteq q'_2$.
Combining this with $\abs{q'_2} = \abs{[u]_k} = k = \abs{[q(u)]_k}$, we obtain $q'_2 = [q(u)]_k$.
Thus, Proposition~\ref{prop:SL.dual-old6} is proven in the case when $k < \abs{q}$.

Now suppose $k > \abs{q}$.
% Note that $\ive{n} \setminus [u]_k$ are the indices of the largest $n-k$ letters of $u$ and $n - \abs{q} > n-k$.
Let $\ell$ be the number of classes in $u$, and consider the contragredient duals $q^*$ and $u^*$ as in Lemma~\ref{le:dual}.
Thus, $n-k < n - \abs{q} = \abs{q^*}$.
Hence, applying the $k < \abs{q}$ case (proven above)
to $n-k$, $u^*$ and $q^*$ instead of $k$, $u$ and $q$,
we see that
$\tup{ [u^*]_{n-k} , q^* }' = \tup{ q^\dagger , [q^*(u^*)]_{n-k} }$
for some queue $q^\dagger$.
Now,
\begin{align*}
 \bigl( (q'_1)^*, (q'_2)^* \bigr)
 &= \tup{ ([u]_k)^* , q^* }'
 = \tup{ [u^*]_{n-k} , q^* }' \\
 & = \tup{ q^\dagger , [q^*(u^*)]_{n-k} }
 = \tup{ q^\dagger , [(q(u))^*]_{n-k} } ,
\end{align*}
where
\begin{itemize}
 \item the first equality follows from~\eqref{eq.rmk:balanced-dual-let.dual-old6};
 \item the second equality is because $([u]_k)^* = [u^*]_{n-k}$; and
 \item the fourth equality follows from Lemma~\ref{le:dual}.
\end{itemize}
Therefore, $(q'_2)^* = [(q(u))^*]_{n-k} = ([q(u)]_k)^*$, so that
$q'_2 = [q(u)]_k$.
Hence, Proposition~\ref{prop:SL.dual-old6} is proven in the case when $k > \abs{q}$.
\end{proof}

\begin{proof}[Proof of Theorem~\ref{thm:permutation}.]
Recall that any permutation in $\SymGp{\ell-1}$ is a product of simple transpositions $s_1, s_2, \ldots, s_{\ell-2}$.
Hence, in order to prove Theorem~\ref{thm:permutation}, it suffices to show that $\swt{u}_{\sigma} = \swt{u}_{\sigma s_i}$ for each $\sigma \in \SymGp{\ell-1}$ and $i \in \ive{\ell-2}$.
Then, Theorem~\ref{thm:permutation} follows by induction on length, \textit{i.e.} the minimal number of simple transpositions needed to write $\sigma$.

In order to prove $\swt{u}_{\sigma} = \swt{u}_{\sigma s_i}$, we need to show, for a $\sigma$-twisted MLQ $\qq$ of type $\mm$ satisfying $u = \qq (1^n)$, that $\fraks_i \qq$ is a $\sigma s_i$-twisted MLQ of type $\mm$ satisfying $u = (\fraks_i \qq) (1^n)$ (since this will show that $\fraks_i$ bijects the former MLQs to the latter).
The only nontrivial part is showing $u = (\fraks_i\qq) (1^n)$.
More generally, we will show that $(\fraks_i\qq)(w) = \qq(w)$ for any word $w \in \mcW_n$.
The proof of this claim reduces to showing that for any configuration $C = \tup{q_1, q_2}$ and any word $w \in \mcW_n$ the dual configuration $\fraks_1 C = C' = (q'_1, q'_2)$ of $C$ satisfies $C'(w) = C(w)$.

Each word $w$ can be obtained from a standard word by a sequence of
merges (each of which sends a word $u$ to $\vee^{(k)} u$ for some
$k \in \set{ p_j(\mm) \mid j \geq 1 }$, where $\mm$ is the type of
$u$). Lemma~\ref{lemma:queue_merge_commute} shows that these merges
commute with the action of a queue (and thus of an MLQ).
Hence, it is sufficient to consider standard words $w$.
Thus, assume that $w$ is standard of type $\mm$.
It is straightforward to see
(using Equation~\eqref{eq:queue_type_change} and $\abs{q'_2} = \abs{q_1}$ and $\abs{q'_1} = \abs{q_2}$)
that the words
$C(w) = q_2\bigl( q_1(w) \bigr)$
and
$C'(w) = q'_2\bigl( q_1'(w) \bigr)$
have the same type.
%indeed, the action of a queue $q$ on a word $w$ modifies its
%type in a predictable way (namely, the number $\abs{q}$ is
%inserted into the type of $w$ at the place where it would
%keep the type weakly increasing).
Let $\nn$ be this type.
We shall now show that
$[C'(w)]_k = [C(w)]_k$
for all $k \in \set{ p_i(\nn) \mid i \geq 0}$.
According to Lemma~\ref{lem:SL.reconstruct-old6}, this will yield $C'(w) = C(w)$, and thus our proof will be complete.

Let $k \in \set{ p_i(\nn) \mid i \geq 0}$.
Thus, $k \in \set{0, 1, \ldots, n} = \set{ p_i(\mm) \mid i \geq 0 }$ (since $w$ is standard).
Hence, $[w]_k$ is well-defined.
Note that $q_i(w)$ and $q'_i(w)$ are also standard words.
Using Proposition~\ref{prop:SL.dual-old6} to compute dual
configurations, we can see how the MLQ
$\qq = \tup{[w]_k, q_1, q_2}$ transforms under the action of
$\fraks_1 \fraks_2 \fraks_1$: Namely, we have
\begin{align*}
\tup{[w]_k, q_1, q_2} & \overset{\fraks_1}{\longmapsto} \tup{\ast, [q_1(w)]_k, q_2}
\\ & \overset{\fraks_2}{\longmapsto} \tup{\ast, \ast, \bigl[ q_2\bigl( q_1(w) \bigr) \bigr]_k}
\\ & \overset{\fraks_1}{\longmapsto} \tup{\ast, \ast, \bigl[ q_2\bigl( q_1(w) \bigr) \bigr]_k},
\end{align*}
where $\ast$ denotes some queue.
Likewise, the action of $\fraks_2 \fraks_1 \fraks_2$ is given by
\begin{align*}
\tup{[w]_k, q_1, q_2} & \overset{\fraks_2}{\longmapsto} \tup{[w]_k, q'_1, q'_2}
\\ & \overset{\fraks_1}{\longmapsto} \tup{\ast, [q'_1(w)]_k , q'_2}
\\ & \overset{\fraks_2}{\longmapsto} \tup{\ast, \ast, \bigl[ q'_2\bigl( q'_1(w) \bigr) \bigr]_k}.
\end{align*}
(See Figure~\ref{fig:crossing_diagrams-old6} for the actions depicted using crossing diagrams.)
Yet, the two maps are equal by Proposition~\ref{prop:braid-old6}.
Thus, the resulting MLQs must be identical:
\[
\tup{\ast, \ast, \bigl[ q'_2\bigl(q'_1(w) \bigr) \bigr]_k}
=
\tup{\ast, \ast, \bigl[ q_2\bigl(q_1(w) \bigr) \bigr]_k}.
\]
Hence, we have
\[
\bigl[ q'_2\bigl( q'_1(w) \bigr) \bigr]_k = \bigl[ q_2\bigl( q_1(w) \bigr) \bigr]_k.
\]
In other words, $[C'(w)]_k = [C(w)]_k$.
\end{proof}

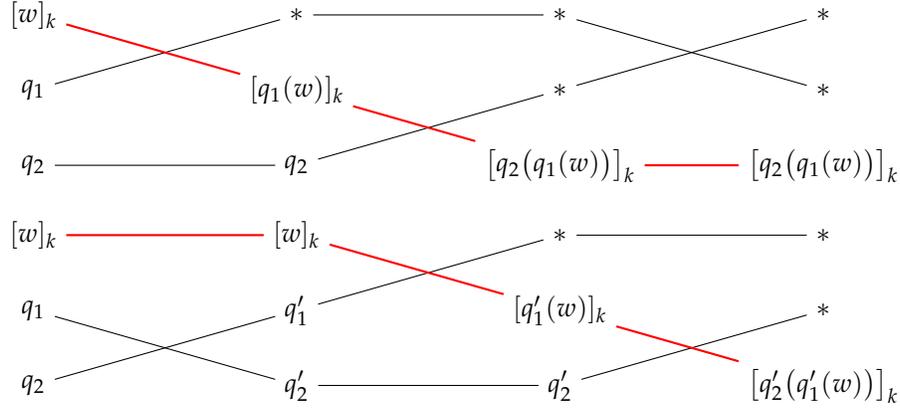
\begin{figure}
\begin{gather*}
\begin{tikzpicture}[xscale=3.5]
\node (i1) at (0,0) {$[w]_k$};
\node (i2) at (0,-1) {$q_1$};
\node (i3) at (0,-2) {$q_2$};
\node (s11) at (1,0) {$\ast$};
\node (s12) at (1,-1) {$[q_1(w)]_k$};
\node (s13) at (1,-2) {$q_2$};
\node (s21) at (2,0) {$\ast$};
\node (s22) at (2,-1) {$\ast$};
\node (s23) at (2,-2) {$\bigl[ q_2\bigl( q_1(w) \bigr) \bigr]_k$};
\node (t1) at (3,0) {$\ast$};
\node (t2) at (3,-1) {$\ast$};
\node (t3) at (3,-2) {$\bigl[ q_2\bigl( q_1(w) \bigr) \bigr]_k$};
\path[-] (i2) edge (s11) (s11) edge (s21) (s21) edge (t2);
\path[-] (i3) edge (s13) (s13) edge (s22) (s22) edge (t1);
\path[-,red] (i1) edge (s12) (s12) edge (s23) (s23) edge (t3);
\end{tikzpicture}
\\
\begin{tikzpicture}[xscale=3.5]
\node (i1) at (0,0) {$[w]_k$};
\node (i2) at (0,-1) {$q_1$};
\node (i3) at (0,-2) {$q_2$};
\node (s11) at (1,0) {$[w]_k$};
\node (s12) at (1,-1) {$q'_1$};
\node (s13) at (1,-2) {$q'_2$};
\node (s21) at (2,0) {$\ast$};
\node (s22) at (2,-1) {$[q'_1(w)]_k$};
\node (s23) at (2,-2) {$q'_2$};
\node (t1) at (3,0) {$\ast$};
\node (t2) at (3,-1) {$\ast$};
\node (t3) at (3,-2) {$\bigl[ q'_2\bigl( q'_1(w) \bigr) \bigr]_k$};
\path[-] (i2) edge (s13) (s13) edge (s23) (s23) edge (t2);
\path[-] (i3) edge (s12) (s12) edge (s21) (s21) edge (t1);
\path[-,red] (i1) edge (s11) (s11) edge (s22) (s22) edge (t3);
\end{tikzpicture}
\end{gather*}
\caption{Crossing diagrams representing the action of $\fraks_1 \fraks_2 \fraks_1$ (top) and $\fraks_2 \fraks_1 \fraks_2$ (bottom).}
\label{fig:crossing_diagrams-old6}
\end{figure}

\end{verlong}

\bibliographystyle{alpha}
\bibliography{queue}{}
\end{document}